\documentclass[11pt]{amsart}
\usepackage[active]{srcltx}
\usepackage[all]{xypic}
\usepackage[dvipsnames]{xcolor}
\usepackage{xifthen}
\usepackage{amsthm}
\usepackage{hyperref}
\usepackage{enumitem}
\setlist[enumerate]{label = $(\alph*)$, leftmargin = *}
\usepackage{amssymb}
\usepackage{amsmath}
\usepackage{stmaryrd}
\usepackage[mathcal]{euscript}
\usepackage[margin=1.3in]{geometry}
\newtheorem{theorem}{Theorem}[section]
\newtheorem{cor}[theorem]{Corollary}
\newtheorem{lemma}[theorem]{Lemma}
\newtheorem{prop}[theorem]{Proposition}
\newtheorem{notacao}[theorem]{Notation}
\theoremstyle{remark}
\newtheorem{rem}[theorem]{Remark}
\newtheorem{ex}[theorem]{Example}
\makeatletter
\def\namedlabel#1#2{\begingroup
  \def\@currentlabel{#2}%
  \label{#1}\endgroup
}

\usepackage{epic,eepic}
\newcount\brickleft
\newcount\brickright
\newcount\brickbase
\newcount\brickheight
\newcount\brickbackup
\newcount\brickforward
\newcount\brickbaseup
\def\brick(#1,#2,#3,#4,#5)(#6,#7){\nullfont
  \multiput(#1,#2)(#3,0)2{\line(0,1){\brickheight}}
  \multiput(#1,#2)(0,\brickheight)2{\line(1,0){#3}}
  \brickbase=#2
  \advance\brickbase by \brickbaseup
  \brickleft=#1
  \advance\brickleft by \brickforward
  \advance\brickleft by #4
  \put(\brickleft,\brickbase){$#6$}
  \brickright=#1
  \advance\brickright by #3
  \advance\brickright by \brickbackup
  \advance\brickright by -#5
  \put(\brickright,\brickbase){$#7$}
}

\def\brickstart(#1,#2,#3,#4,#5)(#6,#7){\nullfont
  \multiput(#1,#2)(#3,0)1{\line(0,1){\brickheight}}
  \multiput(#1,#2)(0,\brickheight)2{\line(1,0){#3}}
  \brickbase=#2
  \advance\brickbase by \brickbaseup
  \brickleft=#1
  \advance\brickleft by \brickforward
  \advance\brickleft by #4
  \put(\brickleft,\brickbase){$#6$}
  \brickright=#1
  \advance\brickright by #3
  \advance\brickright by \brickbackup
  \advance\brickright by -#5
}

\def\brickend(#1,#2,#3,#4,#5)(#6,#7){\nullfont
  \newcount\bricktemp
  \bricktemp#1\relax\advance\bricktemp#3\relax
  \put(\bricktemp,#2){\line(0,1){\brickheight}}
  \multiput(#1,#2)(0,\brickheight)2{\line(1,0){#3}}
  \brickbase=#2
  \advance\brickbase by \brickbaseup
  \brickleft=#1
  \advance\brickleft by \brickforward
  \advance\brickleft by #4
  \brickright=#1
  \advance\brickright by #3
  \advance\brickright by \brickbackup
  \advance\brickright by -#5
  \put(\brickright,\brickbase){$#7$}
}

\def\sbrick(#1,#2,#3){\nullfont
  \multiput(#1,#2)(#3,0)2{\line(0,1){\brickheight}}
  \multiput(#1,#2)(0,\brickheight)2{\line(1,0){#3}}
}

\def\dvert(#1,#2){\nullfont
  \multiput(#1,#2)(1,0)1{\multiput(0,0)(0,200){5}{\line(0,1){55}}}
}
\def\llvert(#1,#2){\nullfont
  \multiput(#1,#2)(1,0)1{\line(0,1){1000}}
}
\def\dbrick(#1,#2,#3){\nullfont
  \multiput(#1,#2)(0,1000)2{\multiput(0,0)(200,0){#3}{\line(1,0){50}}}
}
\newcommand{\DRH}{{\sf DRH}}
\newcommand{\h}{{\sf H}}
\newcommand{\W}{{\sf W}}
\newcommand{\V}{{\sf V}}
\newcommand{\G}{{\sf G}}
\newcommand{\R}{{\sf R}}
\newcommand{\LL}{{\sf L}}
\newcommand{\Sl}{{{\sf Sl}}}
\newcommand{\s}{{{\sf S}}}
\newcommand{\Ab}{{{\sf Ab}}}
\newcommand{\DO}{{{\sf DO}}}

\newcommand{\pseudo}[2]{\overline{\Omega}_{#1} {\sf{#2}}}
\newcommand{\pseudok}[2]{\Omega^{\kappa}_{#1} {\sf{#2}}}
\newcommand{\pseudosig}[2]{\Omega^{\sigma}_{#1} {\sf{#2}}}

\newcommand{\nn}{\mathbb{N}}
\newcommand{\lbf}[2][]{{\sf{lbf}}\ifthenelse{\isempty{#1}}{}{_{#1}}(#2)}
\newcommand{\cum}[1]{\vec{c}(#1)}
\newcommand{\card}[1]{\left\lvert#1\right\rvert}
\newcommand\malcev{\mathop{\raise0.5pt\hbox{\footnotesize$\bigcirc$\kern-8pt\raise0.5pt\hbox{\scriptsize$m$}\kern2pt}}}
\newcommand{\var}[1]{\llbracket#1\rrbracket}
\newcommand{\ev}{{\sf{ev}}}
\newcommand{\dom}{{\sf{Dom}}}
\newcommand{\im}{{\sf{Im}}}
\newcommand{\dirr}{{\sf{right}}}
\newcommand{\esq}{{\sf{left}}}
\newcommand{\Thetap}{{\sf{prod}} \circ \Theta}
\newcommand{\leng}[1]{\left\lceil#1\right\rceil}
\newcommand{\just}[2]{\stackrel{#2}{#1}}

\newcommand{\Req}{\mathrel{\mathcal R}}
\newcommand{\Deq}{\mathrel{\mathcal D}}
\newcommand{\Heq}{\mathrel{\mathcal H}}

\newcommand{\tl}{{\sf{l}}}
\newcommand{\tr}{{\sf{r}}}

\newcommand{\f}{\overline{f}}
\newcommand{\xx}{\overline{x}}
\newcommand{\y}{\overline{y}}
\newcommand{\z}{\overline{z}}

\newcommand{\vs}{\vec{s}}
\newcommand{\vt}{\vec{t}}

\newcommand{\iC}{\mathcal{C}}
\newcommand{\iE}{\mathcal{E}}
\newcommand{\iS}{\mathcal{S}}
\newcommand{\iB}{\mathcal{B}}
\newcommand{\iM}{\mathcal{M}}
\newcommand{\iX}{\mathcal{X}}
\begin{document}
\title{Complete $\kappa$-reducibility of pseudovarieties of the
  form~$\DRH$}
\author{Jorge Almeida and C\'elia Borlido}
\address{Centro de Matem\'atica e Departamento de Matem\'atica, Faculdade de Ci\^encias,
Universidade do Porto, Rua do Campo Alegre, 687, 4169-007 Porto,
Portugal}
\email{jalmeida@fc.up.pt,cborlido@fc.up.pt}
\maketitle
\begin{abstract}
  We denote by $\kappa$ the implicit signature that contains the
  multiplication and the $(\omega-1)$-power.
  It is proved that for any completely $\kappa$-reducible pseudovariety
  of groups $\h$, the pseudovariety $\DRH$ of all finite semigroups
  whose regular $\Req$-classes are groups in $\h$ is completely
  $\kappa$-reducible as well. The converse also holds. The tools
  used by Almeida, Costa, and Zeitoun for proving that the
  pseudovariety of all finite $\Req$-trivial monoids is
  completely $\kappa$-reducible are adapted for the
  general setting of a pseudovariety of the form~$\DRH$.
\end{abstract}
\section{Introduction}

The study of finite semigroups goes back to the beginning of the
1950's, 
having its roots in Theoretical Computer Science.
It was strongly motivated and developed by Eilenberg in collaboration
with Sch\"utzenberger and Tilson in the mid 1970's \cite{MR0530382,
  MR0530383}.
In particular,
Eilenberg \cite[Chapter
VII]{MR0530383} established a
correspondence between varieties of
rational languages and pseudovarieties of semigroups,
which has made possible to study combinatorial properties of the
former through the study of algebraic properties of
the latter.
As a result, it became of interest to
study the decidability of the \emph{membership problem} for
pseudovarieties. That means
to prove either that there exists an algorithm deciding whether a given finite
semigroup belongs to a certain pseudovariety, in which case the
pseudovariety is said to be \emph{decidable}; or to prove that such an
algorithm does not exist, being thus in the presence of an
\emph{undecidable} pseudovariety.
Considering some natural operators on pseudovarieties $\V$ and $\W$,
such as the join $\V \vee \W$,
the semidirect product $\V * \W$, the two-sided semidirect
product $\V **\, \W$, or the Mal'cev product $\V
\malcev \W$, it is also
relevant to decide the membership problem for the resulting
pseudovariety. It turns out that none of these operators preserves
decidability \cite{MR1150933, MR1723477}.
Aiming to guarantee the decidability of pseudovarieties obtained
through the application of $*$, from a stronger property
for the involved pseudovarieties, Almeida \cite{hyperdecidable}
introduced the notion of hyperdecidability. This property consists of a
generalization of inevitability for finite groups
introduced by Ash in \cite{MR1232670}. Since then, other notions like
tameness and reducibility and some other variants were also considered
\cite{unified_theory}.

On the other hand, Brzozowski and Fich~\cite{MR753707}
conjectured that $\Sl * \LL
= {\sf GLT}$
and established the inclusion $\Sl * \LL \subseteq {\sf
  GLT}$.
Motivated by this problem, Almeida and Weil~\cite{drh} considered the dual of
the pseudovariety $\LL$, the pseudovariety $\R$ of $\Req$-trivial
finite semigroups, and described the structure
of the free pro-$\R$ semigroup. Later on, it was proved
by Almeida and Silva~\cite{MR1819091} that the pseudovariety $\R$ is
$SC$-hyperdecidable for the canonical implicit signature $\kappa$, and
by Almeida, Costa and Zeitoun~\cite{JA} that $\R$ is 
completely $\kappa$-reducible. In this paper, we generalize
the results obtained in \cite{JA} for pseudovarieties of the form
$\DRH$, where $\h$ is a pseudovariety of groups and $\DRH$ is the
pseudovariety of semigroups whose regular $\Req$-classes are groups
lying in~$\h$. More precisely, we prove that~$\DRH$ is a completely
$\kappa$-reducible pseudovariety if and only if the pseudovariety of
groups~$\h$ is completely $\kappa$-reducible as well.
Of course, the latter condition holds for every locally finite pseudovariety~$\h$.
However, so far, the
unique known instance of  a completely $\kappa$-reducible non-locally
finite pseudovariety is $\Ab$, the pseudovariety of abelian
groups~\cite{MR2142087}.
Hence, the
pseudovariety ${\sf DRAb}$ is completely $\kappa$-reducible.
On the contrary, since neither the pseudovarieties $\G$ and $\G_p$
(respectively, of all finite
groups, and of all finite $p$-groups, for a prime~$p$) nor
proper non-locally finite subpseudovarieties of $\Ab$ are completely
$\kappa$-reducible~\cite{MR1485465,MR0179239,MR2364777},
we obtain a family of pseudovarieties of the form $\DRH$ that are not
completely $\kappa$-reducible.

In Section~\ref{sec:4} we introduce the basic concepts and set up
the notation used later. Section~\ref{section3} is devoted to general
facts on the structure of the free pro-$\DRH$ semigroup~$\pseudo
A{DRH}$ already known from~\cite{drh}. In particular, we describe
members of~$\pseudo A{DRH}$ by means of certain decorated reduced
$A$-labeled
ordinals.
Section~\ref{sec:500} contains a generalization of a periodicity
phenomenon over pseudovarieties of the form~$\DRH$
that was proved for~$\R$ in~\cite{JA}.
Some 
simplifications concerning the class of systems of equations that we must
consider in order to achieve complete $\kappa$-reducibility of~$\DRH$
are introduced in Section~\ref{section4}, while in
Sections~\ref{section5} and~\ref{section6} we redefine the tools used
in~\cite{JA}, adapting
them for the context of the pseudovarieties $\DRH$. Finally, in
Section~\ref{section7} we prove the main theorem, that is, we prove
that~$\DRH$
is
completely $\kappa$-reducible provided so is $\h$, whose converse
amounts to a simple observation.
\section{General definitions and notation}\label{sec:4}

For the basic concepts and results on (pro)finite semigroups
the reader is referred to \cite{livro,profinite}. The
required topological tools may be found in \cite{topologia}.

The symbols $\Req$, $\le_{\Req}$, $\Deq$, and $\Heq$ denote some of
Green's relations.
Given a semigroup $S$, we denote by $S^I$
the monoid whose underlying set is $S \uplus \{I\}$, where $S$ is a
subsemigroup and $I$
plays the role of a neutral element.
Given $n$ elements $s_1, \ldots, s_n$ of a semigroup $S$, we use the
notation $\prod_{i = 1}^n s_i$ for the product $s_1 s_2
\cdots s_n$.
Given a sequence $(s_n)_{n \ge 1}$ of a semigroup $S$ we call
\emph{infinite product} the sequence $\left(\prod_{i = 1}^n
  s_i\right)_{n \ge 1}$.

If nothing else is said, then we use $\V$ and $\W$ for denoting arbitrary
pseudovarieties of semigroups. Some pseudovarieties referred in this paper
are $\s$, the pseudovariety of all finite semigroups; $\Sl$, the
pseudovariety of all finite semilattices; $\G$, the pseudovariety
of all finite groups; $\G_p$, the pseudovariety of all $p$-groups (for
a prime number $p$); and $\Ab$, the pseudovariety of all finite
Abelian groups. We denote arbitrary subpseudovarieties of $\G$
by $\h$. Our main focus are the pseudovarieties of the form $\DRH$,
that is, the class of all finite semigroups whose regular
$\Req$-classes are groups lying in $\h$, and hence, are also
$\Heq$-classes.
If $\h$ is the trivial pseudovariety of groups ${\sf I} = \var{x=y}$, then $\DRH =
{\sf{DRI}}$ is the pseudovariety $\R$ of all finite $\Req$-trivial
semigroups.

We reserve the letter~$A$ to denote a finite alphabet. Then, $\pseudo
AV$ is the free $A$-generated pro-$\V$ semigroup.
If the pseudovariety $\V$ contains at least one non-trivial semigroup,
then the generating mapping $\iota: A \to \pseudo AV$
is injective. So, we often identify the elements of~$A$ with their
images under~$\iota$.
In the monoid $(\pseudo AV)^I$, we sometimes call $I$ the empty
(pseudo)word.
Also, if $B \subseteq A$, then the inclusion mapping induces an
injective continuous
homomorphism  $\pseudo
BV \to \pseudo AV$. Hence, we look at $\pseudo BV$ as a subsemigroup of
$\pseudo AV$.
On the other hand, if $\W$ is another pseudovariety contained in
$\V$, then $\rho_{\V,\W}$ represents the \emph{natural projection of
  $\pseudo AV$ onto 
  $\pseudo AW$}. We shall write~$\rho_\W$
when~$\V$ is clear from the context. In the case where $\W = \Sl$ we
denote $\rho_\Sl$ by $c$ and call it the \emph{content function}.

Given a pro-$\V$
semigroup $S$ and $u \in \pseudo AV$,
we denote by $u_S : S^A
\to S$ the interpretation in $S$ of the implicit operation induced by
$u$.
An \emph{implicit signature}, usually denoted~$\sigma$, is a set of
implicit operations on $\s$ containing the multiplication.
Of course, every implicit signature $\sigma$ endows $\pseudo AV$
with  a structure
of $\sigma$-algebra under the interpretation of  each one
of its symbols.
We denote
by $\pseudosig AV$ the $\sigma$-subalgebra of $\pseudo AV$
generated by $A$. The implicit signature $\kappa = \{\_\cdot\_,
\_^{\omega-1}\}$ is the \emph{canonical implicit signature}, where
$x^{\omega-1} = \lim_{n \ge 1} x^{n!-1}$.
Elements of $\pseudo AS$ are called \emph{pseudowords}, while elements
of $\pseudosig AS$  are
\emph{$\sigma$-words}

A formal equality $u = v$, with
$u, v \in \pseudo AS$ is called a \emph{pseudoidentity}.
Expressions like \emph{$\V$ satisfies $u = v$}, \emph{$u = v$ holds modulo
  $\V$}, and \emph{$u = v$ holds in $\V$}
mean that the interpretations of $u$ and $v$ coincide on every semigroup $S \in
\V$.
If that is the case, then we may write $u =_\V v$. We have $u =_\V v$
if and only if $\rho_\V(u) = \rho_\V(v)$.

Let $X$ be a finite set of \emph{variables} and $P$ a finite set of
\emph{parameters}, disjoint from $X$. A \emph{pseudoequation} is a
formal expression $u = v$ with $u, v \in \pseudo {X \cup P} S$. If
$u,v \in \pseudosig {X \cup P}S$, then $u = v$ is said to be a
\emph{$\sigma$-equation}, and if $u,v \in (X \cup P)^+$, then it is
called a
\emph{word equation}. A \emph{finite system of pseudoequations}
(respectively, \emph{$\sigma$-equations}, \emph{word equations}) is a
finite set
\begin{equation}
  \label{eq:1}
  \{u_i = v_i \colon i = 1, \ldots, n\},
\end{equation}
where each $u_i = v_i$ is a pseudoequation (respectively,
$\sigma$-equation, word equation). For each
variable $x \in X$, we consider a \emph{constraint} given by a
pair $(\varphi,\nu)$, where $\varphi:\pseudo AS \to S$ is a continuous
homomorphism into a finite semigroup $S$, and $\nu:X \to S$ is a
function.
The \emph{evaluation of the
  parameters} in $P$ is given by a map $\ev: P \to \pseudo AS$. A
\emph{solution modulo~$\V$} of the system~\eqref{eq:1}
\emph{satisfying the given constraints and subject to the evaluation
  of the parameters} is a continuous homomorphism $\delta:
\pseudo {X \cup P} S \to \pseudo AS$ such that the following conditions are
satisfied:
\begin{enumerate} [label = $(S.\arabic*)$]
\item \label{s1} $\delta(u_i) =_\V
  \delta(v_i)$, for $i = 1,\ldots, n$;
\item \label{s2} $\varphi(\delta(x)) = \nu(x)$, for $x \in X$;
\item \label{s3} $\delta(p) = \ev(p)$, for $p \in
  P$.
\end{enumerate}
Without loss of generality, we assume from now on that the semigroup~$S$ has a
content function (see \cite[Proposition~2.1]{MR1834943}).
  
If $\delta(X \cup P) \subseteq \pseudosig AS$, then we say that
$\delta$ is a solution modulo $\V$ of \eqref{eq:1} \emph{in
$\sigma$-words}. In particular, the existence of a solution in
$\sigma$-words implies, by \ref{s3}, that $\ev$ evaluates the parameters
in $\sigma$-words as well. Let $\iC$ be a class of finite systems of
$\sigma$-equations. We say that $\V$ is
\emph{$\sigma$-reducible} with respect to $\iC$ if any system of $\iC$
which has a solution modulo~$\V$ also has a solution modulo $\V$ in
$\sigma$-words. The pseudovariety $\V$ is said to be \emph{completely
  $\sigma$-reducible} if it is $\sigma$-reducible with respect to the
class of all finite systems of $\sigma$-equations.

\section{Structural aspects of the free pro-$\DRH$ semigroup}\label{section3}
\subsection{Preliminaries}

Before describing how to represent pseudowords over $\DRH$
conveniently, we need to introduce a few concepts.

Suppose that $\Sl \subseteq \V$ and let $u \in \pseudo AV$.
A \emph{left basic factorization} of $u$ is a factorization of the form $u =
u_\ell a u_r$, where $u_\ell, u_r \in (\pseudo AV)^I$ and $c(u) =
c(u_\ell) \uplus \{a\}$.
For certain pseudovarieties such a factorization always exists and is unique.
\begin{prop}
  [{\cite{drh,palavra}}]\label{p:1}
  Let $\V \in \{\DRH, \s\}$.
  Then, every element $u \in \pseudo AV$ admits a unique factorization
  of the form $u = u_\ell a u_r$ such that $a \notin c(u_\ell)$
  and $c(u_\ell a) = c(u)$.
\end{prop}

Applying inductively Proposition~\ref{p:1} to the leftmost factor of
the left basic factorization of a pseudoword over $\V \in \{\DRH, \s\}$, we obtain
the following result.
\begin{cor}
  \label{c:5} Let $\V \in \{\DRH, \s\}$ and $u$ be a pseudoword over~$\V$. Then, there exists
  a unique factorization $u = a_1u_1 a_2u_2\cdots  a_nu_n$ such that
  $a_i \notin c(a_1u_1 \cdots a_{i-1}u_{i-1})$, for  every $i = 2, \ldots,
  n$, and $c(u) = \{a_1, \ldots, a_n\}$.
\end{cor}

We refer to the factorization described in Corollary~\ref{c:5} as the \emph{first-occurrences
  factorization} of $u$.

For a pseudoword $u$ over $\V \in \{\DRH, \s\}$, we may also iterate the left basic factorization
of $u$ to the right as follows. We set $u'_0 = u$ and, for each $k \ge 1$,
whenever $u'_{k-1} \neq I$, we let $u'_{k-1} = u_ka_ku_k'$ be the left basic factorization of $u'_{k-1}$. Then, for every such~$k$, the equality $u = u_1a_1 \cdot
u_2a_2\cdots u_ka_k\cdot u_k'$ holds.
Moreover, the content of each
factor $u_ka_k$ decreases as $k$ increases. Since the alphabet $A$ is
finite, the
sequence of contents $(c(u_ka_k))_{k \ge 1}$ is either finite or it stabilizes. The
\emph{cumulative content} of $u$ is the empty set if the sequence is
finite, and is the set $c(u_ma_m)$ if $c(u_ma_m) = c(u_ka_k)$ for every
$k \ge m$. We denote the cumulative content of a pseudoword $u$ by
$\cum u$.
If~$\cum u \neq \emptyset$ and $m$ is the least integer such that
$\cum u = c(u_{m+1}a_{m+1})$, then we say that~$u_m'$ is the \emph{regular
  part of~$u$}.
It may be proved that an element $u$ of~$\pseudo A{DRH}$ is regular if
and only if its content coincides with its cumulative
content~\cite[Corollary~6.1.5]{drh}, that is, if $u$ is its own regular part.
If $\cum u = \emptyset$, then we set $\leng u = k$
if $u_k' =
I$. Otherwise, we set $\leng u = \infty$. We also write $\lbf[\infty]
u$
for the sequence $(u_1a_1, \ldots, u_{\leng u}a_{\leng u},
I,I,\ldots)$ if $\cum u = \emptyset$, and for the sequence $(u_ka_k)_{k
\ge 1}$ otherwise. We denote the $k$-th element of $\lbf[\infty]u$ by
$\lbf[k] u$.
\begin{rem}
  \label{sec:9}
  It is not hard to check that if $\V \in \{\DRH, \s\}$ satisfies the pseudoidentity $uu_0
  = u$, then $\lbf[\infty]{uu_0} = \lbf[\infty]u$ and
  $c(u_0)\subseteq\cum u$.
  Conversely, if $c(u_0) \subseteq \cum u$, then the
  equality $\lbf[\infty] u = \lbf[\infty]{uu_0}$ holds modulo $\V$.
\end{rem}

Suppose that the iteration of the left basic factorization of $u \in
\pseudo A{DRH}$ to the right runs forever. Since $\pseudo A{DRH}$ is a compact
monoid, the infinite product $(\lbf[1] u \cdots \lbf[k] u)_{k \ge 1}$ has,
at least, one accumulation point. Plus, any two
  accumulation points are 
$\Req$-equivalent (cf.\ \cite[Lemma 2.1.1]{drh}). If, in addition,
$u$ is
regular, then the $\Req$-class containing the accumulation points of
the mentioned sequence is regular
\cite[Proposition~2.1.4]{drh} and hence, it is a group.
In that case, we may define the \emph{idempotent designated} by the
infinite product  $(\lbf[1] u \cdots \lbf[k] u)_{k \ge 1}$ to be the
identity of the group to where its accumulation points belong.
It further happens that each regular $\Req$-class of $\pseudo A{DRH}$ is
homeomorphic to a free pro-$\h$ semigroup. This claim consists of a
particular case of the next proposition, which is behind  the results on the representation
of elements of~$\pseudo A{DRH}$ presented in~\cite{drh}, some of which we
state later.
We use $\DO$ and $\overline \h$ to denote the pseudovarieties
consisting, respectively, of all finite
semigroups whose regular $\Deq$-classes are orthodox semigroups, and
of all finite semigroups whose subgroups belong to $\h$.
\begin{prop}
  [{\cite[Proposition 5.1.2]{drh}}]
  \label{p:3}
  Let $\V$ be a pseudovariety such that the inclusions $\h\subseteq \V
  \subseteq \DO \cap \overline \h$ hold. If $e$ is
  an idempotent of $\pseudo AV$ and $H_e$ is its $\Heq$-class, then
  letting $\psi_e(a) = eae$ for each $a \in c(e)$ defines a unique
  homeomorphism $\psi_e:\pseudo{c(e)}H \to H_e$ whose inverse is the
  restriction of $\rho_{\h}$ to $H_e$.
\end{prop}

The following is an important consequence of Proposition~\ref{p:3} which we use later on.
\begin{cor}
  \label{c:2}
  Let $u$ be a pseudoword and $v,w \in (\pseudo AS)^I$ be such that
  $c(v)\cup c(w) \subseteq \cum u$ and $v =_\h w$. Then, the
  pseudovariety $\DRH$ satisfies $uv = uw$.
\end{cor}
We now have all the necessary ingredients to describe the elements of
$\pseudo A{DRH}$ by means of the so-called ``decorated reduced $A$-labeled
ordinals'',
which we do along the next subsection. The construction is based
on~\cite{drh}.
\subsection{Decorated reduced $A$-labeled ordinals}

A \emph{decorated reduced $A$-labeled ordinal} is a triple $(\alpha, \ell,
g)$ where
\begin{itemize}
\item $\alpha$ is an ordinal.
\item $\ell: \alpha \to A$ is a function.  For a limit ordinal $\beta
  \le \alpha$, we let the \emph{cumulative content of $\beta$ with
    respect to $\ell$} be given by
  $$\cum {\beta, \ell} = \{a \in A \colon \exists (\beta_n)_{n \ge 1} \mid
  \cup_{n \ge 1} \beta_n = \beta, \: \beta_n < \beta \text{ and }
  \ell(\beta_n) = a\}.$$
  Later, in Remark \ref{r:1}, we observe that the relationship between
  the cumulative content of
  an ordinal and the cumulative content of a pseudoword makes this
  terminology adequate.
  We further require for $\ell$ the following property:
  \begin{quote}
    for every limit ordinal $\beta < \alpha$, the letter $\ell(\beta)$
    does not belong to the set $\cum {\beta, \ell}$.
  \end{quote}
\item $g: \{\beta \le \alpha\colon \beta \text{ is a limit ordinal}\}
  \to \pseudo AH$ is a function such that $g(\beta) \in \pseudo {\cum{
    \beta, \ell}}H$.
\end{itemize}
We denote the set of all decorated reduced $A$-labeled ordinals by
${\sf rLO}_\h(A)$.

To each pseudoword~$u$ over~$\DRH$, we assign an element of ${\sf rLO}_\h(A)$ as follows.
Let us say that the product $ua$ is \emph{end-marked} if $a \notin
\cum u$.
It is known that the set of all end-marked pseudowords over a finite alphabet
constitutes a well-founded forest under the partial order
$\le_{\Req}$ \cite[Proposition 4.8]{JA}.
Then, $\alpha_u$ is the unique ordinal such that there exists an
isomorphism (also unique)
\begin{equation*}
  \theta_u: \alpha_u \to \{\text{end-marked prefixes of
  }u\}
\end{equation*}
such that $\theta_u(\beta) >_{\Req} \theta_u(\gamma)$ whenever $\beta <
\gamma$.
We let $\ell_u: \alpha_u \to A$ be the function sending
each ordinal $\beta \le \alpha$ to the letter $a$ if $\theta_u(\beta) =
va$.
\begin{rem}\label{r:1}
  We point out that, for every limit ordinal  $\beta \le \alpha$ such
  that $\theta_u(\beta) = va$, we have $\cum{v} = \cum{\beta,
    \ell_u}$.
\end{rem}
It remains to define~$g_u$.
Let~$\beta \le \alpha_u$ be a limit ordinal. By definition
of~$\theta_u$, if $\theta_u(\beta) = va$, then the regular part of~$v$
is nonempty. Then, we set $g_u(\beta)$ to be the projection
onto~$\pseudo AH$
of the regular part of~$v$.
Observe that, by Remark \ref{r:1}, $g_u(\beta)$ defined in that way
belongs to $\pseudo{\cum{\beta, \ell_u}} H$. Hence,
$(\alpha_u, \ell_u, g_u)$ is indeed a decorated reduced $A$-labeled ordinal.
We call~$F$ the mapping thus defined:
\begin{align*}
  F: \pseudo A{DRH} & \to {\sf rLO}_\h(A)
  \\ u &\mapsto (\alpha_u, \ell_u, g_u).
\end{align*}
\noindent It turns out that~$F$ is a bijection~\cite[Theorem~6.1.1]{drh}.
In fact, it is possible to define an algebraic structure on~${\sf
  rLO}_\h(A)$ that turns~$F$ into an isomorphism. We do not include
such construction since we make no explicit use of it.

Let $u \in \pseudo AS$.
Sometimes we abuse notation and write~$\alpha_u$ to refer to~$\alpha_{\rho_\DRH(u)}$.

\begin{notacao}\label{sec:19}
  Let $u \in \pseudo AS$ and take ordinals $\beta \le \gamma \le
  \alpha_u$. Let $\theta_u(\beta) = v a$ and $\theta_u(\gamma) =
  wb$. If $\beta < \gamma$, then we denote by $u[\beta,\gamma{[}$ the
  product $az$, where~$z$ is the unique pseudoword such that $w =
  vaz$.
  We set $u[\beta, \beta{[} = I$.
\end{notacao}
If $u$ is a $\kappa$-word, then the factors of $u$ of the form
$u[\beta,\gamma{[}$ are $\kappa$-words as well. This fact arises as a
consequence of the following lemma when we iterate it inductively.
\begin{lemma}
  [{\cite[Lemma 2.2]{palavra}}]\label{l:6}
  Let $u \in \pseudok AS$ and let $(u_\ell,a,u_r)$ be its left basic
  factorization. Then, $u_\ell$ and $u_r$ are $\kappa$-words.
\end{lemma}

\subsection{Further properties of pseudowords over~$\DRH$}

We proceed with the statement of some structural results to handle
pseudowords modulo~$\DRH$.
Although we could not find the exact statement that fits our purpose,
they seem to be already used in the
literature.
For that reason, we do not include any proof. They may be found in~\cite{phd}.

We first characterize $\Req$-classes of $\pseudo A{DRH}$ by means
of iteration of left basic factorizations to the right.

\begin{lemma}\label{sec:13}
  Let $u, v$ be pseudowords over $\DRH$. Then, $u$ and $v$ lie in the
  same $\Req$-class if and only if $\lbf[\infty] u = \lbf[\infty] v$.
\end{lemma}
As a consequence, we have the following:
\begin{cor}
  \label{c:3}
  Let $u,v \in \pseudo A{DRH}$. Then, the relation $u \Req v$ holds if and only if
  $\alpha_u=\alpha_v$, $\ell_u = \ell_v$ and $g_u|_{\{\beta < \alpha_u
    \colon \beta \text{ is a limit
      ordinal}\}} = g_v|_{\{\beta < \alpha_v \colon \beta \text{ is a
      limit ordinal}\}}$.
\end{cor}
We also have a kind of left cancellative law over $\DRH$.

\begin{cor}
  \label{c:6}
  Let $u$ and $v$ be pseudowords over $\DRH$ that are
  $\Req$-equivalent. Suppose that they admit factorizations $u = u_1au_2$ and $v
  = v_1bv_2$ such that $u_1a$ and $v_1b$ are end-marked.
  If $\alpha_{u_1} = \alpha_{v_1}$, then $a
  = b$, $u_1 = v_1$, and $u_2 \Req v_2$. If, in
  addition, the equality $u = v$ holds, then also $u_2 = v_2$.
\end{cor}

The following result is just a rewriting of the previous corollary that we state for
later reference.

\begin{cor}
  \label{c:7}
   Let $u,v$ be pseudowords that are $\Req$-equivalent modulo
   $\DRH$. Take ordinals
  $\beta < \gamma < \alpha_u = \alpha_v$. Then, the pseudovariety
  $\DRH$ also satisfies $u[\beta, \gamma{[} = v[\beta, \gamma{[}$ and
  $u[\gamma, \alpha_u{[} \Req v[\gamma, \alpha_v{[}$. Moreover, if $u
  =_\DRH v$, then $u[\gamma, \alpha_u{[} =_\DRH v[\gamma,
  \alpha_v{[}$.
\end{cor}

The next lemma can be thought as the key ingredient when proving
our main result. It becomes trivial when $\DRH = \R$.

\begin{lemma}
  \label{sec:16}
  Let $u,v \in \pseudo A{DRH}$ and $u_0, v_0 \in (\pseudo A{DRH})^I$ be such
  that $c(u_0) \subseteq \cum u$ and $c(v_0) \subseteq \cum
  v$. Then, the equality $uu_0 = vv_0$ holds if and
  only if $u \Req v$ and if, in addition, the pseudovariety
  $\h$ satisfies $uu_0 = vv_0$. In particular, by taking $u_0 = I =
  v_0$, we get that $ u = v$ if and only if $u
  \Req v$ and $u =_\h v$.
\end{lemma}
\section{Periodicity modulo $\DRH$}\label{sec:500}
Now, we state and prove two results concerning a
certain periodicity of members of~$\pseudo A{DRH}$. We first
need a few auxiliary lemmas.

\begin{lemma}
  [cf.\ {\cite[Lemma 5.1]{JA}}]
  \label{sec:40}
  Let $u, v$ be pseudowords over $\DRH$ such that $uv^\omega \Req
  v^\omega$.
  If $c(u) \subsetneqq c(v)$, then equality $uv = v$ holds.
\end{lemma}
\begin{proof}
  Let $a$ be a letter in $c(v) \setminus c(u)$. By
  Corollary~\ref{c:5}, we may factorize $v =
  v_1av_2$ with $a \notin c(v_1)$.
  Then, the equality $uv^\omega = v^\omega$ may be rewritten as $uv_1av_2v^{\omega-1} = v_1av_2v^{\omega-1}$.
  Since $a
  \notin c(uv_1)$, again
  Corollary~\ref{c:5} implies $uv_1 = v_1$,
  resulting in turn that~$uv = v$.
\end{proof}

We also recall a lemma related with the pseudovariety $\R$ that may be
used to prove
a weaker similar result for $\DRH$.
\begin{lemma}
  [{\cite[Lemma 5.2]{JA}}]
  \label{sec:15}
  If $u,v \in \pseudo AR$ are such that $vu^2 = u^2$, then $vu = u$.
\end{lemma}
We say that the product  $uv$ of two pseudowords is \emph{reduced} if
$v$ is not the empty word and its first letter does not belong to the
cumulative content of $u$.
\begin{cor}
  \label{c:4}
  If $u,v \in \pseudo A{DRH}$ are such that $vu^2 = u^2$ and the product~$u \cdot u$ is reduced,
  then the equality $vu = u$ holds.
\end{cor}
\begin{proof}
  Since $\R \subseteq \DRH$, the pseudovariety $\R$ satisfies $vu^2 =
  u^2$ and Lemma~\ref{sec:15} yields that it also satisfies $vu =
  u$. Therefore, from Corollary~\ref{c:3} we conclude that $\alpha_{vu} =
  \alpha_u$. As the product~$u \cdot u$ is reduced, it follows
  that $u^2[0, \alpha_u{[}= u$.
  On the other hand, Corollary~\ref{c:7} yields the identity  $vu^2[0,
  \alpha_{vu}{[} = vu^2[0, \alpha_{u}{[} = u^2[0,
  \alpha_{u}{[}$.
  Moreover, either $(vu)\cdot u$ is a reduced product and $vu^2[0,
  \alpha_{vu}{[} = vu$, or we may write $u = u_1\cdot u_2$, with $u_2
  = I$ or
  $u_1\cdot u_2$ a reduced product and $c(u_1)\subseteq
  \cum{uv}$, and then $vu^2[0, \alpha_{vu}{[} =  vuu_1$.
  In any case, $vu$ and $u$ are $\Req$-equivalent. Also, the inclusion~$\h \subseteq \DRH$ implies that
  $vu^2 = u^2$ modulo $\h$ and so, $vu = u$ modulo $\h$. Finally, it follows
  from Lemma~\ref{sec:16} that~$vu = u$.
\end{proof}
Now, we are ready to prove the announced results on the
periodicity in $\pseudo A{DRH}$.
\begin{lemma}
  [cf.\ {\cite[Lemma 5.4]{JA}}]
  \label{sec:12}
  Let $x$ and $y$ be pseudowords such that $x^\omega = y^\omega$
  modulo~$\DRH$. If the products $x \cdot x$ and $y \cdot y$ are
  reduced, then there are pseudowords $u \in \pseudo AS$ and~$v, w \in
  (\pseudo AS)^I$, and positive integers $k,\ell$ such that the following
  pseudoidentities hold in~$\DRH$
  \begin{align*}
    x &= u^k v,
    \\ y  &= u^\ell w,
    \\u  & = vu = wu,
  \end{align*}
  and all the products $u\cdot u$, $u \cdot v$, $u\cdot w$,
  $v\cdot u$, and $w\cdot u$ are reduced.
\end{lemma}
\begin{proof}
   We argue by transfinite induction on $\alpha = \max\{\alpha_x,
  \alpha_y\}$.
  
  If $\alpha_x = \alpha_y$, since the products $x\cdot x$ and $y\cdot y$ are
  reduced, we then have $x = y$ in $\DRH$, by Corollary~\ref{c:7}.
  So, we may choose $u = x$, $v = w = I$, and $k = \ell = 1$.

  From now on, we assume that the pseudovariety $\DRH$ does not
  satisfy $x = y$.
  Suppose, without loss of generality, that $\alpha_x < \alpha_y =
  \alpha$.
  Again, by Corollary~\ref{c:7}, $\DRH$ satisfies
  $$y = y^\omega[0, \alpha_y{[} = x^\omega[0, \alpha_y{[} =
  x^{\omega}[0, \alpha_x{[} \:x^\omega[\alpha_x, \alpha_y{[} =
  xx^\omega[\alpha_x, \alpha_y{[}$$
  and so, $x$ is a prefix of $y$ modulo $\DRH$.
  Thus, the set
  \begin{align*}
    P = \{ m \ge 1 \colon  \exists (y_1, \ldots, y_m \in
    \pseudo AS)\:
     y \le_{\Req} y_1 \cdots y_m \text{ and }  y_i =_\DRH x, \text{ for }  i = 1, \cdots m \}
  \end{align*}
  is nonempty. If it were unbounded then, since $x\cdot x$ is a
  reduced product and by definition of cumulative content, every
  letter of $c(x) = c(y_i)$ would be in the cumulative
  content of~$y$, so that $\vec{c}(y) = c(x) = c(y)$, a contradiction
  with the hypothesis that $y\cdot y$ is a reduced product. Take $m =
  \max(P)$ and let $y = y_1 \cdots y_m y'$, with $y_i =_\DRH x$,
  for $i = 1, \ldots, m$.
  Since $x^{\omega}  =_\DRH  y^{\omega}$, we deduce that
  $\DRH$ satisfies
  \begin{align*}
    x^\omega = y^\omega = y_1 \cdots y_m y' y^{\omega-1} =
    x^my'y^{\omega-1}
  \end{align*}
  which in turn, since the involved products are reduced, implies that
  $\DRH$ also satisfies
  \begin{align*}
    x^{\omega - m} = y' y^{\omega-1}.
  \end{align*}
  In particular, as $y^\omega = x^\omega$ in $\DRH$ (and
  so, $c(x) = c(y)$), we may conclude that $\DRH$ satisfies
  \begin{equation}
    x^\omega = y' y^{\omega-1} x^m = y' x^\omega y^{\omega-1}x^m
    \Req y'x^\omega .\label{eq:23}
  \end{equation}
  We now distinguish two cases.
  \begin{itemize}
  \item If $c(y') \subsetneqq c(x)$ then, by Lemma~\ref{sec:40}, the
    pseudovariety $\DRH$ satisfies $x = y'x$, so that we may
    choose $u = x$, $v = I$, $k = 1$, $w = y'$, and $\ell = m$.
  \item If $c(y') = c(x)$ then, successively multiplying by $y'$ on the left
    the leftmost and rightmost sides of \eqref{eq:23}, we get that the
    relation $x^\omega \Req y'^\omega x^\omega = y'^\omega$ holds in $\DRH$. As
    $x^\omega$ and $y'^\omega$ are both the identity in the same
    regular $\Req$-class, hence in the same group, the mentioned relation is actually an
    equality: $x^\omega =_\DRH y'^\omega$.
    Furthermore, the product $y' \cdot y'$ is reduced because so is
    $y\cdot y$.
    Indeed, $\vec{c}(y') = \vec{c}(y)$, the first letters of $y'$
    and $x$ coincide and, in turn, the first letter of $x$ is the
    first letter of $y$.
    Consequently, $y'$ and $x$ verify the conditions of applicability of
    the lemma and have associated a smaller induction parameter. In
    fact, maximality of $m$ guarantees that $\alpha_{y'} \le \alpha_x
    < \alpha_y = \alpha$. By induction hypothesis, there exist $u \in
    \pseudo AS$, $v,w \in (\pseudo AS)^I$, and $k,\ell > 0$ such that the
    identities
    \begin{equation}
      \begin{aligned}
         x &= u^kv,
        \\ y'  &= u^\ell w,
        \\ u & = vu = wu
      \end{aligned}\label{eq:44}
    \end{equation}
    are valid in $\DRH$, and where all products, including $u \cdot
    u$ are reduced. The computation
    $$ y = x^my' = (u^kv)^mu^\ell w =  u^{km+\ell}w$$
    modulo $\DRH$ justifies that, except for the value of $\ell$, which
    now is $km + \ell$, the choice in~\eqref{eq:44} also fits the original pair
    $x,y$.
  \end{itemize}
\end{proof}
The proof of the next result consists of an induction argument
that is similar to the one used in the proof
of~\cite[Proposition 5.5]{JA}. Here, the induction basis is given by
Lemma~\ref{sec:12}, and Corollary~\ref{c:4} plays the role of~\cite[Lemma 5.2]{JA}.
\begin{prop}
  \label{JA5.5}
  Let $x_0, x_1, \ldots, x_n \in \pseudo AS$ be such that $x_0^\omega
  = x_1^\omega = \cdots = x_n^\omega$ modulo $\DRH$ and suppose that,
  for $i = 0, 1, \ldots, n$, the product $x_i\cdot x_i$ is
  reduced. Then, there exist pseudowords $u \in \pseudo AS$, $v_0,
  v_1, \ldots, v_n \in (\pseudo AS)^I$, and positive integers $p_0, p_1,
  \ldots, p_n$ such that the pseudovariety $\DRH$ satisfies
  \begin{align*}
    x_i &= u^{p_i} v_i, \quad \text{ for } i = 0, 1, \ldots, n,
    \\ u & = v_i u,\quad \text{ for } i = 0, 1, \ldots, n,
  \end{align*}
  and all the products $u\cdot u$, $u\cdot v_i$, and
  $v_i\cdot u$ are reduced.
\end{prop}
\section{Some simplifications concerning reducibility}\label{section4}
Almeida, Costa and Zeitoun \cite{JA} proved that, in order to
achieve complete $\kappa$-reducibility, it is enough to consider
systems of $\kappa$-equations with empty set of parameters (in fact,
they proved the result more generally, for any implicit signature~$\sigma$).

\begin{prop}[{\cite[Proposition 3.1]{JA}}]
  Let $\V$ be an arbitrary pseudovariety. If $\V$ is
  $\kappa$-reducible for systems of $\kappa$-equations without
  parameters, then $\V$ is completely $\kappa$-reducible.
\end{prop}

A pseudovariety $\V$ is said to be \emph{weakly cancellable} if
whenever $\V$ satisfies $u_1au_2 = v_1av_2$ with $a$ 
not belonging to any of the sets $c(u_1)$, $c(u_2)$, $c(v_1)$, and
$c(v_2)$, it also satisfies
$u_1 = u_2$ and $v_1 = v_2$. When $\V$ is a weakly cancellable
pseudovariety, we 
may restrict our study to systems consisting of one single
$\kappa$-equation without parameters.
\begin{prop}[{\cite[Proposition 3.2]{JA}}]
  \label{p:6}
  Let $\V$ be a weakly cancellable pseudovariety. If~$\V$ is
  $\kappa$-reducible for systems consisting of just one
  $\kappa$-equation without parameters, then $\V$ is completely
  $\kappa$-reducible.
\end{prop}
Of course, the pseudovariety
$\DRH$ is weakly cancellable. Indeed, weak cancellability is a
particular instance of uniqueness of the first-occurrences factorization
(recall Corollary~\ref{c:5}).
Actually, we may go even further and, similarly to the case of $\R$
(see~\cite[Lemmas~6.1 and~6.2]{JA}), we prove that, in order to
obtain
complete $\kappa$-reducibility of a pseudovariety~$\DRH$, it suffices
to consider systems of word equations (without parameters).
\begin{lemma}\label{6.1}
  Let $u,v \in \pseudo AS$. Then, $\DRH$ satisfies the pseudoidentity
  $u = v^{\omega - 1}$ if and only if $c(u) = c(v)$, and the
  pseudoidentities 
  $uvu =u$ and $uv = vu$ hold in $\DRH$.
\end{lemma}
\begin{proof}
  Suppose that $\DRH$ satisfies $u = v^{\omega-1}$.
  Since the semigroup~$\pseudo
  A{DRH}$ has a content function, we have $c(u) = c(v^{\omega - 1}) =
  c(v)$. In order to verify that the pseudoidentities
  $uvu=u$ and $uv=vu$ are valid in ${\DRH}$, we may perform
  the following computations:
  \begin{align*}
    u & =_\DRH v^{\omega-1} = v^{\omega - 1}\: (v v^{\omega-1}) =_\DRH uvu,
    \\ uv & =_\DRH v^{\omega-1}v = vv^{\omega-1} =_\DRH vu.
  \end{align*}
  Conversely,
  suppose that $\DRH$ satisfies the pseudoidentities $uvu=u$ and
  $uv=vu$, and $c(u) = c(v)$.
  Then, the following pseudoidentities are valid in
  $\DRH$:
  \begin{align*}
    v^{\omega-1} &= v^{\omega-1} u^\omega
    \qquad\text{by Corollary~\ref{c:2}}
    \\ & = v^{\omega-1} u^{\omega-1} u = (uv)^{\omega-1}u  \qquad
         \text{because $uv =_\DRH vu$}
    \\ & = (uv)u
    \qquad\text{because } uvu =_\DRH u \text{ implies } (uv)^{\omega-1} =_\DRH uv
    \\ & = u.
  \end{align*}
  This concludes the proof.
\end{proof}
Lemma~\ref{6.1} allows us to transform each $\kappa$-equation into
a finite
system of word equations. Therefore, by Proposition~\ref{p:6}, in order to
prove the complete $\kappa$-reducibility of~$\DRH$, it is enough to
consider systems consisting of a single word equation.
We do not include the details of that step, as it is entirely
analogous to~\cite[Proposition 6.2]{JA}.

\begin{prop}\label{6.2}
  The pseudovariety $\DRH$ is completely $\kappa$-reducible if and
  only if it is $\kappa$-reducible for a single word equation without
  parameters.
\end{prop}

Let $u,v \in X^+$ and $\delta: \pseudo XS \to \pseudo
AS$ be a solution modulo $\DRH$ of $u = v$, subject to the constraints
given by the pair $(\varphi: \pseudo AS \to S, \nu:X \to S)$.
The last simplification consists in transforming the word equation $u =
v$ into a more convenient system of equations, namely,
into a system that we denote by
$\iS_{u = v}$ and that is the union of 
systems $\{u'=v'\}$, $\iS_1$ and $\iS_2$ with variables in~$X'$.
We construct $\iS_{u = v}$
inductively as follows.

We use an auxiliary system $\iS_0$ and start with $\iS_0 = \iS_1 =
\iS_2 = \emptyset$,
$X' = X$, $u'= u$, and $v' = v$.
Since~$\DRH$ is a weakly cancellable pseudovariety, the word equation
$u = v$ is equivalent to the equation $u \# = v\#$, where $\# \notin A$ is a
parameter evaluated to itself.
Suppose that, whenever $xy$ is a factor of $u\#v\#$ ($x,y \in X$), the
product $\delta(x)\cdot \delta(y)$ is reduced. Then, we say that the
solution $\delta$ is
\emph{reduced with respect to the equation $u =
  v$}.
If $\delta$ is not reduced with respect to $u = v$, then we pick a
factor $xy$ such that $\delta(x)\delta(y)$ is not 
a reduced product and we distinguish between two situations:
\begin{itemize}
\item If $c(\delta(y)) \subseteq \vec{c}(\delta(x))$, then we add
  a new variable $z$ to $X'$ and we put the equation $xy = z$ in
  $\iS_1$. We also redefine $u'$ and $v'$ by substituting each
  occurrence of the product
  $xy$ in the equation $u'\#v'\#$ by the variable $z$.
\item If $c(\delta(y)) \nsubseteq \vec{c}(\delta(x))$, then we add three
  new variables $y_1$, $y_2$, and $z$ to $X'$ and we put the
  equations $y = y_1y_2$ and $z = xy_1$ in $\iS_0$ and $\iS_1$,
  respectively. We also redefine $u'$ and $v'$ by substituting the product
  $xy$ in the equation $u'\#v'\#$ by the product of variables $zy_2$.
\end{itemize}
In both situations, we can factorize $\delta(y) =
\delta(y)_1\delta(y)_2$, with $\delta(y)_2$ possibly an empty word,
such that $c(\delta(y)_1) \subseteq \vec{c}(\delta(x))$ and the
product $(\delta(x)\delta(y)_1)\cdot \delta(y)_2$ is reduced if
$\delta(y)_2 \neq I$.
We extend $\delta$ to $\pseudo {X'}S$ by letting $\delta(z) =
\delta(x)\delta(y)_1$ and, whenever we are in the second situation,
by letting
$\delta(y_i) = \delta(y)_i$ ($i = 1,2$). Of course,~$\delta$ is a solution modulo
$\DRH$ of the new system of equations $\{u' = v'\} \cup \iS_0 \cup
\iS_1$.

We repeat the described process until the extended solution
$\delta$ is reduced with respect to the equation $u' = v'$. Since
$u$ and $v$ are both words, we have for granted that this iteration
eventually ends. Yet, the extension of $\delta$ to $\pseudo
{X'}S$ (which is a solution modulo $\DRH$ of
$\{u'=v'\}\cup\iS_0\cup\iS_1$) has the property of being reduced with
respect to the equation $u'=v'$.
We further observe that the
resulting system $\iS_1$ may be written as $\iS_1 =
\{x_{(i)}y_{(i)} = z_{(i)}\}_{i = 1}^ N$ and its extended solution
$\delta$ satisfies $c(\delta(y_{(i)})) \subseteq \cum{\delta(x_{(i)})}$.
For each variable $x \in X'$, we set $A_x = \vec{c}(\delta(x))$ and
define $\iS_2=\{x a^{\omega} = x\colon a \in
A_x\}_{x \in X'}$.
The homomorphism $\delta$ is a solution modulo
$\DRH$  of $\iS_2$. Finally, since $\DRH$ is weakly cancellable and
all the products $\delta(y_1)\cdot
  \delta(y_2)$ are reduced, we may
assume that the satisfaction of the equations in
$\iS_0$ by $\delta$ is a consequence of the satisfaction of the
equation $u' = v'$ by $\delta$, without losing the reducibility of
$\delta$ with respect to $u'=v'$. More specifically, if $y = y_1y_2$
is an equation of $\iS_0$, then we take for $u'$ the word $u'\overline\# y$
and for $v'$ the word $v' \overline\# y_1 y_2$, where $\overline\#$ is a new symbol,
working as a parameter evaluated to itself.
In the same fashion, we may also assume
that all the variables of $X'$ occur in $u' = v'$.
Although at the moment it may not be clear to the reader why we
wish that all the variables in $X'$  occur in the equation $u' = v'$,
that becomes useful later, when dealing with certain systems of
equations modulo~$\h$ that intervene in the so-called ``systems of boundary
relations''.
The
resulting system $\{u' = v'\} \cup \iS_1 \cup \iS_2$ is
the one that we denote by $\iS_{u=v}$ and it also has a solution modulo
$\DRH$. The constraints for the variables in $X'$ are those defined by
the described extension of $\delta$ to $\pseudo{X'}S$, namely, we put
$\nu(x) = \varphi(\delta(x))$ for each $x \in X'$.

Conversely, suppose that $\iS_{u = v}$ has a solution modulo $\DRH$
in $\kappa$-words, say $\varepsilon$. Then, it is easily checked
that, by construction, the restriction of $\varepsilon$ to $\pseudo
XS$ is a solution modulo $\DRH$ of the original equation $u = v$.
Moreover, by definition of $\iS_2$, this
solution is such that~$\vec{c}(\varepsilon(x)) =
\vec{c}(\delta(x))$, for all $x \in X'$. As, in addition, $S$ has a
content function, the satisfaction of the constraints yields
that~$c(\varepsilon(y_{(i)})) = c(\delta(y_{(i)}))$ and, in
particular, the inclusion~$c(\varepsilon(y_{(i)}))\subseteq \vec{c}
(\varepsilon(x_{(i)}))$ holds
for all the equations~$x_{(i)} y_{(i)} = z_{(i)}$ in $\iS_1$.

Taking into account Proposition~\ref{6.2}, we have just proved the
following result in which we use the above notation.

\begin{prop}\label{6.3}
  Suppose that the pseudovariety $\DRH$ is $\kappa$-reducible for
  systems of equations of the form
  \begin{equation}
    \label{sist}
    \iS_{u=v} = \{u' = v'\} \cup \iS_1 \cup \iS_2,
  \end{equation}
  where $u'=v'$ is a word equation, $\iS_1 = \{x_{(i)}y_{(i)} =
  z_{(i)}\}_{i = 1}^N$ and $\iS_2 = \{xa^\omega=x \colon a \in A_x\}_{x
    \in X}$,
  which have a solution $\delta$ modulo $\DRH$ that is reduced with respect to
  the equation $u'=v'$ and satisfies $c(\delta(y_{(i)})) \subseteq
    \cum{\delta(x_{(i)})}$, for $i = 1, \ldots, N$. Then, the
  pseudovariety $\DRH$ is completely $\kappa$-reducible.
\end{prop}
\begin{rem}
  It is sometimes more convenient to allow $\delta$ to take its values
  in $(\pseudo AS)^I$.
  For this purpose, we naturally extend the function $\varphi$ to a 
  continuous homomorphism $\varphi^I: (\pseudo AS)^I \to S^I$ by letting
  $\varphi^I(I) = I$.
  It is worth noticing that this assumption does not lead us to trivial
  solutions since the constraints must be satisfied.
  We allow ourselves some flexibility in this point, adopting each scenario
  according to each particular situation, without further mention.
  In the case where we consider the homomorphism $\varphi^I$, we abuse
  notation and denote it by $\varphi$.
\end{rem}

We end this section with a result regarding reducibility of
pseudovarieties of groups that is later used to derive reducibility
properties of $\DRH$.

\begin{lemma}\label{cumH}
  Let $\h$ be a completely $\kappa$-reducible pseudovariety of groups
  and $\iS$ a finite system of $\kappa$-equations with constraints given by
  the pair
  $(\varphi:(\pseudo AS)^I \to S^I, \nu:X \to S^I)$, and with $\delta:
  \pseudo XS \to (\pseudo AS)^I$ as a solution modulo $\h$. Then $\iS$
  has a solution modulo $\h$ in $\kappa$-words, say $\varepsilon$, such
  that $\cum{\varepsilon(x)} = \cum{\delta(x)}$ for all $x \in X$.
\end{lemma}

\begin{proof}
   We prove the result by induction on $m =
  \max\{\card{c(\delta(x))}\colon x \in X\}$. If $m = 0$, then there
  is nothing to prove.
  Otherwise, let $x$ be a variable of $X$.
  Given $i \le \leng{\delta(x)}$, we denote $\lbf[i]{\delta(x)}$
  by $\delta(x)_{i}a_{x,i}$ and write $\delta(x)
  = \lbf[1]{\delta(x)} \cdots \lbf[i]{\delta(x)}\delta(x)_i'$. 
  If $\cum{\delta(x)}$ is the
  empty set, then we have
  \begin{equation}
    \label{eq:63}
    \varphi(\delta(x)) = \varphi(\lbf[1]{\delta(x)}\cdots \lbf[\leng{\delta(x)}]{\delta(x)}).
  \end{equation}
  For the remaining
  variables, since $X$, $A$, and $S$ are finite, there are integers
  $1<k<\ell$ such that
  \begin{align*}
   \cum{\delta(x)}  &= c(\lbf[k+1]{\delta(x)});
    \\  \varphi(\lbf[1]{\delta(x)} \cdots \lbf[k]{\delta(x)}) &=
         \varphi(\lbf[1]{\delta(x)} \cdots \lbf[\ell]{\delta(x)}),
  \end{align*}
  {for all $x\in X$ with $\cum{\delta(x)} \neq \emptyset$.}
  In particular, from the second equality we deduce
  \begin{equation}
    \label{eq:41}
    \varphi(\delta(x)) =
    \varphi(\lbf[1]{\delta(x)} \cdots
    \lbf[k]{\delta(x)})\varphi(\lbf[k+1]{\delta(x)} \cdots
    \lbf[\ell]{\delta(x)})^\omega\varphi(\delta(x)_k').
  \end{equation}
  We
  consider a new set of variables $X'$ given by
  \begin{align*}
    X'& = \{y_{x,1}, b_{x,1}, \ldots, y_{x, \leng{\delta(x)}}, b_{x,
    \leng{\delta(x)}}\colon x \in X \text{ and } \cum{\delta(x)} =
        \emptyset\}
    \\ &\quad \uplus\{y_{x,1}, b_{x,1}, \ldots, y_{x, \ell}, b_{x,
    \ell}, y_x' \colon x \in X \text{ and } \cum{\delta(x)} \neq \emptyset\}
  \end{align*}
  and a new system of equations $\iS'$ with variables in
  $X'$ obtained from $\iS$ by substituting each variable $x$ by the
  product 
  \begin{equation}
    P_x = y_{x,1}b_{x,1} \cdots y_{x,\leng{\delta(x)}}b_{x,
      \leng{\delta(x)}},\label{eq:39}
  \end{equation}
  whenever $\cum{\delta(x)} = \emptyset$, and by the product
  \begin{equation}
    P_x = y_{x,1}b_{x,1} \cdots y_{x,k}b_{x,k}(y_{x,k+1}b_{x,k+1} \cdots
    y_{x,\ell}b_{x,\ell})^\omega y_x',\label{eq:40}
  \end{equation}
  otherwise.
  Let us define the constraints for the variables in~$X'$.
  Let $a \in A$ be a letter. Since $\{a\}$ is a clopen subset of
  $\pseudo AS$, by Hunter's Lemma there exists a continuous
  homomorphism $\varphi_a:\pseudo AS \to S_a$ such that $\{a\} =
  \varphi^{-1}(\varphi(\{a\}))$.
  Representing by $\prod_{a \in A} S_a$
  the direct product of the
  semigroups $S_a$, we let the constraints be given by the pair
  $(\varphi', \nu')$, where $\varphi'$ is the following continuous homomorphism
  \begin{align*}
    \varphi': \pseudo AS &\to S^I \times \prod_{a \in A} S_a
    \\ u & \mapsto (\varphi(u), (\varphi_a(u))_{a \in A}),
  \end{align*}
  and $\nu'$ is the mapping
  \begin{align*}
    \nu': X' & \to S^I \times \prod_{a \in A} S_a
    \\ y_{x,i} & \mapsto \varphi'(\delta(x)_i),
    \\ y_x' &\mapsto \varphi'(\delta(x)'_k),
    \\ b_{x,i} & \mapsto \varphi'(a_{x,i}).
  \end{align*}
  Since $\h$ satisfies $\delta'(P_x) = \delta(x)$, for every
  variable $x \in X$
  (check~\eqref{eq:39} and~\eqref{eq:40}), the homomorphism $\delta'$
  is a solution modulo $\h$ of~$\iS'$. Therefore, as we are
  assuming that the pseudovariety $\h$ is completely
  $\kappa$-reducible,
  there is a solution $\varepsilon': \pseudo
  {X'}S \to
  \pseudo AS$ modulo $\h$ of $\iS'$ such that $\varepsilon'(X')
  \subseteq \pseudok AS$.
  On the other hand, this
  homomorphism $\varepsilon'$ defines a solution in $\kappa$-words
  modulo~$\h$ of the
  original system $\iS$, namely, by letting $\varepsilon(x) =
  \varepsilon'(P_x)$ for each~$x \in X$. Moreover, by definition of
  $(\varphi', \nu')$, we necessarily have $\varepsilon'(b_{x,i}) = a_{x,i}$
  and the fact that $S$ has a content function entails that
  $c(\varepsilon'(y_{x,i})) = c(\delta'(y_{x,i})) = c(\delta(x)_i)$
  and, similarly, that $c(\varepsilon'(y_x')) = c(\delta'(y_x')) =
  c(\delta(x)_k')$. In particular, $a_{x,i}$ does not belong to
  $c(\delta(x)_i)$. So, the
  iteration of left factorization to the right of $\varepsilon(x)$ is
  the one induced by the product~$P_x$, implying that
  $\cum{\varepsilon(x)} = \cum{\delta(x)}$ as intended.
  Finally, we verify that the constraints on~$X$ are satisfied
  by~$\varepsilon$. Taking into account that the definition of
  $(\varphi',\nu')$ yields the equalities $\varphi(\varepsilon'(y_{x,i}b_{x,i})) =
  \varphi(\lbf[i]{\delta(x)})$ and $\varphi(\varepsilon'(y_{x}')) =
  \varphi(\lbf[k]{\delta(x)_k'})$ (for $x \in X$ and $i = 1, \ldots,
  k$), we may compute
  \begin{align*}
    \varphi(\varepsilon(x)) &
    =
      \begin{cases}
        \varphi(\varepsilon'(y_{x,1}b_{x,1} \cdots y_{x,\leng{\delta(x)}}b_{x,
          \leng{\delta(x)}})), \quad\text{ if } \cum{\delta(x)} = \emptyset
        \\ \varphi(\varepsilon'(y_{x,1}b_{x,1} \cdots y_{x,k}b_{x,k}(y_{x,k+1}b_{x,k+1} \cdots
        y_{x,\ell}b_{x,\ell})^\omega y_x')), \quad\text{ otherwise}
      \end{cases}
    \\ & =
         \begin{cases}
           \varphi(\lbf[1]{\delta(x)} \cdots
           \lbf[{\leng{\delta(x)}}]{\delta(x)}),
           \quad\text{ if } \cum{\delta(x)} = \emptyset
        \\  \varphi(\lbf[1]{\delta(x)} \cdots \lbf[k]{\delta(x)})
        \\ \quad \cdot \varphi(\lbf[k+1]{\delta(x)} \cdots
         \lbf[\ell]{\delta(x)} )^\omega \varphi(\delta(x)'_k), \quad\text{ otherwise}
       \end{cases}
    \\ & \kern-13pt\just ={\eqref{eq:63},\,\eqref{eq:41}} \varphi(\delta(x)).
  \end{align*}
  Hence, the homomorphism $\varepsilon$ plays the desired role.
\end{proof}
\section{Systems of boundary relations and their models}\label{section5}

In this section, we define some tools that turn out to be
useful when proving that $\DRH$ is completely $\kappa$-reducible. The
original notion of a boundary equation was given by
Makanin~\cite{Makanin} and it was later adapted by Almeida, Costa and
Zeitoun~\cite{JA} to deal with the problem of complete $\kappa$-reducibility of
the pseudovariety $\R$. Here, we extend the
definitions used in~\cite{JA} to the context of the pseudovariety
$\DRH$, for any pseudovariety of groups $\h$, and use them to prove
that, under certain conditions, the
pseudovariety $\DRH$ is completely $\kappa$-reducible.

From hereon, we fix a word equation $u = v$ and a solution $\delta:
\pseudo XS \to \pseudo AS$
modulo $\DRH$ of
$\iS_{u = v}$ (recall~\eqref{sist}), subject to the constraints given
by the
pair $(\varphi:\pseudo AS \to S, \nu: X \to S)$.
By a \emph{system of boundary relations} we mean a tuple $\iS =
(\iX, J, \zeta, M, \chi, \dirr, \iB, \iB_{\h})$ where
\begin{itemize}
\item $\iX$ is a finite set equipped with an involution without
  fixed points
  $x \mapsto \xx$, whose elements are called
  \emph{variables};
\item $J$ is a finite set equipped with a total order $\le$, whose
  elements are called
  \emph{indices}. If $i$ and $j$ are two
  consecutive indices, then we write $i \prec j$ and we denote $i$ by
  $j^-$;
\item $\zeta: \{(i,j) \in J \times J \colon i \prec j\} \to 2^{S
    \times S^I}$ is a function that is useful to deal with the
  constraints;
\item $M:\{(i,j,\vs) \in J \times J \times (S \times S^I) \colon i
  \prec j, \: s \in \zeta(i,j)\} \to \omega \setminus \{0\}$ is a function
  that determines the number of different factorizations
  in $\pseudo AS$ modulo $\DRH$ that we
  assign to each variable of $\iX$;
\item $\chi: \{(i,j) \in J \times J \colon i \prec j\} \to 2^A$ is a
  function whose aim is to fix the cumulative content of each variable;
\item $\dirr: \iX \to J$ is a function that helps in defining the
  relations we need to attain our goal;
\item $\iB$ is a subset of $J \times \iX \times J \times \iX$, whose
  elements are of the form $(i,x,j,\xx)$. Moreover, if $(i,x,j,\xx)$ is
  an element of $\iB$, then so is $(j, \xx, i, x)$. The elements of
  $\iB$ are called
  \emph{boundary relations} and the boundary relation
  $(j, \xx, i, x)$ is said to be the
  \emph{dual boundary relation} of
  $(i,x,j,\xx)$.
  The pairs $(i,x)$ and $(j, \xx)$ are
  \emph{boxes of $\iB$}.
  Together with the $\dirr$ function, the set $\iB$
  encodes the relations we want to be satisfied  in~$\DRH$;
\item finally, for each pair of indices $i,j$ such that $i \prec j$,
  we consider a symbol $(i \mid j)$ and, for each pair $(\vs, \mu) \in
  \zeta(i,j)$, we consider another symbol $\{i \mid j\}_{\vs,
    \mu}$. These symbols are understood as variables and we denote by
  $X_{(J,\zeta, M)}$ the set of those variables:
  \begin{equation}
    \begin{aligned}
      X_{(J,\zeta,M)} & = \{(i \mid j) \colon i,j \in J, \: i \prec j\}
      \\ & \quad  \cup \{\{i
      \mid j\}_{\vs, \mu} \colon (i,j,\vs) \in \dom(M) \text{ and } \mu
      \in M(i,j,\vs)\}
    \end{aligned}\label{var}
  \end{equation}
  Then, $\iB_{\h}$ is a finite set of $\kappa$-equations with
  variables
  in $X_{(J, \zeta, M)}$ whose solutions are meant to be taken over~$\h$. If $i_0
  \prec \cdots \prec i_n$ is a chain of indices in $J$, then we denote
  by $(i_0\mid i_n)$ the product of variables
  $\prod_{k=1}^n(i_{k-1} \mid i_k)$.
\end{itemize}
Given a variable $x \in \iX$, the \emph{left of x}
is the index $$\esq(x) = \{i
\in I\colon \text{ there exists a box }(i,x) \text{ in } \iB\}.$$

We let
${\sf{prod}}: \pseudo AS \times (\pseudo AS)^I\to \pseudo AS$
be the function sending each pair of pseudowords~$(u,v)$ to its product
$u v$.

A \emph{model}
of the system of boundary relations $\iS$ is a triple
$\iM = (w, \iota, \Theta)$, where
\begin{itemize}
\item $w$ is a possibly empty pseudoword;
\item $\iota: J \to \alpha_w + 1$ is an injective function that
  preserves the order and such that, if $J$ is not the empty set then
  $\iota$ sends $\min(J)$ to $0$ and $\max(J)$ to $\alpha_w$;
\item for each triple $(i,j,\vs)$ in $\dom(M)$ and each $\mu$ in
  $M(i,j,\vs)$,  $\Theta(i,j,\vs, \mu)$ is a pair
  $(\Phi(i,j,\vs, \mu), \Psi(i,j,\vs, \mu))$ of $\pseudo AS \times
  (\pseudo AS)^I$ such that $c(\Psi(i,j, \vs, \mu))\subseteq
  \cum{\Phi(i,j,\vs, \mu)}$.
\end{itemize}
\begin{notacao}
  When there exists a map $\iota: J \to \alpha_w+1$ as above, we may
  write $w(i,j)$ instead of $w[\iota(i), \iota(j){[}$ (recall
  Notation~\ref{sec:19}).
\end{notacao}
 Moreover, the following properties are required for $\iM$:
 \begin{enumerate}[label = (M.\arabic*)]
 \item \label{m1} if $(i,j,\vs) \in \dom(M)$ and
   $\mu \in M(i,j,\vs)$, then $\Thetap(i,j,\vs,\mu) =_\DRH
   w(i,j)$;
 \item \label{m2} if $(i,j,\vs) \in \dom(M)$, $\vs
   = (s_1, s_2)$, and $\mu \in M(i,j,\vs)$, then
   $$\varphi(\Phi(i,j,\vs,\mu)) = s_1 \text{ and }
   \varphi(\Psi(i,j,\vs,\mu)) = s_2\text{;}$$
 \item \label{m3} if $i\prec j$, then $\cum{w(i,j)}
   = \chi(i,j)$;
 \item \label{m4} if $(i,x,j,\xx) \in \iB$, then
   $\DRH$ satisfies $w(i, \dirr (x)) \Req w(j, \dirr(\xx))$;
 \item \label{m5} let $\iC := (J, \iota, M, \Theta)$
   and $\delta_{w, \iC}: \pseudo {X_{(J, \zeta, M)}} S \to \pseudo AS$
   be the
   unique continuous homomorphism defined by
   \begin{equation}
     \begin{aligned}
        \delta_{w, \iC}(i\mid j) &= w(i,j), \\  \delta_{w, \iC} (\{i\mid
       j\}_{\vs, \mu}) &= \Psi(i,j,\vs, \mu).
     \end{aligned}\label{eq:48}
   \end{equation}
   Then, $\delta_{w, \iC}$ is a solution modulo $\h$ of $\iB_{\h}$.
 \end{enumerate}
 We say that $\iM$ is a
 \emph{model of $\iS$ in $\kappa$-words} if $w \in
 (\pseudok AS)^I$  and the coordinates of $\Theta$ are given by
 $\kappa$-words.
By Proposition~\ref{6.3}, to prove that $\DRH$ is completely
 $\kappa$-reducible, it is enough to prove that $\DRH$ is
 $\kappa$-reducible for certain systems of equations
 of the form $\iS_{u=v}$.
 With that in mind, we associate to such a system $\iS_{u =
   v}$ a system of boundary
 relations,
 denoted $\overline{\iS}_{u=v}$. Then, we construct a
 model of $\overline{\iS}_{u=v}$ and prove that the existence of a
 model in $\kappa$-words entails the existence of a solution of the
 original system $\iS_{u = v}$ also in
 $\kappa$-words (Theorem~\ref{main}).
  
 Let $\delta: \pseudo XS \to \pseudo AS$ be a solution modulo $\DRH$ of $\iS_{u = v} =
 \{u' = v'\} \cup \iS_1 \cup \iS_2$  such that $\delta$ is reduced with respect
 to $u' = v'$ and for every equation $xy =z$ of $\iS_1$ we have
 $c(\delta(y)) \subseteq \cum{\delta(x)}$ (recall Proposition~\ref{6.3}).
 Suppose that $u' =
 x_1 \cdots x_r$ and $v' = x_{r+1} \cdots x_t$, and write $\iS_1 =
   \{x_{(i)}y_{(i)} = z_{(i)}\}_{i = 1}^N$ and $\iS_2 = \{xa^\omega =
   x \colon a \in A_x\}_{x \in X}$. Let $G$ be an
 undirected graph whose vertices are given by the set $\{1, \ldots,
 t\}$ and that has an edge connecting the vertices $p$ and $q$  if and
 only if $p \neq q$ and either $x_p = x_q$ or $\{x_p, x_q\} = \{x_{(i)},
z_{(i)}\}$ for a certain $i$. Let $\widehat G$ be a spanning forest for
$G$. We define
\begin{align}\label{S}
\overline{\iS}_{u=v} = (\iX,
J, \zeta, M, \chi, \dirr, \iB, \iB_{\h})
\end{align}
as follows:
\begin{itemize}
\item the set of variables is $$\iX = \{(p,q) \colon \text{ there is an
    edge in } \widehat G \text{ connecting } p \text{ and } q\} \uplus \{\tl\}
  \uplus\{\tr\},$$ and the involution in $\iX$ is given by $\overline{(p,q)}
  = (q,p)$ and by $\overline{\tl} = \tr$;
\item the set of indices is $J = \{i_0, \ldots, i_t\}$ with $i_0
  \prec \cdots \prec i_t$;
\item the function $\zeta$ is defined by $\zeta(i_{p-1}, i_p) =
  \{(\nu({x_p}), I)\}$ for every $p = 1, \ldots, t$;
\item we set $M(i_{p-1}, i_p, (\nu({x_p}), I)) = 1$ for
  every $p = 1, \ldots, t$;
\item the function $\chi$ sends each pair $(i_{p-1}, i_p)$ to the set
  $A_{x_p}$;
\item the $\dirr$ function is given by $\dirr(p,q) = i_p$, $\dirr(\tl) =
  i_r$, and $\dirr(\tr) = i_t$;
\item the set of boundary relations $\iB$ contains the boundary
  relations $(i_0, \tl, i_r, \tr)$, and $(i_r, \tr,
  i_0, \tl)$ plus all the boundary relations of the form $(i_{p-1},
  (p,q), i_{q-1}, (q, p))$, where $(p,q) \in \iX$;
\item we put in $\iB_{\h}$ the equations $(i_0 \mid i_r) = (i_r \mid
  i_t)$ and $(i_{p-1} \mid i_p) = (i_{q-1}\mid i_q)$,
  whenever $x_p = x_q$, and the equation $(i_{p-1}\mid i_p)
  (i_{m-1} \mid i_m) = (i_{q-1} \mid i_q)$ for each $x_px_m = x_q \in \iS_1$.
\end{itemize}

\begin{ex}
  Let $X = \{x,y,z\}$, $u = xyx$, $ v = x^2z$, and let $\delta: \pseudo X S
  \to \pseudo AS$ be defined by $\delta(x) = a$, $\delta(y) =
  (ab)^{p^\omega}$, and $\delta(z) = (ba)^{p^\omega}$.
  Clearly, the homomorphism $\delta$ is a solution modulo $\DRH$ of $u
  = v$ and the system $\iS_{u=v} = \{u'=v'\} \cup \iS_1\cup \iS_2$ is given by
  $u' = xt_{yx}\#_1 y \#$, $v' =  x^2 z \#_1 y\#$, $\iS_1 = \{t_{yx} =
  yx\}$, and $\iS_2 = \{ya^\omega = y, yb^\omega = y, t_{yx}a^\omega = t_{yx},
  t_{yx}b^\omega = t_{yx}\}$.
  The extended solution $\delta$ is obtained by letting
  $\delta(t_{yx}) = (ab)^{p^\omega} a$.
  Then, the set of indices is $J = \{i_0, i_1, \ldots, i_{11}\}$.
  Although the graph~$G$ is unique, there are several
  possibilities for~$\widehat G$, so that the set of variables $\iX$
  is not uniquely determined. One of the possible choices of~$\widehat G$
  produces the following~$\iX$:
  $$\iX = \{(1,6), (6,1), (6,7),(7,6), (2,4), (4,2), (3,9), (9,3), (5,11),
  (11,5), \tl, \tr\}.$$
  We schematize the set of boundary relations $\iB$ in Fig.~\ref{fig:21}.
  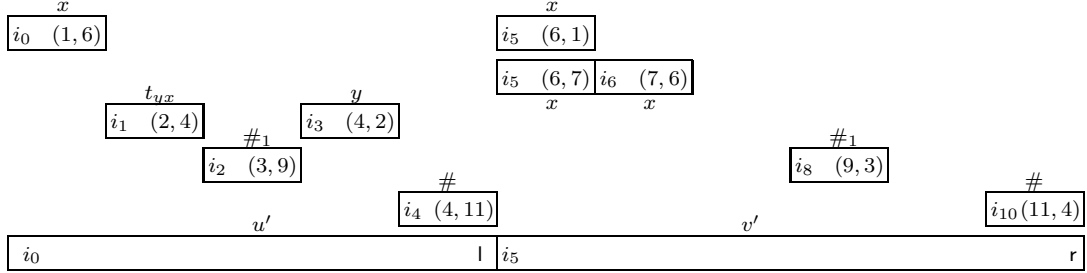
\begin{figure}[htpb]
    \centering
    \setlength{\unitlength}{1.3cm}
    \begin{picture}(10,2.45)(1.4,1.9)
      \put(1,4.35){\line(1,0){1}}\put(1,4){\line(0,1){0.35}}\put(1.05,4.1){\scriptsize$i_0$}
      \put(1,4){\line(1,0){1}}\put(2,4){\line(0,1){0.35}}\put(1.45,4.1){\scriptsize$(1,6)$}
      \put(6,4.35){\line(1,0){1}}\put(6,4){\line(0,1){0.35}}\put(6.05,4.1){\scriptsize$i_5$}
      \put(6,4){\line(1,0){1}}\put(7,4){\line(0,1){0.35}}\put(6.45,4.1){\scriptsize$(6,1)$}
      \put(1.5,4.4){\scriptsize{$x$}}
      \put(6.5,4.4){\scriptsize{$x$}}
      \put(6,3.9){\line(1,0){1}}\put(6,3.55){\line(0,1){0.35}}\put(6.05,3.65){\scriptsize$i_5$}
      \put(6,3.55){\line(1,0){1}}\put(7,3.55){\line(0,1){0.35}}\put(6.45,3.65){\scriptsize$(6,7)$}
      \put(7,3.9){\line(1,0){1}}\put(7,3.55){\line(0,1){0.35}}\put(7.05,3.65){\scriptsize$i_6$}
      \put(7,3.55){\line(1,0){1}}\put(8,3.55){\line(0,1){0.35}}\put(7.45,3.65){\scriptsize$(7,6)$}
      \put(6.5,3.4){\scriptsize{$x$}}
      \put(7.5,3.4){\scriptsize{$x$}}
      \put(2,3.45){\line(1,0){1}}\put(2,3.1){\line(0,1){0.35}}\put(2.05,3.2){\scriptsize$i_1$}
      \put(2,3.1){\line(1,0){1}}\put(3,3.1){\line(0,1){0.35}}\put(2.45,3.2){\scriptsize$(2,4)$}
      \put(4,3.45){\line(1,0){1}}\put(4,3.1){\line(0,1){0.35}}\put(4.05,3.2){\scriptsize$i_3$}
      \put(4,3.1){\line(1,0){1}}\put(5,3.1){\line(0,1){0.35}}\put(4.45,3.2){\scriptsize$(4,2)$}
      \put(2.4, 3.5){\scriptsize{$t_{yx}$}}
      \put(4.5, 3.5){\scriptsize{$y$}}
      \put(3,3){\line(1,0){1}}\put(3,2.65){\line(0,1){0.35}}\put(3.05,2.75){\scriptsize$i_2$}
      \put(3,2.65){\line(1,0){1}}\put(4,2.65){\line(0,1){0.35}}\put(3.45,2.75){\scriptsize$(3,9)$}
      \put(9,3){\line(1,0){1}}\put(9,2.65){\line(0,1){0.35}}\put(9.05,2.75){\scriptsize$i_8$}
      \put(9,2.65){\line(1,0){1}}\put(10,2.65){\line(0,1){0.35}}\put(9.45,2.75){\scriptsize$(9,3)$}
      \put(3.4,3.05){\scriptsize$\#_1$}
      \put(9.4,3.05){\scriptsize$\#_1$}
      \put(5,2.55){\line(1,0){1}}\put(5,2.2){\line(0,1){0.35}}\put(5.05,2.3){\scriptsize$i_4$}
      \put(5,2.2){\line(1,0){1}}\put(6,2.2){\line(0,1){0.35}}\put(5.35,2.3){\scriptsize$(4,11)$}
      \put(11,2.55){\line(1,0){1}}\put(11,2.2){\line(0,1){0.35}}\put(11.05,2.3){\scriptsize$i_{10}$}
      \put(11,2.2){\line(1,0){1}}\put(12,2.2){\line(0,1){0.35}}\put(11.35,2.3){\scriptsize$(11,4)$}
      \put(5.4,2.6){\scriptsize$\#$}
      \put(11.4,2.6){\scriptsize$\#$}
      \put(1,2.1){\line(1,0){6}}\put(1,1.75){\line(0,1){0.35}}\put(1.15,1.85){\scriptsize$i_0$}
      \put(1,1.75){\line(1,0){6}}\put(6,1.75){\line(0,1){0.35}}\put(5.8,1.85){\scriptsize$\tl$}
      \put(6,2.1){\line(1,0){6}}\put(6.05,1.85){\scriptsize$i_5$}
      \put(6,1.75){\line(1,0){6}}\put(12,1.75){\line(0,1){0.35}}\put(11.85,1.85){\scriptsize$\tr$}
      \put(3.5,2.15){\scriptsize$u'$}\put(8.5,2.15){\scriptsize$v'$}
    \end{picture}
    \caption{The set of boundary relations $\iB$.}
  \label{fig:21}
  \end{figure}

  Finally, the set $\iB_\h$ contains the equations
  $(i_0 \mid i_1) = (i_5 \mid i_6) = (i_6 \mid i_7)$,
  $(i_2 \mid i_3) = (i_8 \mid i_9)$,
  $(i_4 \mid i_5) = (i_{10} \mid i_{11})$,
  $(i_0 \mid i_5) = (i_5 \mid i_{11})$,
  and
  $(i_1\mid i_2) = (i_3 \mid i_4)(i_0 \mid i_1)$.
\end{ex}

A candidate to be a model of $\overline{\iS}_{u=v}$ is
$\iM_{u = v} = (w, \iota, \Theta)$, where
\begin{itemize}
\item $w = \delta(uv)$;
\item $\iota: J \to \alpha_w + 1$ is given by $\iota(i_0) = 0$, and
  $\iota(i_p) = \alpha_{\delta(x_1 \cdots x_p)}$, for each $p = 1,
  \ldots, t$;
\item $\Theta(i_{p-1}, i_p, (\nu({x_p}), I), 0) =
  (\delta(x_p), I)$, for $p = 1, \ldots, t$.
\end{itemize}

\begin{prop}\label{7.4}
  The tuple $\overline{\iS}_{u = v}$ in~\eqref{S} is a system of
  boundary relations which has $\iM_{u = v}$ as a model. Moreover, if
  $\overline{\iS}_{u = v}$ admits a model in $\kappa$-words, then the
  system of equations $\iS_{u= v}$ has a solution modulo $\DRH$ in
  $\kappa$-words.
\end{prop}
\begin{proof}
  For the first part, we notice that the Properties~\ref{m1}--\ref{m3}
  of the requirements for being a model are given for
  free 
  from the construction. Let $(i,x,j,\xx)$ be a boundary
  relation. Since each equation $x_{(k)}y_{(k)} = z_{(k)}$ of
  $\iS_1$
  is such that the inclusion $c(\delta(y_{(k)})) \subseteq
  \cum{\delta(x_{(k)})}$ holds,
  whenever an edge in the graph $\widehat G$ links two indices $p$ and $q$,
  the elements $\delta(x_p) = \Phi(i_{p-1}, i_p, (\nu({x_p}), I), 0)$
  and $\delta(x_q) = \Phi(i_{q-1}, i_q, (\nu({x_q}), I), 0)$ are
  $\Req$-equivalent
  modulo $\DRH$. Therefore, unless $(i,x,j,\xx)$ is one of the
  relations $(i_0, \tl, i_r, \tr)$ or $(i_r, \tr, i_0, \tl)$, the
  Property~\ref{m4} is trivially satisfied. For those relations, we
  just need to observe that $w(i_0, \dirr{(\tl)}) = \delta(u)$ and $w(i_r, \dirr{(\tr)}) =
  \delta(v)$.
  The last Property~\ref{m5} translates into the verification of
  pseudoidentities modulo $\h$ that are satisfied by the pseudovariety
  $\DRH$ by construction.
  This proves that $\iM_{u = v}$ is a model of $\overline\iS_{u = v}$.

  For the second assertion, we consider a model of
  $\overline\iS_{u=v}$ in $\kappa$-words, say $\iM' = (w', \iota',
  \Theta')$, and we let
  $\varepsilon: \pseudo XS \to \pseudo AS$ be the continuous
  homomorphism that sends the variable $x$ to $\Thetap(i_{p-1}, i_p,
  (\nu(x_p), I), 0)$,
  where $p$ is such that $x_p = x$. Such an $x_p$ exists for every variable since we are
  assuming that all the variables occur in $u' = v'$.
  It is worth to mention that the value
  modulo $\DRH$ that
  we assign to $\varepsilon(x)$ when $x = x_p$ for some $p$ does not
  depend on the chosen $p$.
  By Property~\ref{m2}, all the constraints imposed by $\iS_{u= v}$ are
  satisfied by~$\varepsilon$.
  The following computation shows that $\DRH$ satisfies $\varepsilon(u')
  = \varepsilon(v')$:
  \begin{align*}
    \varepsilon(u') &= \varepsilon(x_1\cdots x_r) =
                      \varepsilon(x_1)\cdots \varepsilon(x_r)
    \\ &= \Thetap(i_{0}, i_1,(\nu({x_1}), I), 0) \cdots
         \Thetap(i_{r-1}, i_r,(\nu({x_r}), I), 0)
    \\ & \kern-5pt\just = {\text{\ref{m1}}} w'(i_0, i_1)\cdots w'(i_{r-1}, i_r) =
         w'(i_0, i_r)
    \\ & \kern-1pt\just ={(*)} w'(i_r, i_t) = w'(i_r, i_{r+1}) \cdots
         w'(i_{t-1}, i_t)
    \\ &\kern-5pt\just ={\text{\ref{m1}}}\Thetap(i_{r},
         i_{r+1},(\nu({x_{r+1}}), I), 0) \cdots
         \Thetap(i_{t-1}, i_t,(\nu({x_t}), I), 0)
    \\ & = \varepsilon(x_{r+1}) \cdots \varepsilon(x_t) =
         \varepsilon(x_{r+1}\cdots x_t) = \varepsilon(v').
  \end{align*}
  The reason for $(*)$ is the fact that the relation $(i_0, \tl, i_r,
  \tr)$ belongs to $\iB$ and the equation $(i_0\mid i_r) = (i_r \mid
  i_t)$ to $\iB_\h$, together with Properties~\ref{m4} and~\ref{m5},
  and with Lemma~\ref{sec:16}.
  For the system $\iS_2$, we
  point out that its only aim is to fix the cumulative content of the
  variables and Property~\ref{m3} ensures that.
  Finally, let $x_p x_m = x_q$ be an equation of~$\iS_1$. Since
  for such an equation, we have a relation $(i_{p-1}, (p,q), i_{q-1},
  (q,p))$ in $\iB$ and an equation $(i_{p-1}\mid i_p) (i_{m-1}\mid i_m) =
  (i_{q-1} \mid i_q)$ in $\iB_\h$, from~\ref{m4} we deduce that
  $\varepsilon(x_p)$ and $\varepsilon(x_q)$ are $\Req$-equivalent in
  $\DRH$ and from~\ref{m5} that
  $\varepsilon(x_p)\varepsilon(x_m) = \varepsilon(x_q)$ is a valid
  pseudoidentity in $\h$.
  In addition, the assumption that $S$ has a content function together with
  Property~\ref{m2} yield that $c(\delta(x)) = c(\varepsilon(x))$. In
  turn, we already observed that $\cum{\delta(x)} =
  \cum{\varepsilon(x)}$.
  Therefore, as by construction of $\iS_{u = v}$
  we know that $c(\delta(x_m)) \subseteq
  \cum{\delta(x_p)}$, we have $\varepsilon(x_p)\varepsilon(x_m)
  \Req \varepsilon(x_q)$ modulo $\DRH$, and from Lemma~\ref{sec:16} we
  obtain that $\DRH$ satisfies $\varepsilon(x_p)\varepsilon(x_m) =
  \varepsilon(x_q)$.
\end{proof}

The following criterion for having complete $\kappa$-reducibility of a
pseudovariety~$\DRH$ follows from Proposition~\ref{6.3} together with
Proposition~\ref{7.4}.

\begin{cor}\label{gen7.5}
  If every system of boundary relations which has a model also has a
  model in $\kappa$-words, then $\DRH$ is a completely
  $\kappa$-reducible pseudovariety.
\end{cor}


\section{Factorization schemes}\label{section6}

A \emph{factorization scheme}
for a pseudoword $w$ is a tuple $\iC =
(J, \iota, M, \Theta)$, where:
\begin{itemize}
\item $J$ is a totally ordered finite set;
\item $\iota: J \to \alpha_w + 1$ is an injective function that preserves
  the order;
\item $M:\{(i,j,\vs) \in J \times J \times (S \times S^I)\} \to \omega
  \setminus \{0\}$ is a partial function;
\item $\Theta: \{(i,j,\vs, \mu) \colon (i,j,\vs) \in \dom(M), \:
  \mu \in M(i,j,\vs) \} \to \pseudo AS \times (\pseudo AS)^I$ is a
  function that sends each tuple $(i,j,\vs,\mu)$ to a pair
  $(\Phi(i,j,\vs,\mu), \Psi(i,j,\vs, \mu))$ and satisfies
  $c(\Psi(i,j,\vs, \mu)) \subseteq \cum{\Phi(i,j,\vs, \mu)}$.
\end{itemize}
Moreover, if $(i,j,\vs) \in \dom(M)$ and $\mu \in M(i,j,\vs)$, then
the following properties should be satisfied:
\begin{enumerate}[label = (FS.\arabic*)]
\item\label{fs1} $\Thetap(i,j,\vs, \mu) =_\DRH w[\iota(i), \iota(j)[$;
\item\label{fs2} if $\vs = (s_1, s_2)$, then
  $\varphi(\Phi(i,j,\vs, \mu)) = s_1$ and $\varphi(\Psi(i,j,\vs, \mu))
  = s_2$.
\end{enumerate}
We say that $\iC$ is a \emph{factorization scheme in $\kappa$-words}
if the coordinates of $\Theta$ take $\kappa$-words as values. It is
easy to check that, given a system of boundary relations $\iS$ and a
model $\iM$ for $\iS$, the pair $(\iS, \iM)$ determines a
factorization scheme for $w$, namely $(J, \iota, M, \Theta)$, which we
denote
by $\iC(\iS, \iM)$. Furthermore, a factorization scheme $\iC$ for a
pseudoword $w$
induces functions $\zeta_{w,\iC}$ and $\chi_{w,\iC}$ as follows
\begin{equation}
\begin{aligned}
  \zeta_{w,\iC}: \{(i,j) \in J \times J \colon i \prec j\} &\to
  2^{S\times S^I} \\ (i,j)  &\mapsto \{\vs \colon
  (i,j,\vs) \in \dom(M)\},
\end{aligned}\label{zeta}
\end{equation}
and
\begin{equation}
  \label{chi}
  \begin{aligned}
    \chi_{w,\iC}: \{(i,j) \in J \times J \colon i \prec j\}  &\to 2^A
    \\ (i,j) &\mapsto \cum{w[\iota(i), \iota(j)[}.
  \end{aligned}
\end{equation}
The reason for using this notation becomes clear with the
following lemma, whose proof we leave to the reader.
\begin{lemma}\label{zetachi}
  Let $\iS = (\iX, J, \zeta, M, \chi, \dirr, \iB, \iB_\h)$ be a system of
  boundary relations, $w$ a pseudoword, and $\iC = (J, \iota, M,
  \Theta)$ a factorization scheme for $w$. We define $\iM = (w, \iota,
  \Theta)$ as a candidate for a model of $\iS$. If $\zeta = \zeta_{w,\iC}$
  and  $\chi = \chi_{w,\iC}$, then the Properties~\ref{m1}--\ref{m3} are
  satisfied.
\end{lemma}

For $k = 1,2$, let $\iC_k = (J_k, \iota_k, M_k, \Theta_k)$ be a
factorization scheme for $w$. We say that $\iC_1$ is a
\emph{refinement}
of $\iC_2$ if the following properties are
satisfied:
\begin{enumerate}[label = (R.\arabic*)]
\item\label{rr1} $\im(\iota_2) \subseteq
  \im(\iota_1)$;
\item\label{rr2} there exists a function
  $$\Lambda:\{(i,j,\vs, \mu) \colon (i,j,\vs) \in \dom(M_2),
  \:\mu \in M_2(i,j,\vs)\} \to \bigcup_{k \ge 1} (S\times S^I)^k
  \times \omega$$ such that, if $\Lambda(i,j,\vs, \mu) = ((\vt_1,
  \ldots, \vt_n), \mu')$, then the following holds:
  \begin{enumerate}[label = (R.2.\arabic*)]
  \item\label{rr21} there are $n+1$ elements $i_0,
    \ldots, i_n$ in $J_1$ such that $i_0 \prec \cdots \prec i_n$,
    $\iota_2(i) = \iota_1(i_0)$, and $\iota_2(j) = \iota_1(i_n)$;
  \item\label{rr22} $(i_{m-1},i_m,\vt_m) \in
    \dom(M_1)$, for $m = 1, \ldots, n$;
  \item\label{rr23} if $\vs = (s_1,s_2)$ and $\vt_m
    = (t_{m,1}, t_{m,2})$ for $m = 1, \ldots, n$, then  the equalities
    $s_1 =
    t_{1,1}t_{1,2} \cdots t_{n-1,1}t_{n-1,2} \cdot t_{n,1}$ and $s_2 =
    t_{n,2}$ hold;
  \item\label{rr24} $\mu' \in M_1(i_{n-1}, i_n,
    \vt_n)$ and $\Psi_2(i,j,\vs,\mu) = \Psi_1(i_{n-1}, i_n, \vt_n,
    \mu')$ modulo $\h$.
  \end{enumerate}
\end{enumerate}
We call the function $\Lambda$ in~\ref{rr2} a
\emph{refining function from
  $\iC_2$ to $\iC_1$}.

\begin{prop}[cf.\ {\cite[Proposition 8.1]{JA}}]\label{8.1}
  Let $\iC_k = (J_k, \iota_k, M_k, \Theta_k)$ ($k = 1,2$) be
  factorization schemes for a given pseudoword $w$. Then, there is a
  factorization scheme $\iC_3 = (J_3, \iota_3, M_3, \Theta_3)$ for $w$
  which is a common refinement of~$\iC_1$ and~$\iC_2$. Moreover,
  if~$\iC_1$ and~$\iC_2$ are both refinement schemes in
  $\kappa$-words, then
  we may choose $\iC_3$ with the same property.
\end{prop}
\begin{proof}
  Let $J_3 = \iota_1(J_1) \cup \iota_2(J_2)$ and $\iota_3: J_3
  \hookrightarrow \alpha_w+1$ be the inclusion of ordinals.
  Starting with $\Theta_3$ defined nowhere, we extend it
  inductively as follows.
  Fix $k, \ell \in \{1,2\}$ with $k \neq \ell$, and
  let $i \prec j$ in $J_k$. Let $p_1, \ldots,
  p_m \in J_\ell$ be the indices that are sent by $\iota_\ell$ to an
  ordinal between $\iota_k(i)$ and $\iota_k(j)$ and suppose that
  $\{\beta_0, \beta_1, \ldots, \beta_n\} = \iota_k(\{i,j\}) \cup
  \iota_\ell(\{p_1, \ldots, p_m\})$ with $\beta_0 < \cdots <
    \beta_n$.
  Then, for $r = 1, \ldots, n$, the relation $\beta_{r-1} \prec
  \beta_r$ holds in $J_3$. We fix $\vs \in
  \zeta_{w,\iC_k}(i,j)$, with $\vs = (s_1, s_2)$.
  For each $r < n$, let
  \begin{align*}
   \vt_r & = (\varphi(\Phi_k(i,j,\vs,0)[\beta_{r-1}, \beta_r{[}), I),
    \\ \mu_r  &= \{\overline\mu \colon \Theta_3(\beta_{r-1}, \beta_r, \vt_r,
    \overline\mu) \text{ is defined}\}+1.
  \end{align*}
  We set
  $$\Theta_3(\beta_{r-1}, \beta_r, \vt_r, \mu_r) = (\Thetap_k(i,j,
  \vs, 0)[\beta_{r-1}, \beta_r{[}, I).$$
  For $r = n$, we take
  $$\vt_n = (\varphi(\Phi_k(i,j,\vs,\mu))[\beta_{n-1}, \beta_n{[},
  s_2).$$
  Then, for each $\mu \in M(i,j,\vs)$,
  we set
  \begin{align*}
  \Theta_3(\beta_{n-1}, \beta_n, \vt_n, \mu') &=
  (\Phi_k(i,j, \vs, \mu)[\beta_{n-1}, \beta_n{[},
  \Psi_k(i,j,\vs, \mu)),
  \\ \Lambda_k(i,j,\vs, \mu) &= ((\vt_1, \ldots, \vt_n), \mu'),
  \end{align*}
  where
  $$\mu' = \{\overline\mu \colon \Theta_3(\beta_{n-1}, \beta_n,
  \vt_n, \overline\mu) \text{ is defined}\}+1.$$

  We repeat this process for all possible choices of $k$, $ \ell$,
  $ i$, and $j$.
  Finally, we set $M_3(\beta,\gamma,\vt) = \{\mu\colon
  \Theta_3(\beta,\gamma,\vt,\mu) \text{ is defined}\}$ whenever
  $\Theta_3(\beta,\gamma, \vt, 0)$ is defined.

  Then, the way the construction was performed guarantees not only that
  $\iC_3$ is a factorization scheme for $w$, but also that it is a
  common refinement of $\iC_1$ and $\iC_2$. Moreover, it follows from
  Lemma~\ref{l:6} that
  if $\iC_1$ and
  $\iC_2$ are both factorization schemes in $\kappa$-words, then so
  is~$\iC_3$.
\end{proof}

If $\iC_1 = (J_1, \iota_1, M_1, \Theta_1)$ is a factorization scheme
for $w$,
then it induces a set of factorizations for~$w$. However, it might be
useful to consider the set of factorizations that we obtain by
multiplying some of the adjacent factors. To this end, we define
what is a
\emph{candidate for a refining function to $\iC_1$ with
  respect to $J_2$}: given a totally ordered finite set $J_2$ and an
order preserving injective function $\iota_2: J_2 \to \alpha_w+1$ such
that $\im(\iota_2) \subseteq \im (\iota_1)$, it consists of a partial
function
$$\Lambda:\{(i,j,\vs,\mu) \in J_2 \times J_2 \times (S \times
S^I) \times \omega \colon i \prec j\} \to
\bigcup_{k\ge 1} (S \times S^I)^k \times \omega$$ such that
\begin{enumerate}[label = (C.\arabic*)]
\item\label{c1} $\dom(\Lambda)$ is finite;
\item\label{c2} if $(i,j,\vs,\mu) \in
  \dom(\Lambda)$ and $\mu' \in \mu$, then $(i,j,\vs,\mu') \in
  \dom(\Lambda)$;
\item\label{c3} If $(i,j,\vs,\mu) \in
  \dom(\Lambda)$ and $\Lambda(i,j,\vs, \mu) = ((\vt_1, \ldots, \vt_n),
  \mu')$, then
  \begin{enumerate}[label = (C.3.\arabic*)]
  \item\label{c31} there exist $n+1$ elements $i_0, \ldots, i_n \in J_1$, such that $i_0 \prec \cdots \prec
    i_n$, $\iota_2(i) = \iota_1(i_0)$, and $\iota_2(j) = \iota_1(i_n)$;
  \item\label{c32} if $\vs = (s_1,s_2)$ and
    $\vt_m  = (t_{m,1}, t_{m,2})$ for $m = 1, \ldots, n$, then  the
    equalities $s_1 =
    t_{1,1}t_{1,2} \cdots t_{n-1,1}t_{n-1,2} \cdot t_{n,1}$ and $s_2 =
    t_{n,2}$ hold;
  \item\label{c33} for $m = 1, \ldots, n$,
    $(i_{m-1}, i_m, \vt_m) \in \dom(M_1)$ and $\mu' \in M_1(i_{n-1},
    i_n, \vt_n)$.
  \end{enumerate}
\end{enumerate}
Given a candidate $\Lambda$ for a refining function to $\iC_1$ with
respect to $J_2$, we define the tuple $\iC_2 = (J_2, \iota_2, M_2, \Theta_2)$ as
follows:
\begin{itemize}
\item we let $\dom(M_2) = \{(i,j,\vs) \colon \exists \mu \in \omega
  \mid (i,j,\vs,\mu) \in \dom(\Lambda)\}$;
\item if $(i,j,\vs) \in \dom(M_2)$, then we let $M_2(i,j,\vs) = \{\mu
  \colon (i,j,\vs,\mu) \in \dom(\Lambda)\}$;
\item let $(i,j,\vs) \in \dom(M_2)$ and $\mu \in M(i,j,\vs)$. If
  $\Lambda(i,j,\vs,\mu) =  ((\vt_1, \ldots, \vt_n), \mu')$ and $i_0
  \prec \cdots \prec i_n$ in $J_1$ are such that $\iota_2(i) =
  \iota_1(i_0)$ and $\iota_2(j) = \iota_1(i_n)$, then we define
  \begin{align*}
    & \Phi_2(i,j,\vs,\mu) = \left(\prod_{m=1}^{n-1}
      \Thetap_1(i_{m-1}, i_m, \vt_m, 0)\right) \cdot \Phi_1(i_{n-1},
      i_n, \vt_n, \mu'); \\
    & \Psi_2(i,j,\vs,\mu) = \Psi_1(i_{n-1},i_n,\vt_n,\mu').
  \end{align*}
  We put $\Theta_2(i,j,\vs,\mu) = (\Phi_2(i,j,\vs,\mu),
  \Psi_2(i,j,\vs,\mu))$.
\end{itemize}
We say that $\iC_2$ is the
\emph{restriction of $\iC_1$ to $J_2$ with
  respect to $\Lambda$}. The following result justifies this
terminology. It is a routine matter to prove it.
\begin{prop}\label{8.2}
  Let $\iC_1$, $\iC_2$ and $\Lambda$ be as above. Then,
  \begin{enumerate}
  \item\label{item:2} $\iC_2$ is a factorization scheme for $w$;
  \item\label{item:3} $\iC_1$ is a refinement of $\iC_2$;
  \item\label{item:4} $\Lambda$ is a refining function from $\iC_2$ to $\iC_1$.
  \end{enumerate}
  Moreover, if $\iC_1$ is a factorization scheme in $\kappa$-words,
  then so is $\iC_2$.
\end{prop}

We proceed with a few notes describing general situations that
appear repeatedly later.

\begin{rem}\label{m}
  Let $w$ be a pseudoword and $\iC =(J,\iota, M, \Theta)$  a
  factorization scheme for $w$. Suppose that $\iC_1 = (J_1, \iota_1,
  M_1, \Theta_1)$ is a refinement of the factorization scheme $\iC$ and
  let $\Lambda$ be a refining function from $\iC$ to $\iC_1$. Finally,
  suppose that $\iC_1' = (J_1, \iota_1', M_1, \Theta_1')$ is a
  factorization scheme for another pseudoword~$w'$.
  The function $\Lambda$ is
  clearly a candidate for a refining function to $\iC_1'$ with respect
  to $J$. Moreover, if $\iC' = (J, \iota', M', \Theta')$ is the
  restriction of $\iC_1'$ with respect to $\Lambda$, then $M' = M$.
\end{rem}

\begin{notacao}
  \label{not1}
  Suppose that $\iS = (\iX,J,\zeta, M, \chi, \dirr, \iB, \iB_\h)$ is a
  system of boundary relations that has $\iM = (w, \iota, \Theta)$
  as a model.
  Let $\iC_1 = (J_1, \iota_1, M_1, \Theta_1)$ be a refinement of
  $\iC(\iS, \iM)$ and let $\Lambda$ be a refining
  function from $\iC_1$ to $\iC(\iS, \iM)$. Define $\xi = \iota_1^{-1}
  \circ \iota$.
  We denote by $\xi_{\Lambda}(\iB_{\h})$
  the system of
  $\kappa$-equations with variables in $X_{(J_1, \zeta_{\iC_1}, M_1)}$
  (recall~\eqref{var} and~\eqref{zeta}) obtained from $\iB_{\h}$ by
  substituting each variable $(i\mid j)$ by $(\xi(i) \mid \xi(j))$ and
  each variable $\{i \mid j\}_{\vs, \mu}$ by $\{\xi(j)^-\mid
  \xi(j)\}_{\vt_n, \mu'}$, where $\Lambda(i,j,\vs,\mu) = ((\vt_1,
  \ldots, \vt_n), \mu')$.
\end{notacao}
\begin{rem}
  \label{sec:21}
  Using the notation above, the homomorphism
  $\delta_{w, \iC_1}$ (recall~\eqref{eq:48}) is a solution modulo~$\h$ of the system
  $\xi_\Lambda(\iB_\h)$.
\end{rem}
\begin{rem}
  \label{simp}
  Keeping again the notation, suppose that we are given a pseudoword
  $w_1'$ and a factorization scheme $\iC_1'=(J_1, \iota_1', M_1,
  \Theta_1')$ for $w_1'$, such that $\delta_{w_1', \iC_1'}$ is a
  solution modulo $\h$ of $\xi_\Lambda(\iB_\h)$.
  Further assume
  that there exists a factorization scheme of the form $\iC' = (J,
  \iota', M, \Theta')$ for another pseudoword $w'$ such that
  $\zeta_{w', \iC'} = \zeta$ and the following pseudoidentities are
  valid in $\h$, for every $(i\mid j), \{i \mid j\}_{\vs, \mu} \in
  X_{(J, \zeta, M)}$:
  \begin{align*}
   w'(i,j)  & = w_1'(\xi(i), \xi(j));
    \\ \Psi'(i,j,\vs,\mu) & = \Psi_1'(\xi(j)^-, \xi(j), \vt_n, \mu').
  \end{align*}
  Then, the homomorphism $\delta_{w', \iC'}$ is a solution modulo $\h$
  of $\iB_\h$.
\end{rem}
\section{Complete $\kappa$-reducibility of the pseudovarieties~$\DRH$}\label{section7}

Suppose that $\DRH$ is a completely $\kappa$-reducible
pseudovariety and consider a finite system of $\kappa$-equations $\iS =
\{u_i = v_i\}_{i = 1}^n$ with
variables in~$X$ and constraints given by the pair $(\varphi:\pseudo
AS \to S,\ \nu:X \to S)$. Let~$\delta:\pseudo XS \to
\pseudo AS$ be a solution modulo~$\h$ of~$\iS$.
For a new variable $x_0 \notin X$, we consider a new finite system
of $\kappa$-equations given by $\iS' = \{x_0 u_i = x_0 v_i\}_{i =
  1}^n$ and, writing $A = \{a_1, \ldots, a_k\}$, we set the
constraints on $X \cup \{x_0\}$ to be given by the pair $(\varphi,
\nu')$, where $\nu'|_X = \nu$ and $\nu'(x_0) = \varphi((a_1\cdots a_k)^\omega)$.
By Corollary~\ref{c:2}, the continuous homomorphism $\delta'$ defined by
\begin{align*}
  \delta' : \pseudo{X \uplus \{x_0\}} S& \to \pseudo AS
  \\ x &\mapsto
  \delta(x), \text{ if } x \in X \\ x_0 &\mapsto (a_1 \cdots a_k)^\omega
\end{align*}
is a solution modulo $\DRH$ of $\iS'$. Since we are assuming
that~$\DRH$ is completely $\kappa$-reducible, there exists a solution
in $\kappa$-words
modulo~$\DRH$ of~$\iS'$. Of course, any solution
modulo $\DRH$ of~$\iS'$ 
provides a solution modulo~$\h$ of~$\iS$,
by restriction to~$\pseudo XS$.
Hence, we proved the following.
\begin{prop}\label{p:7}
  If $\DRH$ is a completely $\kappa$-reducible pseudovariety,
  then $\h$ is completely $\kappa$-reducible as well.
\end{prop}
It is known that neither any proper non-locally finite
subpseudovariety of~$\Ab$~\cite{MR2364777} nor the pseudovarieties
$\G$~\cite{MR1485465} and
$\G_p$~\cite{MR0179239}  are completely
$\kappa$-reducible.
Hence, we have the following.
\begin{cor}\label{c:1}
  Let $\h$ be either a proper non-locally finite subpseudovariety
  of~$\Ab$, or one of the pseudovarieties~$\G$
  and~$\G_p$. Then,~$\DRH$ is not completely $\kappa$-reducible.
\end{cor}
In fact, it may be proved that both $\DRH$, for
$\h \subsetneqq \Ab$ non-locally finite, and ${\sf DRG}_p$ are not even
$\kappa$-reducible~\cite{phd}, meaning that they are not $\kappa$-reducible with
respect to the class of systems of equations that may be
obtained from finite graphs (see \cite{steinbergJA} for details).

Our next goal is to prove that $\h$ being completely
$\kappa$-reducible also suffices
for so being~$\DRH$. With that in
mind, throughout this section we fix a pseudovariety of groups~$\h$
that is completely $\kappa$-reducible.
In view of Corollary~\ref{gen7.5}, we should prove the following.

\begin{theorem}\label{main}
  Let $\iS$ be a system of boundary relations that has a model. Then,
  $\iS$ has a model in $\kappa$-words.
\end{theorem}

We fix the pair
$(\iS, \iM)$, where
\begin{equation}
  \label{S,M}
  \begin{aligned}
    \iS &= (\iX, J, \zeta, M, \chi, \dirr, \iB, \iB_\h) \text{ is a system
      of boundary relations},
    \\ \iM &= (w, \iota, \Theta) \text{ is a model of } \iS,
  \end{aligned}
\end{equation}
and we define the parameter
\begin{equation}
  [\iS, \iM] = (\alpha, n),\label{pind}
\end{equation}
where $\alpha $ is the largest ordinal of the form $\iota(c)$ such
that there exists a box $(i,x)$ with $\dirr(x) = c$ if $\iB \neq
\emptyset$, and is $0$ otherwise, and $n$ is the number of boxes
$(i,x)$ such that $\iota(\dirr(x)) = \alpha$.
We denote by $r$ the index $\iota^{-1}(\alpha)$.
In order to prove Theorem~\ref{main}, we argue by transfinite
induction on the parameter
$[\iS, \iM]$, where the pairs $(\alpha, n)$ are ordered 
lexicographically. The
induction step amounts to associating to each pair $(\iS, \iM)$ a new
pair $(\iS_1, \iM_1)$ such that the following properties are satisfied:
\begin{enumerate}[label = (P.\arabic*)]
\item\label{p1} $[\iS_1, \iM_1] < [\iS, \iM]$;
\item\label{p2} if $\iS_1$ has a model in
  $\kappa$-words, then $\iS$ also has a model in $\kappa$-words.
\end{enumerate}

Depending on the set of boundary relations $\iB$, we
consider the following cases:
\begin{description}
\item[Case 1\namedlabel{caso1}{1}] There is a box $(i,x)$ in $\iB$
  such that $i = r =\dirr(x)$.
\item[Case 2\namedlabel{caso2}{2}] There is a boundary relation $(i,x,i,
  \xx)$ such that $\dirr(x) = r = \dirr(\xx)$.
\item[Case 3\namedlabel{caso4}{3}] There is a boundary relation
  $(i,x,j,\xx)$ such that $i < j$, $\dirr(x) = r = 
  \dirr(\xx)$, and the inclusion $c(w(i,j)) \subsetneqq
    c(w(i,\dirr(x)))$  holds.
\item[Case 4\namedlabel{caso5}{4}] There is a boundary relation
  $(i,x,j,\xx)$ such that $\dirr(x) < \dirr(\xx) = r$.
\item[Case 5\namedlabel{caso6}{5}] There is a boundary relation
  $(i,x,j,\xx)$ such that $i < j$,   $\dirr(x) = r = \dirr(\xx)$, and
  $c(w(i,j)) = c(w(i,\dirr(x)))$.
\end{description}
In each case, we assume that all the preceding cases do not apply.
In~\cite[Section~9]{JA}, where the analogous result for the
pseudovariety $\R$ is proved, the cases that are considered are
similar. However, the difference in definition of the induction
parameter~\eqref{pind} justifies the fact of needing to deal with one
less case in the present work.

\subsection{Induction basis}

If the induction parameter $[\iS, \iM]$ is $(0,0)$, then either
 $\iB = \emptyset$ or all the boundary relations of
$\iB$ are of the form $(\min(J),x,\min(J),\xx)$ with $\dirr(x) = \min(J)
= \dirr(\xx)$. In any case,
Property~\ref{m4} for a model of $\iS$ becomes trivial. Hence,
having a model in $\kappa$-words amounts to
having, for each $(i,j,\vs) \in \dom(M)$ and each $\mu \in M(i,j,\vs)$,
a pair of $\kappa$-words $(\Phi(i,j,\vs, \mu), \Psi(i,j,\vs, \mu))$
such that the Properties~\ref{m1}--\ref{m3} and~\ref{m5} are
satisfied. Note that the Property~\ref{m1} means that we should have
$$\Phi(i,j,\vs_1, \mu_1) \Psi(i,j,\vs_1, \mu_1)
=_\DRH  \Phi(i,j,\vs_2, \mu_2) \Psi(i,j,\vs_2, \mu_2)$$
for all $(i,j,\vs_k) \in \dom(M)$ and $\mu_k \in M(i,j,\vs_k)$, $k =
1,2$. We formalize that in the following proposition.

\begin{prop}\label{baseind}
  Suppose that $\h$ is a completely $\kappa$-reducible pseudovariety of groups.
  Let $\iS_1 = \{x_{i,1}y_{i,1} = \cdots = x_{i,n_i}y_{i, n_i}\}_{i = 1}^N$ and
  let $\iS_2$ be a finite system of $\kappa$-equations (possibly
  with parameters in $P$).
  Let $X$ be the set of variables occurring in $\iS_1$ and $\iS_2$
  and suppose that the constraints for the variables are given by the
  pair $(\varphi,\nu)$.
  Let $\delta:
  \pseudo {X \cup P}S \to
  (\pseudo AS)^I$ be a solution modulo $\DRH$ of $\iS_1$ which is also
  a solution modulo $\h$ of $\iS_2$ and such that, for $i = 1, \ldots,
  N$ and $p = 1, \ldots, n_i$, $c(\delta(y_{i,p})) \subseteq
  \cum{\delta(x_{i,p})}$. Then, there exists a continuous homomorphism
  $\varepsilon: \pseudo {X \cup P}S \to (\pseudo AS)^I$ such that
  \begin{enumerate}
  \item\label{a} $\varepsilon(X) \subseteq (\pseudok AS)^I$;
  \item\label{b} $\varepsilon$ is a solution modulo $\DRH$ of $\iS_1$;
  \item\label{c} $\varepsilon$ is a solution modulo $\h$ of $\iS_2$;
  \item\label{e} $\cum{\varepsilon(x)} = \cum{\delta(x)}$, for all the
    variables $x \in X$.
  \end{enumerate}
\end{prop}
\begin{proof}
  We argue by induction on $m = \max\{\card{c(\delta(x_{i,p}))}
  \colon i = 1, \ldots, N; \: p = 1, \ldots, n_i \}$. Note that, if
  $\delta(x_{i,1}) = I$, then we may discard the
  equations $x_{i,1}y_{i,1} = \cdots = x_{i,n_i}y_{i, n_i}$. Hence,
  when $m = 0$, the result amounts to proving the existence
  of~$\varepsilon$ satisfying~\ref{a},~\ref{c} and~\ref{e}.
  But that comes for free from the fact that $\h$ is completely
  $\kappa$-reducible, together with Lemma~\ref{cumH}.
  
  Now, assume that $m \ge 1$ and suppose that $\delta(x_{i,p}) \neq
  I$, for all $i,p$.
  For each variable~$x$ and each $k \ge 1$ such that
  $\lbf[k]{\delta(x)}$ is defined we write
  \begin{align*}
    \lbf[k]{\delta(x)} &= \delta(x)_ka_{x,k},
    \\ \delta(x) &= \lbf[1]{\delta(x)} \cdots \lbf[k]{\delta(x)} \delta(x)_k'.
  \end{align*}
  Since $X$, $A$ and $S$ are finite, there exist $1 \le k < \ell$ such
  that, for all $x \in X$ with $\cum{\delta(x)} \neq \emptyset$,
  the following equalities hold:
  \begin{align*}
    \cum{\delta(x)} & = c(\lbf[k+1]{\delta(x)});
    \\  \varphi(\lbf[1]{\delta(x)} \cdots \lbf[k]{\delta(x)})&=
         \varphi(\lbf[1]{\delta(x)} \cdots \lbf[\ell]{\delta(x)}).
  \end{align*}
  In particular, the latter equality yields
  \begin{equation}
    \label{eq:50}
    \varphi(\delta(x)) = \varphi(\lbf[1]{\delta(x)} \cdots
    \lbf[k]{\delta(x)})\varphi(\lbf[k+1]{\delta(x)} \cdots
    \lbf[\ell]{\delta(x)})^\omega \varphi(\delta(x)_k').
  \end{equation}
  For $i = 1, \ldots, N$, set 
  $$\ell_i =\begin{cases}
    \ell, &\text{ if $\cum{\delta(x_{i,1})} \neq \emptyset$},
    \\ \leng{\delta(x_{i,1})}, &\text{ otherwise}.
  \end{cases}
  $$
  We consider a new set of variables $X'$ given by
  \begin{align*}
    X' = X & \uplus \{x_{i,p;j}\colon i = 1, \ldots, N;\: p = 1, \ldots,
    n_i;\: j = 1, \ldots, \ell_i\}
    \\ & \uplus \{x_{i,p}'\colon i = 1, \ldots, N;\: p = 1, \ldots, n_i;\:
         \cum{\delta(x_{i,p})} \neq \emptyset\},
  \end{align*}
  where the variables $x_{i_1,p;j}$ and $x_{i_2,q;j}$, and the
  variables
  $x_{i_1,p}'$ and $x_{i_2,p}'$ (if defined) are the same,
  whenever the variables $x_{i_1,p}$ and $x_{i_2,q}$ are also the
  same.
  We also consider the following systems of equations with variables in $X'$:
  \begin{itemize}
  \item $\iS_1' = \{x_{i,1;j} = \cdots = x_{i,n_i;j}\colon i = 1,
    \ldots, N; \: j = 1, \ldots,
    \ell_i\}$;
  \item $\iS_2'$ is the system of equations obtained from $\iS_2$ by
    substituting each one of the variables $x_{i,p}$ by the product
    $P_{i,p}$ given by
    \begin{align*}
      P_{i,p} =
      \begin{cases}
        x_{i,p;1}a_{x_{i,p},1}\cdots x_{i,p;\ell} x_{i,p}',\quad\text{ if
          $\cum{\delta(x_{i,p})}\neq \emptyset$},
        \\ x_{i,p;1}a_{x_{i,p},1}\cdots x_{i,p;\ell_i}, \quad\text{ otherwise};
      \end{cases}
    \end{align*}
  \item $\iS_2'' = \{x_{i,1}'z_{i,1} = \cdots = x_{i,n_i}'z_{i,n_i}
    \colon i = 1, \ldots, N;\: \cum{\delta(x_i,1)} \neq \emptyset\}$, where we take
    $$z_{i,p} =
    \begin{cases}
      P_{j,q}, \quad\text{ if $y_{i,p} = x_{j,q}$ for some $j = 1,\ldots,
        N$; $q = 1, \ldots, n_j$},
      \\ y_{i,p}, \quad\text{ otherwise}.
    \end{cases}
    $$
  \end{itemize}
  In the systems $\iS_2'$ and $\iS_2''$ the letters in $A$ work as
  parameters evaluated to themselves, so that the system of equations
  $\iS_2'\cup \iS_2''$ has parameters in $P' = P \cup A$.
  We let the constrains for the variables be given by the pair
  $(\varphi, \nu')$, where the map $\nu'$ is given by
  \begin{equation}
    \begin{aligned}
      \nu': X'& \to S
      \\ x &\mapsto \nu(x), \quad\text{if $x \in X$};
      \\ x_{i,p;j} &\mapsto \varphi(\delta(x_{i,p})_j), \quad\text{if
        $x_{i,p;j} \in X' \setminus X$};
      \\ x_{i,p}' &\mapsto \varphi(\delta(x_{i,p})_k'), \quad\text{if
        $x_{i,p}' \in X' \setminus X$};
    \end{aligned}\label{eq:64}
  \end{equation}
  
  Let $\delta': \pseudo {X' \cup P'} S \to \pseudo AS$ be the
  continuous homomorphism defined by
  \begin{align*}
    \delta'(y) &= \delta(y), \quad\text{ if } y \in X \cup P;
    \\ \delta'(x_{i,p;j}) &= \delta(x_{i,p})_j, \quad\text{ if $i = 1,
         \ldots, N; \: p = 1, \ldots, n_i; \: j = 1, \ldots, \ell_i$};
    \\ \delta'(x_{i,p}') &= \delta(x_{i,p})_k', \quad\text{ if $i = 1,
         \ldots, N; \: p = 1, \ldots, n_i$; $\cum{\delta(x_{i,p})}
         \neq \emptyset$ };
    \\  \delta'(a) &= a, \quad\text{ if }a \in A.
  \end{align*}
  Then, $\delta'$ is a solution modulo $\DRH$ of $\iS_1'$
  which is also a solution modulo $\h$ of $\iS_2' \cup \iS_2''$.
  Since we decreased the induction parameter and the pair $(\iS_1',
  \iS_2'\cup \iS_2'')$ satisfies the hypothesis of the proposition,
  we may invoke the induction hypothesis to derive the  existence of a
  solution in $\kappa$-words 
  modulo
  $\DRH$ of $\iS_1'$, and modulo $\h$ of $\iS_2' \cup
  \iS_2''$, satisfying condition~\ref{e}.

  Now, we define the continuous homomorphism $\varepsilon: \pseudo {X \cup P}S \to
  \pseudo AS$ by:
  \begin{align*}
    \varepsilon(x_{i,p})
    & =
      \begin{cases}
        \varepsilon'(x_{i,p;1}a_{x_{i,p},1}\cdots
        x_{i,p;k}a_{x_{i,p},k})
        \\ \quad\cdot \varepsilon'(x_{i,p;k+1}a_{x_{i,p},k+1}\cdots
        x_{i,p;\ell}a_{x_{i,p},\ell})^\omega
        \varepsilon'(x_{i,p}'), \quad\text{ if $\cum{\delta(x_{i,p})} \neq
          \emptyset$};
        \\ \varepsilon'(P_{i,p}), \quad\text{ if $\cum{\delta(x_{i,p})} = \emptyset$};
      \end{cases}
    \\ \varepsilon(x) & = \varepsilon'(x), \quad\text{ otherwise}.
  \end{align*}
  Clearly, $\varepsilon(X)\subseteq \pseudok AS$.
  Moreover, since we are assuming that $S$ has a content function, it
  follows from $\varphi \circ \varepsilon' = \varphi 
  \circ \delta'$ that $ \cum{\varepsilon(x_{i,p})}=
  \cum{\delta(x_{i,p})}$, for all $i,p$.
  For the other variables $x \in X$, the condition~\ref{e} for
  $\varepsilon$ follows from the same condition for~$\varepsilon'$.

  Let us verify that $\varepsilon$ is a solution modulo $\DRH$ of
  $\iS_1$ and a solution modulo $\h$ of $\iS_2$.
  Since $\varepsilon'$ is a solution modulo $\DRH$ of $\iS_1'$, for
  every pair of variables $x_{i,p}, x_{i,q}$, $\DRH$ satisfies
  $\varepsilon'(x_{i,p;j}) = \varepsilon(x_{i,q;j})$, for $j = 1,
  \ldots, \ell_i$.
  Further, since $\delta$ is a solution modulo $\DRH$ of $\iS_1$ we
  also have $a_{x_{i,p;j}} = a_{x_{i,q;j}}$.
  Thus, we get
  \begin{align*}
    \varepsilon(x_{i,p})
    & =
      \begin{cases}
        \varepsilon'(x_{i,p;1}a_{x_{i,p},1}\cdots
        x_{i,p;k}a_{x_{i,p},k})
        \\ \quad\cdot \varepsilon'(x_{i,p;k+1}a_{x_{i,p},k+1}\cdots
        x_{i,p;\ell}a_{x_{i,p},\ell})^\omega
        \varepsilon'(x_{i,p}'), \quad\text{ if $\cum{\delta(x_{i,p})} \neq
          \emptyset$};
        \\ \varepsilon'(x_{i,p;1}a_{x_{i,p},1}\cdots
        x_{i,p;\ell_i}a_{x_{i,p},\ell_i}), \quad\text{ if $\cum{\delta(x_{i,p})} = \emptyset$};
      \end{cases}
    \\ & =
         \begin{cases}
           \varepsilon'(x_{i,q;1}a_{x_{i,q},1}\cdots
           x_{i,q;k}a_{x_{i,q},k})
           \\ \quad\cdot \varepsilon'(x_{i,q;k+1}a_{x_{i,q},k+1}\cdots
           x_{i,q;\ell}a_{x_{i,q},\ell})^\omega
           \varepsilon'(x_{i,p}'), \quad\text{ if $\cum{\delta(x_{i,p})} \neq
             \emptyset$};
           \\ \varepsilon'(x_{i,q;1}a_{x_{i,q},1}\cdots
           x_{i,q;\ell_i}a_{x_{i,q},\ell_i}), \quad\text{ if
             $\cum{\delta(x_{i,p})} = \emptyset$}.
         \end{cases}
  \end{align*}
  In the second situation, when $\cum{\delta(x_{i,p})} =
  \emptyset$, since $c(\delta(y_{i,p}))\subseteq
  \cum{\delta(x_{i,p})}$, it follows that $\DRH$ satisfies
  $\varepsilon(x_{i,p}y_{i,p}) = \varepsilon(x_{i,p})=
  \varepsilon(x_{i,q})=\varepsilon(x_{i,q}y_{i,q}) $.
  Otherwise, if $\cum{\delta(x_{i,p})} \neq \emptyset$, the above
  equalities imply the relation $\varepsilon(x_{i,p}y_{i,p}) \Req
  \varepsilon(x_{i,q}y_{i,q})$ modulo $\DRH$.
  Also, since $\varepsilon'$ is a solution modulo
  $\h$ of~$\iS_2''$,  we may use Lemma~\ref{sec:16} to conclude that $\DRH$
  satisfies $\varepsilon(x_{i,p}y_{i,p}) =
  \varepsilon(x_{i,q}y_{i,q})$.
  Thus, the homomorphism $\varepsilon$ is a solution modulo $\DRH$ of
  $\iS_1$.
  On the other hand, the pseudovariety $\h$ satisfies
  $\varepsilon(P_{i,p}) = \varepsilon(x_{i,p})$. By definition of
  $\iS_2'$ it follows that $\varepsilon$ is a solution modulo $\h$ of
  $\iS_2$.
  Finally, due to~\eqref{eq:50} and~\eqref{eq:64}, the constraints for the
  variables of~$X$ are satisfied by~$\varepsilon$.
\end{proof}

\subsection{Factorization of a pair $(\iS,
  \iM)$}\label{sec:18}

Instead of repeating the same argument several times, we use this
subsection to describe a general construction that is performed later in some of
the considered cases.

Let $\iE$ be a subset of $\iB$ such that, if $(i,x,j,\xx) \in \iE$, then
$(j,\xx, i,x) \notin \iE$.
Suppose that we are given a set of pairs of ordinals $\Delta =
\{(\beta_e, \gamma_e)\}_{e \in \iE}$ such that, for each boundary relation $e = (i_e,
x_e, j_e, \xx_e) \in \iE$, the following properties are satisfied:
\begin{enumerate}[label = (F.\arabic*)]
\item\label{f1} $\iota(i_e) < \beta_e <
  \iota(\dirr(x_e))$ and $\iota(j_e) < \gamma_e < \iota(\dirr(\xx_e))$;
\item\label{f2} $w[\iota(i_e),
  \beta_e{[} =_\DRH w[\iota(j_e), \gamma_e{[}$.
\end{enumerate}
We say that the
\emph{factorization of $(\iS, \iM)$ with respect to
  $(\iE, \Delta)$} is the pair $(\iS_0, \iM_0)$, where
$$\iS_0 = (\iX_0, J_0, \zeta_0, M_0, \chi_0, \dirr_0, \iB_0,(\iB_\h)_0)
\text{ and }\iM_0 = (w_0, \iota_0, \Theta_0),$$
are defined as follows:
\begin{itemize}
\item the set of variables $\iX_0$ contains all the variables from $\iX$
  and a pair of new variables $y_e$, $\y_e$ for each relation $e \in
  \iE$;
\item we take $w_0 = w$;
\item  we let $J_0$, $\iota_0$, $M_0$ and $\Theta_0$ be given by the
  factorization scheme $\iC_0 = (J_0, \iota_0, M_0, \Theta_0)$, which
  is chosen to be a common refinement of the factorization schemes
  $\iC(\iS, \iM)$ and $(\{\beta_e, \gamma_e\}_{e \in \iE}, \{\beta_e, \gamma_e\}_{e
      \in \iE} \hookrightarrow \alpha_w + 1, \emptyset, \emptyset)$
  for $w$. We denote by $\ell_e$ and $k_e$ the indices
  $\iota_0^{-1}(\beta_e)$ and $\iota_0^{-1}(\gamma_e)$ in $J_0$,
  respectively, by $\xi$ the composite function $\iota_0^{-1} \circ
  \iota$, and we let
  $$\Lambda: \{(i,j,\vs, \mu) \colon (i,j,\vs) \in \dom(M), \: \mu \in
  M(i,j, \vs)\} \to \bigcup_{k \ge 0} (S \times S^I)^k \times \omega
  \setminus \{0\}$$
  be a refining function from $\iC(\iS, \iM)$ to $\iC_0$;
\item the maps $\zeta_0$ and $\chi_0$ are, respectively,
  $\zeta_{w_0,\iC_0}$ and $\chi_{w_0,\iC_0}$ (recall~\eqref{zeta} and~\eqref{chi});
\item the $\dirr_0$ function assigns $\xi(\dirr(x))$ to each variable $x
  \in \iX$ and, for each $e \in \iE$, we let $\dirr_0(y_e) = \ell_e$ and
  $\dirr_0(\y_e) = k_e$;
\item the set of boundary relations $\iB_0$ is obtained by putting the boundary
  relation $(\xi(i), x, \xi(j), \xx)$ whenever $(i,x,j,\xx)$ neither
  belongs to $\iE$ nor is the dual of a boundary relation of
  $\iE$, and the boundary relations $(\xi(i_e), y_e, 
  \xi(j_e), \y_e)$, $(\ell_e, x_e, k_e,\xx_e)$ and their duals for
  each $e \in \iE$;
\item the set $(\iB_\h)_0$ contains $\xi_\Lambda(\iB_\h)$ as well as the equation
  $(\xi(i_e)\mid \ell_e) = (\xi(j_e)\mid k_e)$, for each $e
  \in \iE$.
\end{itemize}

The way we construct $\iB_0$ is illustrated in Fig.~\ref{fig:19}.
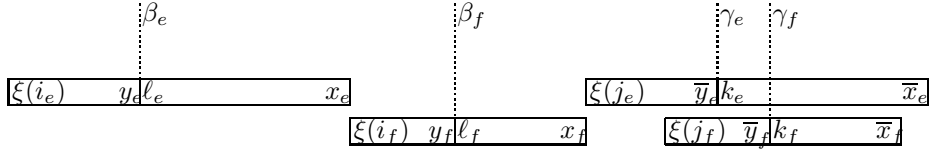
\begin{figure}[htpb]
  \centering
  \unitlength=0.0035mm \brickbackup=-600\brickheight=1000
  \begin{picture}(35000,6000)(0,-2000)
    \small
    \brick(0,1000,13000,100,350)(\xi(i_e),x_e)
    \brick(22000,1000,13000,100,350)(\xi(j_e), \overline {x}_e)
    \brick(13000,-500,9000,100,400)(\xi(i_f),x_f)
    \brick(25000,-500,9000,100,350)(\xi(j_f), \overline {x}_f)
    \multiput(5000,1000)(0,1000)4{\dvert(0,0)}\llvert(5000,1000)\put(5100,4200){$\beta_e$}
    \multiput(27000,1000)(0,1000)4{\dvert(0,0)}\llvert(27000,1000)\put(27100,4200){$\gamma_e$}
    \multiput(17000,0)(0,1000)5{\dvert(0,0)}\llvert(17000,-500)\put(17100,4200){$\beta_f$}
    \multiput(29000,0)(0,1000)5{\dvert(0,0)}\llvert(29000,-500)\put(29100,4200){$\gamma_f$}
    \put(4200,1240){$y_e$}\put(5100,1240){$\ell_e$}
    \put(26120,1240){$\overline y_e$}\put(27100,1240){$k_e$}
    \put(16000,-260){$y_f$}\put(17100,-260){$\ell_f$}
    \put(28000,-260){$\overline y_f$}\put(29100,-260){$k_f$}
  \end{picture}
  \caption{Factorization of $(\iS, \iM)$, when $\iE = \{e, f\}$.}
  \label{fig:19}
\end{figure}

\begin{prop}\label{p:8}
  The triple $\iM_0$ is a model of $\iS_0$ such that $[\iS_0, \iM_0] =
  [\iS, \iM]$ and the Property~\ref{p2} is satisfied.
\end{prop}
\begin{proof}
  The facts that $\iM_0$ is a model of~$\iS_0$ and  $[\iS_0, \iM_0] =
  [\iS, \iM]$ are easy to derive from the construction.
  
  For Property~\ref{p2}, we suppose that $\iM_0' = (w_0',
  \iota_0', \Theta_0')$ is a model of $\iS_0$ in $\kappa$-words and we
  take $\iM' = (w', \iota', \Theta')$, where $w' = w_0'$, $\iota' =
  \iota_0' \circ \xi$, and $\Theta'$ is given by the factorization scheme
  $\iC' = (J, \iota', M, \Theta')$ corresponding to the restriction of
  $\iC(\iS_0, \iM_0)$ with respect to~$\Lambda$ (cf.\ Remark~\ref{m}). We claim that~$\iM'$ is a model 
  of~$\iS$ (in $\kappa$-words by
  Proposition~\ref{8.2}). Properties~\ref{m1} and~\ref{m2} are a
  consequence of $\iC'$ being a
  factorization scheme for $w'$.
  A simple computation shows that $\cum{w'(i,j)} = \chi(i,j)$, so
  that we have~\ref{m3}.
  Property~\ref{m4} is straightforward for all boundary relations
  except for the relations $(i_e, x_e, j_e, \xx_e)$ and their duals.
  In this case, since $(\xi(i_e), y_e, \xi(j_e),
  \y_e)$ belongs to~$\iB_0$, $(\xi(i_e) \mid \ell_e) =
  (\xi(j_e) \mid k_e)$ belongs to $(\iB_\h)_0$, and $\iM_0'$ is a
  model of $\iS_0$, we invoke Lemma~\ref{sec:16} to conclude that
  $\DRH$ satisfies $w_0'(\xi(i_e), \ell_e) = w_e'(\xi(j_e), k_e)$.
  On the other hand, the relation
  $(\ell_e, x_e, k_e, \xx_e)$ also belongs to $\iB_0$, so that the
  relation $w_0'(\ell_e, \dirr_0(x_e)) \Req w_0'(k_e, \dirr_0(\xx_e))$
  holds modulo $\DRH$. Thus, we obtain $w'(i_e, \dirr(x_e)) \Req
  w'(j_e, \dirr(\xx_e)) $ modulo~$\DRH$.
  Finally, since $\xi_\Lambda(\iB_\h) \subseteq (\iB_\h)_0$, we may
  use Remark~\ref{simp} to conclude that in order to prove Property~\ref{m5} it
  is enough to show that the following identities hold in $\h$:
  \begin{align*}
    w'(i,j) &= w'_0(\xi(i), \xi(j)),
    \quad \text{ for all } i \prec j \text{ in } J;
    \\  \Psi'(i,j,\vs,\mu) &= \Psi_0'(\xi(j)^-, \xi(j), \vt,\mu'),
    \quad\text{ for all } (i,j,\vs, \mu) \in \dom(M) \times M(i,j,\vs)
    \\ & \hspace{5cm}\text{ and } ((\ldots, \vt), \mu') = \Lambda(i,j,\vs).
  \end{align*}
  The first one follows from the definition of $w_0'$
  and $\iota'$, while the second is implied by the fact that $\iC'$ is the
  restriction of $\iC'_0$ with respect to $\Lambda$.
  \end{proof}

\subsection{Case 1}

When we are in Case~\ref{caso1}, we have at least one empty box
$(r,x)$. Since for every pseudoword $w$ we have $w(r, \dirr(x)) =
w(r,r) = I$, we may delete the boundary relations involving empty
boxes.
In this way we obtain a new system of boundary relations $\iS_1$
which has exactly the same models as $\iS$ and so, Property~\ref{p2}
is satisfied. Moreover, the parameter associated to $(\iS_1, \iM)$ is
smaller than the parameter associated to $(\iS, \iM)$ since we removed
some boxes ending at $r$. Therefore, Property~\ref{p1} also holds.

\subsection{Case 2}

In this case, there exists a boundary relation $(i,x,i,\xx)$ 
with $\dirr(x) = r = \dirr(\xx)$. Since such a boundary relation yields a
trivial relation in~\ref{m4}, we may argue as in the previous case and
simply delete $(i,x,i,\xx)$ and its dual from $\iS$ obtaining thus a new pair
$(\iS_1, \iM)$ satisfying~\ref{p1} and~\ref{p2}.
\subsection{Case 3}

This is the case where we assume the existence of a boundary relation
$(i_0,x_0,j_0,\xx_0)$ such that $i_0 < j_0$, $\dirr(x_0) = r =
\dirr(\xx_0)$ and $c(w(i_0,j_0)) \subsetneqq c(w(i_0, \dirr(x_0)))$.

Let $a \in c(w(i_0, r)) \setminus c(w(i_0, j_0))$. Since $i_0 < j_0$,
the letter $a$ also belongs to $w(j_0, r)$. Therefore, by
Corollary~\ref{c:5}, 
there are unique factorizations given by $w(i_0, r) = u_i \: a\: v_i$ and $w(j_0, r) =
u_j \: a\: v_j$ such that $a \notin c(u_i) \cup c(u_j)$ and $\DRH$ satisfies the
equality $u_i = u_j$ and the relation $v_i \Req v_j$. Thus, the
decreasing of the induction parameter in this case is achieved by
discarding the segment $[\iota(i_0)+\alpha_{u_i}, \iota(r){[}$ in the
boundary relation $(i_0, x_0, j_0, \xx_0)$ as it is outlined in
Fig.~\ref{fig:caso4} below.

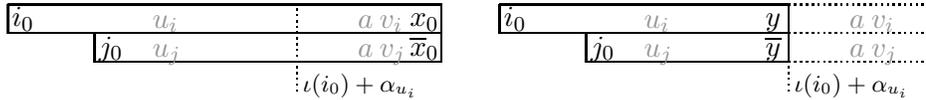
\begin{figure}[htpb]
  \centering
  \unitlength=0.00385mm \brickbackup=-900\brickheight=1000
  \begin{picture}(35000,4000)(0,-2000)
    \brick(0,1000,15000,100,200)(i_0,x_0)
    \brick(3000,0,12000,100,200)(j_0, \xx_0)
    \dvert(10000,0) \dvert(10000,1000)
    \dvert(10000,-1000)
    \put(10200, -1000){\footnotesize{$\iota(i_0)+\alpha_{u_i}$}}
    \put(12100,1200){\textcolor{Gray}{$a\:v_i$}}
    \put(12100,200){\textcolor{Gray}{$a\:v_j$}}
    \put(5000,1200){\textcolor{Gray}{$u_i$}}
    \put(5000,200){\textcolor{Gray}{$u_j$}}
    \brick(17000,1000,10000,100,-100)(i_0,y)
    \brick(20000,0,7000,100,-100)(j_0, \y)
    \dvert(27000,0) \dvert(27000,1000)
    \dvert(27000,-1000)
    \put(27200, -1000){\footnotesize{$\iota(i_0)+\alpha_{u_i}$}}
    \dbrick(27000,1000,25)
    \dvert(32000,1000)
    \dbrick(27000,0,25)
    \dvert(32000,0)
    \put(29100,1200){\textcolor{Gray}{$a\:v_i$}}
    \put(29100,200){\textcolor{Gray}{$a\:v_j$}}
    \put(22000,1200){\textcolor{Gray}{$u_i$}}
    \put(22000,200){\textcolor{Gray}{$u_j$}}
  \end{picture}
  \caption{Discarding the segment $[\iota(i_0)+\alpha_{u_i},
    \iota(r){[}$ in the boundary relation $(i_0, x_0, j_0, \xx_0)$.}
  \label{fig:caso4}
\end{figure}
Let $\iE = \{(i_0, x_0, j_0, \xx_0)\}$ and $\Delta = \{(\iota(i_0) +
\alpha_{u_i},\iota(i_0) + \alpha_{u_i})\}$.
By the above, the pair~$(\iE, \Delta)$ satisfies~\ref{f1}
and~\ref{f2}. Let $(\iS_0, \iM_0)$ be the factorization of~$(\iS,
\iM)$ with respect to~$(\iE, \Delta)$. Then, the pair~$(\iS_0, \iM_0)$
is covered by Case~\ref{caso2} and we may use it in order to decrease
the induction parameter.

Before proceeding with Cases~\ref{caso5} and~\ref{caso6} we perform an
auxiliary step that is useful in both of the remaining
cases.

\subsection{Auxiliary step}\label{aux_step}

We are interested in modifying some of the boundary relations of the
form $(i,x,j,\xx)$ such that $i < j$ and $\dirr(x) = r
= \dirr(\xx)$, so we assume that there exists at least one.
For each $i_0 \in \{i \in J \colon i < r\}$, let
$\iE(\iS, i_0) = \{(i,x,j,\xx) \colon \dirr(x) = r = \dirr(\xx), \:i < j, \:i \le
i_0\}$. Our goal is to prove the existence of a new
pair $(\iS_1, \iM_1)$ that keeps the induction parameter unchanged,
satisfies Property~\ref{p2}, and such that $\iE(\iS_1, i_0) =
\emptyset$. We first construct a pair $(\iS_0, \iM_0)$
satisfying the first two properties and such that $\card
{\iE(\iS_0, i_0)} < \card {\iE(\iS, i_0)}$. Then we argue by induction
to conclude the existence of such a pair $(\iS_1, \iM_1)$.

If $\iE(\iS, i_0) \neq \emptyset$, then we fix a boundary relation
$(k_0, x_0, k_1, \xx_0) \in \iE(\iS, i_0)$. Property~\ref{m4} yields $
w(k_0,k_1) w(k_1, r) = w(k_0,r) \Req w(k_1,r)$ modulo $\DRH$,
which in turn implies that $\DRH$ satisfies $w(k_0,r) \Req w(k_0, k_1)^{\omega}
  w(k_1,r)$.
As we are assuming that the Case~\ref{caso4} does
not hold, the contents of $w(k_0, k_1)$ and $w(k_1, r)$ are the same,
and so, $\DRH$ satisfies $ w(k_0, r) \Req w(k_0,
k_1)^{\omega}$.
Moreover,
since the product $w(k_0, k_1) \cdot w(k_0, k_1)$ is reduced, we may
use Corollary~\ref{c:3} to obtain
$\alpha_{w(k_0, r)} = \alpha_{w(k_0, k_1)^\omega} = \alpha_{w(k_0,
  k_1)} \cdot \omega.$
In particular, setting $\beta_p = \iota(k_0) + \alpha_{w(k_0, k_1)}
\cdot p$ for every $p \ge 0$, the inequality $\beta_p < \alpha = \iota(r)$
holds. On the other hand, as $k_0 \le i_0 < r$, we also have
$\alpha_{w(k_0, i_0)} < \alpha_{w(k_0, r)} = \alpha_{w(k_0, k_1)}
\cdot \omega$ and therefore there exists an integer $n \ge 1$ such
that $\alpha_{w(k_0, i_0)} < \alpha_{w(k_0, k_1)} \cdot n$. We fix
such an $n$ and we take $\iE = \{(k_0, x_0, k_1, \xx_0)\}$ and $\Delta
= \{(\beta_n, \beta_{n+1})\}$.
It is easy to check that the pair $(\iE, \Delta)$ satisfies
both~\ref{f1} and~\ref{f2}.
So, we let
$(\iS_0, \iM_0)$ be the factorization of $(\iS, \iM)$ with respect to
$(\iE, \Delta)$. Intuitively, the transformation performed in the
step $(\iS, \iM) \mapsto (\iS_0, \iM_0)$ is represented in pictures
\ref{fig:aux_step1} (before) and~\ref{fig:aux_step2} (after).
\begin{figure}[htpb]
  \centering
  \unitlength=0.00385mm \brickbackup=-600\brickheight=1000
  \begin{picture}(32000,3000)(500,-1050)
    \brickstart(0,1000,13500,100,400)(k_0,x_0)
    \brickend(31000,1000,2000,100,400)(k_0,x_0)
    \put(17000,2000){\line(1,0){11000}}
    \dbrick(13500,1000,90)
    \dbrick(13500,0,90)
    \put(17000,1000){\line(1,0){11000}}
    \put(17000,0){\line(1,0){11000}}
    \brickstart(4000,0,9500,100,400)(k_1, \xx_0)
    \brickend(31000,0,2000,100,400)(k_1, \xx_0)
    \dvert(4000,1000)
    \dvert(8000,1000)
    \dvert(8000,0) \dvert(8000,-1000)
    \put(8000,-1000){\footnotesize{$\beta_2$}}
    \dvert(12000,1000)
    \dvert(12000,0) \dvert(12000,-1000)
    \put(12000,-1000){\footnotesize{$\beta_3$}}
    \dvert(18000,1000)
    \dvert(18000,0) \dvert(18000,-1000)
    \put(18000,-1000){\footnotesize{$\beta_{n-1}$}}
    \dvert(22000,1000)
    \dvert(22000,0) \dvert(22000,-1000) \put(22000,
    -1000){\footnotesize{$\beta_n$}}
    \dvert(26000,1000)
    \dvert(26000,0) \dvert(26000,-1000) \put(26000,
    -1000){\footnotesize{$\beta_{n+1}$}}
    \dvert(20300,1000)
    \dvert(20300,0) \dvert(20300,-1000)
    \put(20300,-1000){\footnotesize{$\iota(i_0)$}}
    \put(14500,250){$\cdots$}
    \put(14500,1250){$\cdots$}
    \put(29000,250){$\cdots$}
    \put(29000,1250){$\cdots$}
  \end{picture}
  \caption{Original relation $(k_0, x_0, k_1, \xx_0)$ in the
    system of boundary relations $\iS$.}
  \label{fig:aux_step1}
\end{figure}
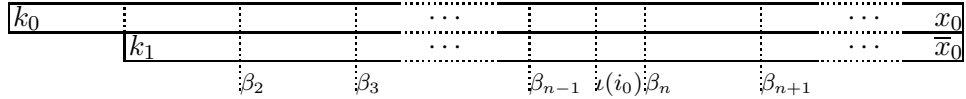
\begin{figure}[htpb]
  \centering
  \unitlength=0.00385mm \brickbackup=-600\brickheight=1000
  \begin{picture}(32000,1500)(500,-1000)
    \brickstart(0,1000,11000,100,400)(k_0,x_0)
    \brickend(31000,1000,2000,100,400)(k_0,x_0)
    \dbrick(11000,1000,100)
    \dbrick(11000,0,100)
    \put(14500,2000){\line(1,0){13500}}
    \put(14500,1000){\line(1,0){13500}}
    \put(14500,0){\line(1,0){13500}}
    \brickstart(4000,0,7000,100,400)(k_1, \xx_0)
    \brickend(31000,0,2000,100,400)(k_1, \xx_0)
    \llvert(22000,1000) \put(21300,1250){$y$}
    \put(22100,1250){{$k_n$}}
    \dvert(22000,0) \dvert(22000,-1000) \put(22000,
    -1000){\footnotesize{$\beta_n$}}
    \dvert(26000,1000) \put(25300,250){$\y$}
    \put(26100,250){{$k_{n+1}$}}
    \llvert(26000,0) \dvert(26000,-1000) \put(26000,
    -1000){\footnotesize{$\beta_{n+1}$}}
    \dvert(20300,1000)
    \dvert(20300,0) \dvert(20300,-1000)
    \put(20300,-1000){\footnotesize{$\iota(i_0)$}}
    \put(12000,250){$\cdots$}
    \put(12000,1250){$\cdots$}
    \put(29000,250){$\cdots$}
    \put(29000,1250){$\cdots$}
  \end{picture}
  \caption{Factorization of the relation $(k_0, x_0, k_1,\xx_0)$ in the
    new system of boundary relations $\iS_0$.}
  \label{fig:aux_step2}
\end{figure}
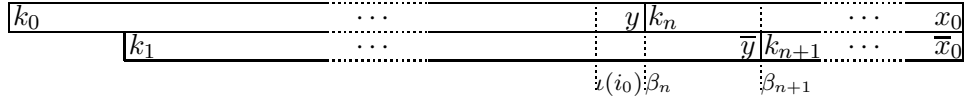

The definition of $(\iS_0, \iM_0)$ yields the following:
\begin{lemma}\label{l:10}
  Let $(\iS_0, \iM_0)$ be the pair defined above. Then the following
  holds:
  \begin{enumerate}
  \item Cases~\ref{caso2} and~\ref{caso4} do not apply to the system
    of boundary relations $\iS_0$;
  \item\label{item:5} the inequality $\card{ \iE (\iS_0, i_0)} < \card{ \iE (\iS,
      i_0)}$ holds.
  \end{enumerate}
\end{lemma}

Recall that, by Proposition~\ref{p:8}, we also have $[\iS_0, \iM_0] = [\iS, \iM]$
and Property~\ref{p2} is satisfied by $(\iS_0, \iM_0)$.
Thus, arguing by induction, we may assume, without loss of generality, that
given a system $\iS$ in Cases~\ref{caso5} or~\ref{caso6} we have
$\iE(\iS, i_0) = \emptyset$, for all $i_0 < r$ in~$J$.

\subsection{Case 4}

In this case we suppose that the Cases~\ref{caso1},~\ref{caso2} and~\ref{caso4} do
not hold and that there is a boundary relation $(i,x,j,\xx)$
such that $\dirr(\xx) < \dirr(x) = r$.
Consider the index $\ell = \min\{\esq(x) \colon \dirr(\xx) < \dirr(x) = r\}$. By the
auxiliary step in Subsection~\ref{aux_step}, we may assume
without loss of generality that all 
boundary relations $(i,x,j,\xx)$ with $\dirr(x) = r = \dirr(\xx)$  
are such that $i,j > \ell$.
Let $x_0 \in \iX$ be such that $\esq(x_0) = \ell$ and $\dirr(\xx_0) <
\dirr(x_0) = r$, and let $\ell^* \in J$ be such that $(\ell, x_0, \ell^*, \xx_0)
\in \iB$. We set $r^* = \dirr(\xx_0)$.
Since Case~\ref{caso1} does not hold, we know that $\ell < r$.
The intuitive idea consists in transferring all the information
comprised in the factor $w(\ell, r)$ to the factor $w(\ell^*, r^*)$
in order to decrease the induction parameter by discarding the
factors $w(r^-, r)$ and $w[\iota(\ell^*) + (\iota(r^-) -
\iota(\ell)), \iota(r^*)[$ intervening in the boundary relation $(\ell, x_0, \ell^*,
\xx_0)$. See Fig.~\ref{fig:caso5}.

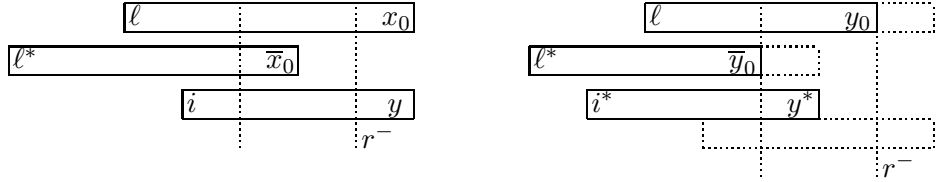
\begin{figure}[htpb]
  \centering
  \unitlength=0.00385mm \brickbackup=-900\brickheight=1000
  \begin{picture}(33000,8000)(0,-3000)
    \brick(6000,0,8000,100,0)(i, y)
    \brick(0, 1500, 10000, 100, 200)(\ell^*, \xx_0)
    \brick(4000, 3000, 10000, 100, 200)(\ell,x_0)
    \dvert(12000, -1000) \put(12200, -1000){$r^-$}
    \dvert(12000, 0)
    \dvert(12000, 1000)
    \dvert(12000, 2000)
    \dvert(12000, 3000)
    \dvert(8000, -1000) 
    \dvert(8000, 0)
    \dvert(8000, 1000)
    \dvert(8000, 2000)
    \dvert(8000, 3000)
    \brick(20000,0,8000,100,200)(i^*, y^*)
    \brick(18000, 1500, 8000, 100, 200)(\ell^*, \y_0)
    \brick(22000, 3000, 8000, 100, 200)(\ell,y_0)
    \dvert(30000,-2000)
    \dvert(30000, -1000) \put(30200, -2000){$r^-$}
    \dvert(30000, 0)
    \dvert(30000, 1000)
    \dvert(30000, 2000)
    \dvert(30000, 3000)
    \dvert(26000, -2000)
    \dvert(26000, -1000) 
    \dvert(26000, 0)
    \dvert(26000, 1000)
    \dvert(26000, 2000)
    \dvert(26000, 3000)
    \dbrick(30000,3000,10) \dvert(32000, 3000)
    \dbrick(26000,1500,10) \dvert(28000, 1500)
    \dbrick(24000,-1000,40) \dvert(24000, -1000) \dvert(32000, -1000)
  \end{picture}
  \caption{Transferring the segment $(\ell, r)$ to the segment $(\ell^*,r^*)$
    and discarding the final segments of the boxes $(\ell, x_0)$ and
    $(\ell^*, \xx_0)$.}
  \label{fig:caso5}
\end{figure}
More formally, we define the set of transport positions by
$$T = \{i \in J \colon \exists \text{ box } (i,x) \text{ such that }
\dirr(x) = r\} \cup \{r^-, r\}.$$
Observe that $\min(T) = \ell$ and $\max(T) = r$. Hence, for $i \in
T$ we may define the index $i^{\circ} = \iota(\ell^*) + (\iota(i) -
\iota(\ell))$.
Some useful properties of $\_^{\circ}$ are stated in the
next lemma.
\begin{lemma}\label{9.3}
  The function $\_^{\circ}: T \to \alpha_w+1$ satisfies the following:
  \begin{enumerate}
  \item \label{9.3.1} it preserves the order and is injective;
  \item \label{9.3.2} for every $i < j$ in $T$, the pseudovariety
    $\DRH$ satisfies the equality $w[i^{\circ}, j^{\circ}{[} = w(i,j)$
    if $j < r$ and the relation $w[i^{\circ}, r^{\circ}{[} \Req
    w(i,r)$;
  \item \label{9.3.3} for every $i \in T$, the inequality
    $i^\circ < \iota(i)$ holds.
  \end{enumerate}
\end{lemma}
\begin{proof}
  We omit the proofs of assertions~\ref{9.3.1}
  and~\ref{9.3.3} since they express properties of ordinal numbers and 
  thus, are entirely analogous to the proofs of the corresponding properties
  in~\cite[Lemma~9.3]{JA}.

  Let us prove~\ref{9.3.2}.
  Since $(\ell, x_0, \ell^*, \xx_0)$ is a boundary relation in $\iB$ and
  $\iM$ is a model of~$\iS$, we have
  $w(\ell, r) = w(\ell, \dirr(x_0)) \Req w(\ell^*, \dirr(\xx_0)) =
  w(\ell^*, r^*)$ modulo~$\DRH$.
  Further,
  the fact that $\ell^\circ =\iota(\ell^*)$ and $r^\circ = \iota(r^*)$,
  implies that $\DRH$ satisfies $w(\ell,r) \Req w[\ell^\circ, r^\circ[$.
  On the other hand, since $j^\circ - i^\circ =\iota(j) - \iota(i)$,
  we may use Corollary~\ref{c:7} twice to first conclude that,
  for $j < r$, $\DRH$
  satisfies $w(\ell, j) = w[\ell^\circ, j^\circ{[}$ and then, that it
  satisfies the desired identity $w(i, j) = w[i^\circ, j^\circ{[}$.
  Similarly, when $j = r$, we get that $\DRH$ satisfies $w[i^\circ,
  j^\circ{[} \Req w(i,r)$.
\end{proof}

Before defining a new pair $(\iS_1, \iM_1)$, we still need to consider
a factorization scheme for the pseudoword $w$, in order to memorize the
information on constraints that we lose when transforming $\iS$
according to Fig.~\ref{fig:caso5}. We let $\iC_0 = (J_0, \iota_0, M_0,\Theta_0)$
be defined as follows:
\begin{itemize}
\item $J_0 = \{i^\circ \colon i \in T\}$;
\item $\iota_0: J_0 \hookrightarrow \alpha_w+1$ is the inclusion of
  ordinals;
\item By Lemma~\ref{9.3}\ref{9.3.2} the pseudowords
   $w(r^-,r)$ and $w[(r^-)^\circ, r^\circ[$ are $\Req$-equivalent
  modulo $\DRH$. Therefore, since Property~\ref{m1} holds for $(\iS,
  \iM)$, given $\vs \in \zeta(r^-,
  r)$ and $\mu \in M(r^-, r, \vs)$ the pseudowords $\Phi(r^-, r, \vs,
  \mu)$ and $w[(r^-)^\circ, r^\circ[$ are $\Req$-equivalent modulo
  $\DRH$ as well. For each such pair $(\vs, \mu)$, we fix a
  pseudoword $v_{\vs, \mu} \in (\pseudo AS)^I$ such that 
  \begin{equation}
    w[(r^-)^\circ, r^\circ{[} =_\DRH \Phi(r^-, r, \vs, \mu) \:
    v_{\vs,\mu}.\label{eq:51}
  \end{equation}
  In particular, it follows that $\Phi(r^-, r, \vs, \mu) \:
  v_{\vs,\mu}$ and $\Phi(r^-, r, \vs, \mu)$ are $\Req$-equivalent
  modulo $\DRH$. Combining Remark~\ref{sec:9} with Lemma
  \ref{sec:13}, we may deduce the inclusion $c(v_{\vs,\mu}) \subseteq \cum{\Phi(r^-, r, \vs,
    \mu)} = \cum{w[(r^-)^\circ, r^\circ{[}}$.
  
  Since $\zeta(r^-, r)$ is a finite set, we may write
  $\zeta(r^-, r) = \{\vs_1, \ldots, \vs_m\}$.
  Let $\vs_p = (s_{p,1}, s_{p,2})$ and denote by $\vt_{p, \mu}$ the
  pair $(s_{p,1}, \varphi(v_{\vs_p,\mu}))$ for each $\vs_p \in
  \zeta(r^-, r )$ and $\mu \in M(r^-, r, \vs_p)$.
  We define~$\Theta_0$ inductively as follows:
  \begin{itemize}
  \item start with $\Theta_0 = \emptyset$;
  \item for each $p \in \{1, \ldots , m\}$ and $\mu \in M(r^-, r,
    \vs_p)$, we set
    \begin{align*}
       \mu_{p, \mu} &= \{\overline \mu\colon \Theta_0((r^-)^\circ,
        r^\circ, \vt_{p, \mu}, \overline \mu) \text{ is defined}\};
      \\  \Theta_0((r^-)^\circ, r^\circ, \vt_{p, \mu}, \mu_{p, \mu})
           & = (\Phi(r^-, r, \vs_p, \mu), v_{\vs_p, \mu}).
    \end{align*}
  \end{itemize}
\item the map $M_0$ is given by $M_0((r^-)^\circ, r^\circ, \vt ) = \{\mu'\colon
  \Theta_0(r^-, r, \vt, \mu') \text{ is defined}\}$, whenever $\vt =
  \vt_{p, \mu}$ for certain $p = 1, \ldots, m$ and $\mu \in M(r^-, r,
  \vs_p)$. Observe that we may have $\vt_{p, \mu} = \vt_{p',\mu'}$
  with $(p, \mu) \neq (p', \mu')$.
\end{itemize}
\begin{lemma}
  The tuple $\iC_0$ just constructed is a factorization scheme for
  $w$.
\end{lemma}
\begin{proof}
  Since $r^- \prec r$ in $J$, Lemma~\ref{9.3}\ref{9.3.1} yields
  $(r^-)^\circ \prec r ^\circ$ in $J_0$. Therefore, the domain
  of $\Theta_0$ is compatible with the definition of factorization
  scheme. Moreover, the definition of $M_0$
  guarantees that the relationship between the domains of $\Theta_0$
  and of $M_0$ is the correct one.
  To prove~\ref{fs1}, let $\vs_p \in \zeta(r^-, r)$ and $\mu \in
  M(r^-, r, \vs_p)$. In $\DRH$, we have
  \begin{align*}
    \Thetap_0((r^-)^\circ, r^\circ, \vt_{p, \mu}, \mu_{p, \mu})
    & \just={\text{def.}} \Phi(r^-, r, \vs_p, \mu) \:v_{\vs_p, \mu}
    \\ & \just={\eqref{eq:51}} w[(r^-)^\circ, r^\circ{[},
  \end{align*}
  as intended.
  On the other hand, recalling that $\vt_{p, \mu} = (s_{p,1}, \varphi(v_{\vs_p, \mu}))$
  and that $\iM$ is a model of $\iS$, Property~\ref{m2} yields
  $$\varphi(\Phi_0((r^-)^\circ, r^\circ, \vt_{p, \mu}, \mu_{p, \mu}))
  = \varphi(\Phi(r^-, r, \vs_p, \mu)) = s_{p,1}$$
  and by construction,
  $$\varphi(\Psi_0((r^-)^\circ, r^\circ, \vt_{p, \mu}, \mu_{p,
    \mu})) = \varphi(v_{\vs_p, \mu}).$$
  This proves~\ref{fs2}.
\end{proof}

We are now ready to proceed with the construction of the new pair
$(\iS_1, \iM_1)$, where
$$\iS_1 = (\iX_1, J_1, \zeta_1, M_1, \chi_1, \dirr_1, \iB_1, (\iB_\h)_1)
\text{ and } \iM = (w_1, \iota_1, \Theta_1).$$
We take as set of variables $\iX_1$ the old set $\iX$ together with a pair
of new variables $y_i$ and~$\y_i$, for each $i \in T \setminus \{r\}$.
The pseudoword $w_1$ is $w$. Let
$\iC_1 = (J_1, \iota_1, M_1, \Theta_1)$ be a common refinement of
$\iC(\iS, \iM)$ and $\iC_0$. The elements $J_1$, $M_1$, $\iota_1$ and
$\Theta_1$ are those given by $\iC_1$. To simplify the notation, we
set $\xi = \iota_1^{-1}\circ \iota$ and $i^\bullet =
\iota_1^{-1}(i^\circ)$. The refining functions from $\iC(\iS, \iM)$
to $\iC_1$ and from $\iC_0$ to $\iC_1$ are given, respectively, by
$\Lambda$ and $\Lambda_0$.
The functions $\zeta_1$ and $\chi_1$ are the ones induced by $\iC_1$,
namely 
$\zeta_1 = \zeta_{w_1,\iC_1}$ and $\chi_1 = \chi_{w_1,\iC_1}$ (recall
\eqref{zeta} and~\eqref{chi}). The $\dirr_1$
function is given by
\begin{align*}
  \dirr_1: \iX_1 &\to  J_1
  \\ x &\mapsto  \xi(\dirr(x)),\quad \text{ if } x \in \iX \text{ and }
                                  \dirr(x) < r;
  \\ x &\mapsto  r^\bullet, \quad \text{ if } x \in \iX \text{ and }
                               \dirr(x) = r;
  \\ y_i &\mapsto  \xi(i), \quad \text{ if } i \in T \setminus \{r\};
  \\ \y_i & \mapsto  i^\bullet, \quad  \text{ if } i \in T \setminus\{r\}.
\end{align*}
We define $\iB_1$ iteratively by:
\begin{enumerate}[label = (\arabic*)]\addtocounter{enumi}{-1}
\item set  $\iB' = \iB \setminus \{(\ell, x_0, \ell^*, \xx_0), (\ell^*, \xx_0, \ell, x_0)\}$;
\item \label{eq1} start with $\iB_1 = \{(\xi(\ell), y_i, \ell^\bullet, \y_i),
  (\ell^\bullet, \y_i, \xi(\ell), y_i) \colon i \in T \setminus \{r\}\}$;
\item\label{eq2} for each variable $x \in \iX$ such that $\dirr(x) = r$ and for each
  boundary relation $(i,x,j, \xx) \in \iB'$, we add to $\iB_1$ two new boundary
  relations as follows:
  \begin{enumerate}
  \item \label{eq2a} if $\dirr(\xx) < r$, then add the
    relations $(i^\bullet, x, \xi(j), \xx)$ and $(\xi(j), \xx,i^\bullet,
    x)$;
  \item  \label{eq2b}if $\dirr(\xx) = r$, then add the
    relations $(i^\bullet, x,j^\bullet, \xx)$ and $(j^\bullet,
    \xx,i^\bullet, x)$;
  \end{enumerate}
\item  \label{eq3} for each variable $x \in \iX$ such that
  $\dirr(x) < r$ and $\dirr(\xx) < r$ and for each boundary relation
  $(i,x,j, \xx) \in \iB'$, we add to $\iB_1$ the boundary relations
  $(\xi(i), x, \xi(j), \xx)$ and $(\xi(j), \xx, \xi(i), x)$.
\end{enumerate}
Finally, in $(\iB_\h)_1$ we include all the equations of the set
$\xi_\Lambda(\iB_\h)$ as well as the following:
\begin{itemize}
\item $(\xi(\ell) \mid \xi(r^-)) = (\ell^\bullet \mid (r^-)^\bullet)$;
\item $(\xi(r^-) \mid \xi(r)) = ((r^-)^\bullet \mid r^\bullet)
  \cdot \{(r^\bullet)^- \mid r^\bullet\}_{\vt\,'_{p,\mu},
    \mu'_{p, \mu}}^{\omega-1}\cdot \{\xi(r)^- \mid \xi(r)\}_{\vs\,'_p, \mu'}$,
  for each $\vs_p \in \zeta(r^-, r)$ and $\mu \in M(r^-, r,
  \vs_p)$. Here, we are writing
  \begin{align*}
     \Lambda(r^-, r, \vs_p, \mu) &=((\ldots, \vs\,'_p), \mu');
    \\  \Lambda_0((r^-)^\circ, r^\circ, \vt_{p, \mu}, \mu_{p, \mu}) &=
  ((\ldots, \vt\,'_{p, \mu}), \mu'_{p, \mu}).
  \end{align*}
\end{itemize}

\begin{prop}\label{9.4}
  The tuple  $\iM_1$ is a model of the system of boundary relations $\iS_1$.
\end{prop}
\begin{proof}
  Properties~\ref{m1}--\ref{m3} are satisfied as a consequence of
  Lemma~\ref{zetachi}.
  It is a routine computation to check~\ref{m4}.
  By Remark~\ref{sec:21}, the homomorphism $\delta_{w, \iC_1} =
  \delta_{w_1,\iC_1}$ is a
  solution modulo $\h$ of $\xi_\Lambda(\iB_\h)$.
  Also, the homomorphism
  $\delta_{w_1, \iC_1}$ is a solution modulo $\h$ of the equation
  $(\xi(\ell) \mid \xi(r^-)) = (\ell^\bullet \mid (r^-)^\bullet)$ as a
  consequence of the fact that $\DRH$ satisfies $w_1(\xi(\ell),
  \xi(r^-)) = w_1(\ell^\bullet, (r^-)^\bullet)$, which
  follows from~\ref{m4}.
  Finally, the equations of the form
  $(\xi(r^-) \mid \xi(r)) = ((r^-)^\bullet \mid r^\bullet) \cdot
  \{(r^\bullet)^- \mid r^\bullet\}_{\vt\,'_{p,\mu}, \mu'_{p,
      \mu}}^{\omega-1}\cdot \{\xi(r)^- \mid \xi(r)\}_{\vs\,'_p, \mu'}$
  are satisfied by $\delta_{w_1, \iC_1}$ modulo $\h$ since the
  following pseudoidentities are valid in $\h$:
  \begin{align*}
    \delta_{w_1, \iC_1}(\xi(r^-) \mid \xi(r)) & = w_1(\xi(r^-), \xi(r)) = w(r^-, r)
    \\ & = \Phi(r^-, r, \vs_p, \mu) \Psi(r^-, r, \vs_p, \mu)
       \quad\text{by Property~\ref{m1} for $(\iS, \iM)$}
    \\ & = w[(r^-)^\circ, r^\circ{[} \: v_{\vs_p, \mu}^{\omega-1} \:\:
         \Psi(r^-, r, \vs_p, \mu)
       \quad\text{by~\eqref{eq:51}}
    \\ & = w_1((r^-)^\bullet, r^\bullet) \: \Psi_0((r^-)^\circ, r^\circ, \vt_{p, \mu},
         \mu_{p, \mu})^{\omega-1} \: \Psi(r^-, r, \vs_p, \mu)
    \\ & =  w_1((r^-)^\bullet, r^\bullet) \:  \quad\quad\quad\quad \text{by~\ref{rr24} for $\Lambda$ and $\Lambda_0$}
    \\ & \quad \cdot \Psi_1((r^\bullet)^-, r^\bullet, \vt\,'_{p, \mu},
         \mu'_{p, \mu})^{\omega-1}\: \Psi_1(\xi(r)^-, \xi(r), \vs\,'_p,
         \mu')
    \\ & =  \delta_{w_1, \iC_1}\left(((r^-)^\bullet\mid r^\bullet)
         \cdot \{(r^\bullet)^- \mid r^\bullet\}_{\vt\,'_{p,\mu}, \mu'_{p,
         \mu}}^{\omega-1}\cdot \{\xi(r)^- \mid \xi(r)\}_{\vs\,'_p,
         \mu'}\right).
  \end{align*}
  With this, we may conclude that $\iM_1$ is a model of $\iS_1$.
\end{proof}
\begin{prop}
  Properties~\ref{p1} and~\ref{p2} are satisfied by the pairs $(\iS,
  \iM)$ and~$(\iS_1, \iM_1)$.
\end{prop}
\begin{proof}
  Property~\ref{p1} is trivial.
  For Property~\ref{p2}, we may let $\iM_1' = (w_1', \iota_1', \Theta_1')$
  be a model of $\iS_1$ in $\kappa$-words and we construct a new
  triple $\iM' = (w', \iota', \Theta')$ as follows. We
  fix a pair $(\vs_q, \mu_0) \in \zeta(r^-, r) \times M(r^-, r,
  \vs_q)$, for a certain $q \in \{1, \ldots, m\}$. We write
  $\Lambda(r^-, r, \vs_q, \mu_0) = ((\ldots, \vs\,'_q), \mu_0')$ and
  $ \Lambda_0((r^-)^\circ, r^\circ, \vt_{q, \mu_0},
  \mu_{q, \mu_0}) = ((\ldots, \vt\,'_{q, \mu_0}), \mu'_{q,
    \mu_0})$. The $\kappa$-word~$w'$ is given by
  \begin{align*}
    w' & =  w_1'[0, \iota_1'(\xi(r^-)){[}
    \\ & \quad \cdot  w_1'((r^-)^\bullet,
         r^\bullet) \Psi'_1((r^\bullet)^-, r^\bullet, \vt\,'_{q, \mu_0},
         \mu'_{q, \mu_0})^{\omega - 1} \Psi'_1(\xi(r)^-, \xi(r), \vs\,'_q,
         \mu_0')
    \\ & \quad \cdot w_1'[\iota'_1(\xi(r)), \alpha_{w'_1}{[}.
  \end{align*}
  Note that Lemma~\ref{l:6} yields that $w'$ is indeed a $\kappa$-word.
  For $i \in J$, we let $\iota'(i)$ be given~by
  $$ \iota'(i) =
  \begin{cases} \iota'_1(\xi(i)), \quad\text{ if } i \le r^-;
    \\ \iota'(r^-) + (\iota'_1(r^\bullet) -
    \iota'_1((r^-)^\bullet)), \quad \text{ if } i = r;
    \\ \iota'(r) + (\iota'_1(\xi(i)) - \iota'_1(\xi(r))), \quad \text{
      if } i > r.
  \end{cases} $$
  Finally, we define $\Theta'$. For $i \prec j \le r^-$ or $r \le
  i \prec j$ in $J$, $\vs \in \zeta(i,j)$ and $\mu \in M(i,j,\vs)$,
  let $\Lambda(i,j,\vs,\mu) = ((\vt_1, \ldots, \vt_n), \mu')$ and
  $\xi(i) = i_0 \prec i_1 \prec \cdots \prec i_n = \xi(j)$. Then, we
  take
  \begin{align*}
       \Phi'(i,j,\vs, \mu)
       &= \left(\prod_{k = 1}^{n-1}
        \Thetap'_1(i_{k-1}, i_k, \vt_k,0)\right) \Phi'_1(\xi(j)^-, \xi(j),
      \vt_n, \mu'),
      \\  \Psi'(i,j,\vs, \mu)
       &=  \Psi'_1(\xi(j)^-, \xi(j), \vt_n, \mu').
    \end{align*}
    On the other hand, when $(i,j) = (r^-, r)$, $\vs_p \in \zeta(r^-, r)$, and $\mu \in
  M(r^-, r, \vs_p)$,  we write $\Lambda(r^-, r, \vs_p, \mu) = ((\vt_1, \ldots, \vt_n),
  \mu')$, $\Lambda_0((r^-)^\circ, r^\circ, \vt_{p, \mu}, \mu_{p,
    \mu}) = ((\vt\,'_1, \ldots, \vt\,'_n), \mu'_{p, \mu})$, and we let
  $(r^-)^\bullet = i_0 \prec i_1 \prec \cdots
    \prec i_n = r^\bullet$. We define
  \begin{align*}
      \Phi'(r^-, r, \vs_p, \mu)
      & = \left(\prod_{k =
        1}^{n-1} \Thetap'_1(i_{k-1}, i_k, \vt\,'_k, 0)\right)
        \Phi'_1((r^{\bullet})^-, r^\bullet, \vt\,'_n, \mu'_{p,\mu}),
      \\ \Psi'(r^-, r, \vs_p, \mu)
      &=  \Psi'_1(\xi(r)^-, \xi(r), \vt_n, \mu').
    \end{align*}
  It is worth observing that, since each component of $\Theta'_1$ is a
  $\kappa$-word, the components of $\Theta'$ are $\kappa$-words as
  well.
  
  Let us verify that $\iM'$ is a model of $\iS$.
  For Properties
  \ref{m1} and~\ref{m2}, take an element $(i,j, \vs) \in \dom(M)$ and
  let $\mu \in
  M(i,j,\vs)$.
  Property~\ref{m1} follows from the same property for the
  pair~$(\iS_1, \iM_1')$ and, when $(i,j) \neq (r^-, r)$,~\ref{m2} follows from the
  same property for $(\iS_1, \iM_1')$ and from
  Property~\ref{rr23} for~$\Lambda$.
  To prove~\ref{m1} when $(i,j) = (r^-, r)$ is more delicate. We suppose
  that $\vs = \vs_p$.
  Using the construction of~$\Theta'$ and the Property~\ref{m1} for
  $(\iS_1, \iM_1')$, it may be derived that $\DRH$ satisfies
  \begin{equation}
    \label{eq:52}
    \Thetap'(r^-, r, \vs_p, \mu) = w_1'((r^-)^\bullet, r^\bullet)\Psi'_1((r^{\bullet})^-,
    r^\bullet, \vt\,'_n, \mu'_{p,\mu})^{\omega-1} \Psi'_1(\xi(r)^-,
    \xi(r), \vt_n, \mu').
  \end{equation}
  In turn, since $c(\Psi'_1((r^{\bullet})^-, r^\bullet, \vt\,'_n,
  \mu'_{p,\mu})^{\omega-1} \Psi'_1(\xi(r)^-,\xi(r), \vt_n, \mu'))
  \subseteq \cum{ w_1'((r^-)^\bullet, r^\bullet)}$, 
  $\DRH$ also satisfies
  \begin{align}
      w_1'((r^-)^\bullet, r^\bullet)\Psi'_1((r^{\bullet})^-,
      r^\bullet, \vt\,'_n, \mu'_{p,\mu})^{\omega-1} \Psi'_1(\xi(r)^-,
      \xi(r), \vt_n, \mu') & \Req w'_1((r^-)^\bullet, r^\bullet)\nonumber
      \\ & \Req w'(r^-, r).
    \label{eq:8}
  \end{align}
  On the other hand, since the equations
  \begin{align}
    (\xi(r^-) \mid \xi(r))
    &= ((r^-)^\bullet \mid r^\bullet) \cdot
      \{(r^\bullet)^- \mid r^\bullet\}_{\vt\,'_{n}, \mu'_{p,
      \mu}}^{\omega-1}\cdot \{\xi(r)^- \mid \xi(r)\}_{\vt_n, \mu'};\label{eq:53}
    \\  (\xi(r^-) \mid \xi(r))
    &= ((r^-)^\bullet \mid r^\bullet) \cdot
      \{(r^\bullet)^- \mid r^\bullet\}_{\vt\,'_{q,\mu_0}, \mu'_{q,
      \mu_0}}^{\omega-1}\cdot \{\xi(r)^- \mid \xi(r)\}_{\vs\,'_q, \mu_0'}\label{eq:54}
  \end{align}
  belong to $(\iB_\h)_1$, $\h$ satisfies
  \begin{align}
    \Thetap'(r^-, r, \vs_p, \mu)
    & \just{=}{\eqref{eq:52}} w'_1((r^-)^\bullet, r^\bullet)\nonumber
      \\ & \qquad \cdot
      \Psi'_1 ((r^{\bullet})^-, r^\bullet,\vt\,'_n,
      \mu'_{p,\mu})^{\omega-1} \Psi'_1(\xi(r)^-,\xi(r), \vt_n,\mu')\nonumber
    \\ & \just ={\eqref{eq:53}} w'_1(\xi(r^-), \xi(r))\nonumber
    \\ & \just ={\eqref{eq:54}} w_1'((r^-)^\bullet, r^\bullet)
    \\ & \qquad\cdot \Psi'_1((r^{\bullet})^-,
         r^\bullet,\vt\,'_{q, \mu_0},\mu'_{q,\mu_0})^{\omega-1}
         \Psi'_1(\xi(r)^-,\xi(r), \vs\,'_{q},\mu'_0)\nonumber
    \\ & \just={\text{def.}} w'(r^-,r).\label{eq:55}
  \end{align}
  Using~\eqref{eq:52},~\eqref{eq:8},~\eqref{eq:55} and Lemma
  \ref{sec:16}, we finally get that $\DRH$ satisfies the
  pseudoidentity  $\Thetap'(r^-, r, \vs_p, \mu) =
  w'(r^-,r)$, obtaining~\ref{m1}. 
  For Property~\ref{m3}, let $i \prec j$ in $J$. Then, we have
  \begin{align*}
    \cum{w'(i,j)} & =
    \begin{cases}
      \cum{w'_1(\xi(i), \xi(j))}, \quad\text{ if } (i,j) \neq
      (r^-, r);
      \\ \cum{w'_1((r^-)^\bullet, r^\bullet)}, \quad \text{ if } (i,j) =
      (r^-, r);
    \end{cases}
    \\ & =
         \begin{cases}
           \chi_1(\xi(i), \xi(j)),  \quad\text{ if } (i,j) \neq
           (r^-, r);
           \\ \chi_1((r^-)^\bullet, r^\bullet), \quad \text{ if } (i,j) =
          ( r^-, r);
         \end{cases} \quad\text{ by \ref{m3} for $(\iS_1,\iM'_1)$ }
    \\ & =
         \begin{cases}
           \cum{w[\iota_1(\xi(i)), \iota_1(\xi(j))[},  \quad\text{ if } (i,j) \neq
           (r^-, r);
           \\  \cum{w[\iota_1((r^-)^\bullet), \iota_1(r^\bullet)[},  \quad\text{ if } (i,j) =
           (r^-, r);
         \end{cases} \quad\text{ by definition~\eqref{chi}} 
    \\ &=
    \begin{cases}
      \cum{w(i,j)},  \quad\text{ if } (i,j) \neq (r^-, r);
      \\ \cum{w[(r^-)^\circ, r^\circ[},\quad \text{ if } (i,j) = (r^-, r);
    \end{cases}\quad\text{ by definition of } \_^\bullet, \_^\circ \text{
                 and } \xi
    \\ & =
         \begin{cases}
           \cum{w(i,j)},  &\text{ if } (i,j) \neq (r^-, r);
           \\ \cum{w(r^-, r)},& \text{ if } (i,j) = (r^-, r);
         \end{cases} \quad\text{ by Lemma  \ref{9.3}\ref{9.3.2}}
    \\ & = \chi(i,j) \quad\text{ by \ref{m3} for$(\iS, \iM)$}.
  \end{align*}
  To prove that Property~\ref{m4} holds, we first notice that,
  for every $i < j < r$ in $T$, 
  \begin{align}
    w_1'(\xi(i), \xi(j)) =_\DRH w_1'(i^\bullet,
    j^\bullet). \label{eq:3}
  \end{align}
  Now, let $(i,x)$ be a
  box in $\iB'$. Using the definitions of $w'$ and of $\iota'$ we may compute
  \begin{align*}
    w'(i,\dirr(x)) & =
                    \begin{cases}
                      w_1'(\xi(i), \xi(\dirr(x))), \quad \text{ if }
                      \dirr(x) \le r^-;
                      \\  w_1'(\xi(i), \xi(r^-)) \:
                      w_1'((r^-)^\bullet, r^\bullet)
                      \\  \, \cdot \Psi'_1((r^\bullet)^-,
                      r^\bullet, \vt\,'_{q, \mu_0}, \mu'_{q,
                        \mu_0})^{\omega - 1} \Psi'_1(\xi(r)^-, \xi(r),
                      \vs\,'_q, \mu_0'), \,\text{ otherwise;}
                    \end{cases}
    \\ & \kern-5pt \just ={\eqref{eq:3}}
         \begin{cases}
            w_1'(\xi(i), \xi(\dirr(x))),\quad \text{ if }
                      \dirr(x) \le r^-;
           \\ w_1'(i^\bullet, (r^-)^\bullet)\: w_1'((r^-)^\bullet,
           r^\bullet)
           \\  \, \cdot \Psi'_1((r^\bullet)^-, r^\bullet, \vt\,'_{q,
             \mu_0}, \mu'_{q,\mu_0})^{\omega - 1} \Psi'_1(\xi(r)^-,
           \xi(r), \vs\,'_q, \mu_0'), \,\text{ otherwise;}
         \end{cases}
    \\ & \Req
         \begin{cases}
           w_1'(\xi(i),\dirr_1(x)), \quad\text{ if } \dirr(x) \le r^-;
           \\ w_1'(i^\bullet, \dirr_1(x)), \quad\text{ otherwise.}
         \end{cases}
  \end{align*}
  Taking into account the steps~\ref{eq2} and~\ref{eq3} in the
  construction of $\iB_1$, it is now
  easy to deduce that~\ref{m4} holds for all the relations added in
  those steps. It remains to verify that $w'(\ell, r)$ and $w'(\ell^*, r^*)$ are
  $\Req$-equivalent modulo $\DRH$. For that purpose, we show that the
  following relations hold in $\DRH$:
  \begin{align*}
    w'(\ell,r) & = w'(\ell, r^-) w'(r^-, r)
    \\ & \kern-3pt\just={\text{def.}} w_1'(\xi(\ell), \xi(r^-))
    \\ & \quad \cdot w_1'((r^-)^\bullet, r^\bullet)
         \Psi'_1((r^\bullet)^-, r^\bullet, \vt\,'_{q, \mu_0}, \mu'_{q,
         \mu_0})^{\omega - 1} \Psi'_1(\xi(r)^-, \xi(r), \vs\,'_q,
         \mu_0')
    \\ & \kern-7pt\just = {\eqref{eq:3}} w_1'(\ell^\bullet,
         (r^-)^\bullet)
    \\ & \quad \cdot 
         w_1'((r^-)^\bullet, r^\bullet) \Psi'_1((r^\bullet)^-,
         r^\bullet, \vt\,'_{q, \mu_0}, \mu'_{q,\mu_0})^{\omega - 1}
         \Psi'_1(\xi(r)^-, \xi(r), \vs\,'_q, \mu_0')
    \\ & = w_1'(\ell^\bullet, r^\bullet) \Psi'_1((r^\bullet)^-,
         r^\bullet, \vt\,'_{q, \mu_0}, \mu'_{q,\mu_0})^{\omega - 1}
         \Psi'_1(\xi(r)^-, \xi(r), \vs\,'_q, \mu_0')
    \\ & \Req  w_1'(\ell^\bullet, r^\bullet),\quad \text{ because }
         \\ & \hspace{1cm}{\text{$c(\Psi'_1((r^\bullet)^-,
         r^\bullet, \vt\,'_{q, \mu_0}, \mu'_{q,\mu_0})^{\omega - 1}
         \Psi'_1(\xi(r)^-, \xi(r), \vs\,'_q, \mu_0')) \subseteq \cum{ w_1'(\ell^\bullet, r^\bullet)}$}}
    \\ & = w_1'(\iota_1^{-1}(\ell^\circ),
         \iota_1^{-1}(r^\circ)) = w_1'(\iota_1^{-1}(\iota(\ell^*)),
         \iota_1^{-1}(\iota(r^*)))
    \\ & = w_1'(\xi(\ell^*), \xi(r^*)) = w'(\ell^*, r^*).
  \end{align*}
  Finally, since $\xi_{\Lambda}(\iB_\h) \subseteq (\iB_\h)_1 $, in Remark
  \ref{simp} we observed that, in order to prove that Property~\ref{m5}
  is satisfied, it  is enough to prove that $\h$ satisfies
  \begin{align}
    w'(i,j) &= w_1'(\xi(i), \xi(j)) \label{eq:7}
    \\\Psi'(i,j,\vs, \mu) &= \Psi_1'(\xi(j)^-, \xi(j), \vs\,',
    \mu'),
    \label{eq:6}
  \end{align}
  for every $(i,j,\vs, \mu) \in \dom(M) \times
  M(i,j,\vs, \mu)$, where $\Lambda(i,j,\vs,\mu) = ((\ldots, \vs\,'), \mu')$.
  The pseudoidentity~\eqref{eq:7} follows straightforwardly from the definition of
  $w'$, except when $(i,j) = (r^-, r)$. In that case, by computing~\eqref{eq:7} modulo $\h$, we get
  \begin{align*}
    w'(r^-, r) & = w_1'((r^-)^\bullet, r^\bullet) \Psi'_1((r^\bullet)^-,
                 r^\bullet, \vt\,'_{q, \mu_0}, \mu'_{q, \mu_0})^{\omega - 1}
                 \Psi'_1(\xi(r)^-, \xi(r), \vs\,'_q, \mu'_0)
    \\ & = w_1'(\xi(r^-), \xi(r) ),
  \end{align*}
  where the last equality holds because the equation
  \begin{equation*}
    (\xi(r^-) \mid \xi(r)) = ((r^-)^\bullet \mid r^\bullet) \cdot
    \{(r^\bullet)^- \mid r^\bullet\}_{\vt\,'_{q,\mu_0},
      \mu'_{q,\mu_0}}^{\omega-1}\cdot \{\xi(r)^- \mid \xi(r)\}_{\vs\,'_q,
      \mu'_0}
  \end{equation*}
  belongs to $(\iB_\h)_1$ and $\iM_1'$ is a model of $\iS_1$.
  Lastly, the pseudoidentity~\eqref{eq:6} corresponds precisely to the
  definition of $\Theta'$.
  Thus, $\iM'$ is a model of $\iS$ in $\kappa$-words and so,
  Property~\ref{p2} holds for the pair~$(\iS_1, \iM_1)$.
\end{proof}
\subsection{Case 5}
Finally, it remains to consider the case where $\iB$ has a boundary
relation of the form $(i,x,j,\xx)$ with $\dirr(x) = r = \dirr(\xx)$ and
none of the Cases~\ref{caso1}--\ref{caso5} hold. In particular, the
non occurrence of Cases~\ref{caso2},~\ref{caso4} and~\ref{caso5} implies
that all the boundary relations $(i,x,j,\xx)$ verifying $i \le j$ and
$\dirr(\xx) = r$ are such that $i < j$, $\dirr(x) = r$ and the equality $c(w(i,j)) =
c(w(i,r))$ holds.

We consider the index
$$c = \max \{\min(J), \max\{\dirr(x)\colon \dirr(x) < r\}, \max\{i \in J
\colon i < r \text{ and }\nexists \text{ a box }(i,x)\}\}$$
and we let $\iE = \{(i,x,j,\xx) \in \iB \colon i<j;\: \dirr(x) = r =
\dirr(\xx)\}$.
By the auxiliary step, we may assume that all the boundary relations
of $\iE$ are such that $c < i,j < r$. Since the auxiliary step
consists in successively factorizing a boundary relation from
$\iE$ with respect to a pair of ordinals both greater than $\iota(c)$
(recall Fig.~\ref{fig:aux_step2} and Lemma~\ref{l:10}), it follows
that for every index $c < i < r$ there exists a box $(i,x)$ such that
$\dirr(x) = r$.
Observe that the choice of $c$ guarantees that all the indices in the
original set of boundary relations already satisfy this condition.
Moreover, since $\iE$ contains all the boxes ending in
$r$, if $(i,x)$ is a box such that $\dirr(x) = r$, then $c < i < r$.

Now, we let
$\ell = \max\{i \in J \colon \text{ there exists } (i,x,j, \xx) \in
\iE\}$.
Using the construction presented in Subsection~\ref{aux_step} to align the
left of each variable intervening in~$\iE$ (as schematized in
Fig.~\ref{fig:4}),
\begin{figure}[htpb]
  \centering
  \unitlength=0.0035mm \brickbackup=-650\brickheight=1000
  \begin{picture}(33000,7000)(1500,-4000)
    \footnotesize
    \brick(3000,2000,14000,100,200)(\ell, x)
    \brick(5000,1000,12000,100,200)(j, \xx)
    \brick(0,-1000,17000,100,300)(i_e,x_e)
    \brick(6000,-2000,11000,100,300)(j_e, \xx_e)
    \multiput(3000,-4000)(0,1000)6{\dvert(0,0)}\put(3100,-3800){$\iota(\ell)$}
    \multiput(9000,-4000)(0,1000)3{\dvert(0,0)}\put(9100,-3800){$\beta_e
      = \iota(j_e) + (\iota(\ell) - \iota(i_e))$}
    \footnotesize
    \brick(22000,2000,14000,100,200)(\xi(\ell), x)
    \brick(24000,1000,12000,100,200)(\xi(j), \xx)
    \brick(22000,-1000,14000,100,300)(\ell_e,x_e)
    \brick(28000,-2000,8000,100,300)(k_e, \xx_e)
    \scriptsize
    \brick(19000,-4000,3000,20,30)(\xi(i_e),y_e)
    \brick(25000,-4000,3000,20,30)(\xi(j_e), \y_e)
  \end{picture}
  \caption{Aligning a boundary relation on the left with $\ell$.}
  \label{fig:4}
\end{figure}
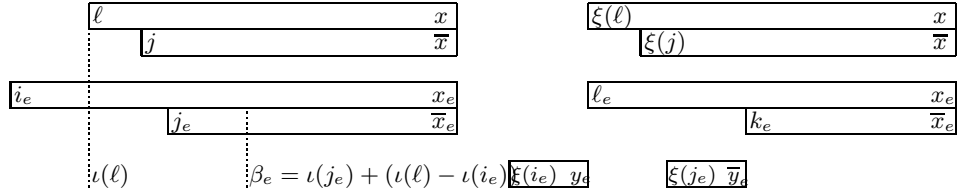%
we may assume, without loss of generality, that the set~$\iE$ defined
above is given by $\iE = \{(\ell, x_1, j_1, \xx_1), \ldots, (\ell, x_n,
j_n, \xx_n)\},$ with $j_1 \le j_2 \le \cdots \le j_n$.
We notice that, by definition
of the index $c$, we have $j_n \prec r$ in~$J$.
Since $\iM$ is a model of $\iS$, $\DRH$ satisfies
$w(\ell, j_m) w(\ell, r) \Req w(\ell, j_m) w(j_m, r) = w(\ell, r),$
for $m = 1, \ldots, n$. Multiplying successively by $w(\ell, j_m)$ on
the left, we get that $\DRH$ satisfies $w(\ell, j_m)^\omega w(\ell, r) \Req w(\ell,r)$.
Since $\cum{w(\ell, j_m)^\omega} = c(w(\ell, j_m)) = c(w(\ell,r))$, it
follows that $\DRH$ satisfies
\begin{equation}
  \label{eq:16}
  w(\ell, r) \Req w(\ell, j_1)^\omega \Req \cdots \Req
  w(\ell, j_n)^\omega.
\end{equation}
But all the pseudowords $w(\ell, j_m)^\omega$ represent the identity
in the same maximal subgroup of $\pseudo A{DRH}$ where they belong
(recall Proposition~\ref{p:3}).
Therefore, all the elements
$w(\ell,j_m)^\omega$ are the same over $\DRH$.
Then, Proposition~\ref{JA5.5} applied
to the elements $w(\ell, j_1), \ldots, w(\ell, j_n)$ guarantees the
existence of pseudowords $u \in \pseudo AS$, $v_1, \ldots, v_n \in
(\pseudo AS)^I$ and of positive integers $h_1, \ldots, h_n$ such that,
for $m = 1, \ldots, n$ we have
\begin{equation}
  \begin{aligned}
    w(\ell, j_m) & =_\DRH u^{h_m}v_m,
    \\ v_m u & =_\DRH u,
  \end{aligned}
\label{eq:15}
\end{equation}
where all the products $u\cdot u$, $u \cdot v_m$ and $v_m \cdot u$ are
reduced. Note that $h_n$ is the maximum of $\{h_1, \ldots, h_n\}$.

We observe that the pseudoidentities in~\eqref{eq:15} imply
that every finite power of $u$ is a prefix of $w(\ell, j_m)^\omega$,
which in turn, by~\eqref{eq:16}, is $\Req$-equivalent to $w(\ell, r)$
modulo $\DRH$. Since the semigroup $S$ where the constraints are
defined is finite, this allows us to find some periodicity on
them. With this in mind, to deal with the constraints, we consider a
big enough direct power of the semigroup $S$, more specifically, the
semigroup $T = S^K$, with $K =  \sum_{\vs \in
  \zeta(j_n,r)} M(j_n,r, \vs)$,
and we take $N = \card T + 2$.
Let us construct a new pair $(\iS_1,
\iM_1)$ as follows:
$$\iS_1 = (\iX_1, J_1, \zeta_1, M_1, \chi_1, \dirr_1, \iB_1, (\iB_\h)_1)
\text{ and } \iM_1 = (w_1, \iota_1, \Theta_1),$$
where
\begin{itemize}
\item the set of variables is $\iX_1 = \iX \uplus \{y_q, \y_q\}_{q = 1}^h
  \uplus \{z_m, \z_m\}_{m = 1}^n \uplus \{f_i, \f_i\}_{i = 1}^N$,
  where variables with different names are assumed to be distinct;
\item the pseudoword in the model is $w_1 = w$;
\item let $O$ be the set containing the following ordinals:
  \begin{itemize}
  \item $\beta_0 = \iota(\ell)$;
  \item $\beta_q = \beta_0 + \alpha_u \cdot q$, for $q = 1, \ldots,
    h_n+1$;
  \item $\gamma_m = \beta_0 + (\iota(j_m) - \beta_{h_m})$, for $m = 1,
    \ldots, n$;
  \item $\delta_p = \beta_0 + \alpha_u \cdot h_np$, for $p = 0, \ldots,
    N$.
  \end{itemize}
  We let $\iC_1 = (J_1, \iota_1, M_1, \Theta_1)$ be a common
  refinement of the factorization schemes $\iC(\iS, \iM)$ and $(O, O
  \hookrightarrow \alpha_w+1, \emptyset, \emptyset)$ for $w$ and
  $$\Lambda:\{(i,j,\vs,\mu)\colon (i,j,\vs) \in \dom(M), \: \mu \in
  M(i,j,\vs)\} \to \bigcup_{k \in \nn} (S \times S^I)^k \times \omega$$
  be a refining function from $\iC(\iS, \iM)$ to $\iC_1$.
  The factorization scheme $\iC_1$ supplies the items $J_1$,
  $\iota_1$, $M_1$ and $\Theta_1$ and the items $\zeta_1$ and $\chi_1$
  by taking $\zeta_1 = \zeta_{w_1, \iC_1}$ and $\chi_1 = \chi_{w_1,
    \iC_1}$ (recall~\eqref{zeta} and~\eqref{chi}).
  We denote $b_q = \iota_1^{-1}(\beta_q)$, $c_m =
  \iota_1^{-1}(\gamma_m)$, $d_p = \iota_1^{-1}(\delta_p)$, and $\xi =
  \iota_1^{-1} \circ \iota$;
\item the function $\dirr_1$ is given by
  $$\dirr_1 (x) =
  \begin{cases}
    \xi(\dirr(x)), \quad\text{ if } x \in \iX;
    \\ b_q, \quad\text{ if } x = y_q;
    \\ b_{q+1}, \quad \text{ if } x = \y_q;
    \\b_{h_m+1}, \quad\text{ if } x \in \{z_m, \z_m\};
    \\ d_p, \quad\text{ if } x = f_p;
    \\d_{p+1}, \quad\text{ if } x = \f_p;
  \end{cases}
  $$
\item in the set $\iB_1$ we include the following boundary relations:
  \begin{itemize}
  \item $(\xi(i), x, \xi(j), \xx)$, if $(i,x,j,\xx) \in \iB \setminus
    (\iE \cup \{\text{dual of } e\colon e \in \iE\})$;
  \item $(b_{q-1}, y_q, b_q, \y_q)$ and $(b_q, \y_q, b_{q-1}, y_q)$,
    for $q = 1, \ldots, h$;
  \item $(b_{h_m}, z_m, \xi(j_m), \z_m)$ and $(\xi(j_m), \z_m,
    b_{h_m}, z_m)$, for $m = 1, \ldots, n$;
  \item $(d_{p-1}, f_p, d_p, \f_p)$ and $(d_p, \f_p, d_{p-1}, f_p)$,
    for $p = 1, \ldots, N-1$;
  \end{itemize}
\item the set $(\iB_\h)_1$ consists of the following equations:
  \begin{itemize}
  \item all the equations of $\xi_\Lambda(\iB_\h)$;
  \item $(b_0 \mid b_1) = (b_1 \mid b_2) = \cdots = (b_{h_n} \mid
    b_{h_n+1})$;
  \item $(b_{h_m} \mid b_{h_m+1}) = (\xi(j_m) \mid b_{h_m+1})$, for $m
    = 1, \ldots, n$;
    \item $(d_0 \mid d_1) = (d_1 \mid d_2) = \cdots = (d_{N-1} \mid d_N)$.
  \end{itemize}
\end{itemize}

We leave to the reader to verify that $\iM_1$ is indeed a model of $\iS_1$.
\begin{prop}
  Let $(\iS_1, \iM_1)$ be the pair defined above. Then, $\iM_1$ is a
  model of $\iS_1$.
\end{prop}

Since in $\iB_1$ there are no boxes ending at $r$, we decrease the
first component of the induction parameter, and so, Property
\ref{p1} holds. Before proving that Property~\ref{p2} also holds,
we define integers $1 \le H < M < N$ that later play an essential
role.

Recall that, in $J_1$, we have $\xi(j_n) \prec b_{h_n+1} \preceq d_2
\prec d_3 \prec \cdots \prec d_N \prec \xi(r)$, where $b_{h_n+1} =
d_2$ if and only if $h_n = 1$. Therefore, for each $(\vs, \mu) \in
\zeta(j_n, r) \times M(j_n, r, \vs)$, the first component of
$\Lambda(j_n,r, \vs)$ belongs to $(S \times S^I)^N$ if $h_n = 1$ or to
$(S \times S^I)^{N+1}$, otherwise. We assume that $h_n > 1$. The same
argument can be used when $h_n = 1$, simply by working with $N$ instead
of $N+1$. We may write
\begin{equation}
\Lambda(j_n, r, \vs, \mu) = \left(\left(\vt^{\;(\vs, \mu)}_1, \ldots,
    \vt_{N+1}^{\;(\vs,\mu)}\right), \mu_{\vs, \mu}\right),\label{eq:27}
\end{equation}
with $\vt_i^{\;(\vs, \mu)} = \left(t_{i,1}^{(\vs,\mu)}, t_{i,2}^{(\vs,
    \mu)}\right)$. Let $\vt_1, \ldots, \vt_{N+1} \in T$ satisfy the
following properties:
\begin{itemize}
\item each element $\vt_i$ is a tuple whose coordinates are of the
  form $t_{i,1}^{(\vs, \mu)}t_{i,2}^{(\vs, \mu)}$, for certain $\vs
  \in \zeta(j_n,r)$ and $\mu \in M(j_n, r, \vs)$;
\item for $i \in \{1, \ldots, K\}$ and $k_1 \neq k_2$, if the $k_1$-th
  coordinate of $\vt_i$ is $t_{i,1}^{(\vs_1, \mu_1)}t_{i,2}^{(\vs_1,
    \mu_1)}$ and the $k_2$-th coordinate of $\vt_i$ is
  $t_{i,1}^{(\vs_2, \mu_2)}t_{i,2}^{(\vs_2, \mu_2)}$, then $(\vs_1,
  \mu_1) \neq (\vs_2, \mu_2)$;
\item for $i \in \{2, \ldots, K\}$, if the $k$-th coordinate of
  $\vt_1$ is $t_{1,1}^{(\vs, \mu)}t_{1,2}^{(\vs, \mu)}$, then the
  $k$-th coordinate of $\vt_i$ is $t_{i,1}^{(\vs, \mu)}t_{i,2}^{(\vs,
    \mu)}$.
\end{itemize}
Since $N-1 > \card T$, there exist $1 \le H < M < N$ such that
$\vt_1 \cdots \vt_H = \vt_1 \cdots \vt_M$, which implies that
$$\vt_1 \cdots \vt_H \vt_{H+1} \cdots \vt_M = \vt_1 \cdots \vt_H
(\vt_{H+1} \cdots \vt_M)^\omega.$$
In order to ease the notation,  for each $\vs = (s_1,
s_2) \in \zeta(j_n, r)$ and $\mu \in 
M(j_n, r, \vs)$, we define
\begin{equation}
  \begin{aligned}
    & s_1^{(\vs, \mu)} =  t_{1,1}^{(\vs, \mu)} t_{1,2}^{(\vs, \mu)}
    \cdot t_{2,1}^{(\vs, \mu)} t_{2,2}^{(\vs, \mu)} \cdots
    t_{H,1}^{(\vs, \mu)} t_{H,2}^{(\vs, \mu)},
    \\ & s_2^{(\vs, \mu)} = t_{H+1,1}^{(\vs, \mu)}\: t_{H+1,2}^{(\vs,
      \mu)} \cdots t_{M,1}^{(\vs, \mu)} t_{M,2}^{(\vs, \mu)},
    \\ & s_{3,1}^{(\vs, \mu)} =  t_{M+1,1}^{(\vs, \mu)} \:
    t_{M+1,2}^{(\vs, \mu)} \cdots t_{N,1}^{(\vs, \mu)} \:
    t_{N,2}^{(\vs, \mu)} \: t_{N+1,1}^{(\vs, \mu)},
    \\ & s_{3,2}^{(\vs, \mu)} = t_{N+1, 2}^{(\vs, \mu)}.
  \end{aligned}\label{eq:28}
\end{equation}
Then, since $\Lambda$ satisfies~\ref{rr23}, we have $s_1 = s_1^{(\vs, \mu)} \cdot \left(s_2^{(\vs,
    \mu)}\right)^{\omega+1} \cdot s_{3,1}^{(\vs, \mu)}$ and $s_{3,2}^{(\vs, \mu)} =  s_2$.

Next, we verify that Property~\ref{p2} is satisfied, as claimed
before.
\begin{prop}
  Suppose that there exists a model $\iM_1' = (w_1', \iota_1',
  \Theta_1')$ of $\iS_1$ in $\kappa$-words. Then, there is a model of
  $\iS$ in $\kappa$-words as well.
\end{prop}
\begin{proof}
  Let $\iM' = (w', \iota', \Theta')$ be constructed as follows.
  The $\kappa$-word $w'$ is set to be
  \begin{align*}
    w' = w_1'[0, \iota_1'(d_M){[} \cdot (w_1'(d_H, d_M))^\omega
    w_1'(d_M, \xi(r))\cdot w_1'[\iota_1'(\xi(r)), \alpha_{w_1'}{[}.
  \end{align*}
  The map  $\iota': J \to \alpha_{w'}+1$ is given by $\iota'(i)
  = \iota_1'( \xi(i))$ if $i < r$, $\iota'(r) =  \alpha_{w_1'[0,\iota_1'(d_M){[} \cdot (w_1'(d_H,
    d_M))^\omega }$, and $\iota'(i) = \iota'(r) + (\iota_1' (\xi(i)) -
  \iota_1'(\xi(r)))$ if $i > r$.
  In order to define $\Theta'$, we first consider the following
  auxiliary pseudowords:
\begin{itemize}
  \item for each $i \prec j \le j_n$ and each $r \le i \prec j$ in
    $J$, each $\vs \in \zeta(i,j)$ and each $\mu \in M(i,j, \vs)$, if
    $\Lambda(i,j,\vs, \mu) = ((\vt_1, \ldots, \vt_k), \mu')$ and $\xi(i) =
    i_0 \prec i_1 \prec \cdots \prec i_k = \xi(j)$, then we take
    \begin{align*}
      \Phi_0'(i,j,\vs, \mu)
      &= \left(\prod_{m = 1}^{k-1}
        \Thetap_1'(i_{m-1}, i_m, \vt_m, 0)\right) \cdot
        \Phi_1'(\xi(j)^-, \xi(j), \vt_k, \mu');
      \\ \Psi_0'(i,j,\vs, \mu)
      &= \Psi_1'(\xi(j)^-, \xi(j), \vt_k, \mu');
    \end{align*}
  \item for each $\vs \in \zeta(j_n, r)$ and $\mu \in M(j_n, r, \vs)$,
    we set (recall the notation in~\eqref{eq:27})
    \begin{align*}
      \Phi_0'(j_n, H, \vs, \mu) &= \Thetap_1'\left(\xi(j_n), b_{h+1},
                                  \vt_1^{\;(\vs, \mu)}, 0\right)
      \\ & \quad \cdot \Thetap_1'\left(b_{h+1},
        d_2, \vt_2^{\;(\vs,\mu)}, 0\right)
       \cdot \prod_{m = 3}^H
        \Theta_1'\left(d_{m-1}, d_m, \vt_m^{\;(\vs, \mu)}, 0\right);
      \\ \Phi_0'(H,M,\vs, \mu) &= \prod_{m = H+1}^M
           \Thetap_1'\left(d_{m-1}, d_m, \vt_m^{\;(\vs,\mu)}, 0\right);
      \\ \Phi_0'(M, r, \vs, \mu) &= \left(\prod_{m = M+1}^N
           \Thetap_1'\left(d_{m-1},d_m, \vt_m^{\;(\vs, \mu)},
                                   0\right)\right)
      \\ & \quad\cdot \Phi_1'\left(d_N, \xi(r), \vt_{N+1}^{\;(\vs,
           \mu)}, \mu_{\vs, \mu}\right);
      \\ \Psi_0'(M, r, \vs, \mu) &= \Psi_1'\left(d_N, \xi(r),
           \vt_{N+1}^{\;(\vs, \mu)}, \mu_{\vs, \mu}\right).
    \end{align*}
  \end{itemize}
  Now, for $i \prec j$ in $J$, $(i,j, \vs) \in \dom(M)$ and $\mu \in
  M(i,j,\vs)$ we define
  \begin{align*}
    &\Theta'(i,j,\vs, \mu) =
      (\Phi_0'(i,j,\vs,\mu),\Psi_0'(i,j,\vs,\mu)), \text{ whenever }j
      \neq r;
    \\ &\Theta'(j_n, r, \vs, \mu) = (\Phi_0'(j_n, H, \vs, \mu)\cdot
         \Phi_0'(H, M, \vs, \mu)^{\omega+1} \cdot \Phi_0'(M, r, \vs,
         \mu),\Psi_0'(M, r, \vs, \mu)).
  \end{align*}
  Now, we verify that $\iM'$ just defined is a model of $\iS$. Let
  $(i,j,\vs) \in \dom(M)$ be such that $\vs = (s_1, s_2)$, and $\mu \in
  M(i,j,\vs)$.
  Suppose that $j \neq r$, write $\Lambda(i,j,\vs, \mu) =
  ((\vt_1, \ldots, \vt_k), \mu')$,  and let $\xi(i) = i_0 \prec i_1
    \prec \cdots \prec i_k = \xi(j)$. Then, using the definition of
  $(\Phi_0', \Psi_0')$, it is easy to derive~\ref{m1} using the same
  property for the pair $(\iS_1, \iM_1')$.
  Similarly, invoking
  Property~\ref{m2} for the pair $(\iS_1, 
  \iM_1')$ and writing $\vt_m = (t_{m,1}, t_{m,1})$, we may deduce the
  equalities $\varphi(\Phi'(i,j,\vs, \mu)) = \left(\prod_{m = 1}^{k-1}
    t_{m,1}t_{m,2} \right) \cdot t_{k,1}$ and
  $\varphi(\Psi'(i,j,\vs, \mu)) = t_{k,2}$. In turn,
  Property~\ref{rr23} for $\Lambda$ yields Property~\ref{m2} for the
  pair $(\iS, \iM')$.
  We justify~\ref{m3} with the following computation:
  \begin{align*}
    \cum{w'(i,j)} & = \cum{w_1'(\xi(i), \xi(j))} \quad\text{by
                    definition of $\iota'$ and $w'$}
    \\ & = \chi_1(\xi(i), \xi(j)) \quad\text{by~\ref{m3} for
         $(\iS_1, \iM_1')$}
    \\ & = \cum{w[\iota_1(\xi(i)), \iota_1(\xi(j)){[}} \quad \text{by
         definition~\eqref{chi} of $\chi_1 = \chi_{w_1, \iC_1}$}
    \\ & = \cum{w(i,j)}
    \\ & = \chi(i,j) \quad\text{by~\ref{m3} for
         $(\iS, \iM)$}.
  \end{align*}
  Now, consider the case where $i = j_n$ and $j=r$.
  Again, we may use Property~\ref{m1} for $(\iS_1, \iM_1')$ to obtain
  the identity $\Thetap'(j_n,r,\vs, \mu) = w'(j_n,r)$ in $\DRH$,
  thereby proving \ref{m1}.
  In order to prove \ref{m2}, we use the same property for the pair
  $(\iS_1, \iM_1')$ to derive the following equalities (recall \eqref{eq:28}):
  \begin{align*}
    \varphi(\Phi'(j_n, r, \vs, \mu'))
    & = s_1^{(\vs, \mu)} \cdot \left (s_2^{(\vs,
         \mu)}\right)^{\omega+1} \cdot  s_{3,1}^{(\vs, \mu)} =
      s_1,
    \\  \varphi(\Psi'(j_n, r, \vs, \mu'))
    &  = s_{3,2}^{(\vs, \mu)} =  s_2.
  \end{align*}
  To establish~\ref{m3}, we observe that,
  since $S$ has a content function and thanks to Property~\ref{m2} for
  both pairs $(\iS, \iM)$ and $(\iS_1, \iM_1')$, the content of the
  corresponding segments in $w$ and in $w_1'$ does not change. Therefore,
  the equalities
  \begin{equation}
    \begin{aligned}
     c(w_1'(\xi(j_n), d_M))  & = c(w[\iota(j_n), \delta_M{[}) = c(w[\beta_0,
      \beta_1{[}),
      \\  c(w_1'(d_H, d_M))  & = c(w[\delta_H, \delta_M{[}) = c(w[\beta_0,
      \beta_1{[}),
      \\ c(w_1'(d_M, \xi(r)))  & = c(w[\delta_M, \iota(r){[}) = c(w[\beta_0,
      \beta_1{[})
    \end{aligned}\label{eq:31}
  \end{equation}
  hold. Thus, we also have
  \begin{align*}
    \cum{w'(j_n, r)} & = \cum{w_1'(\xi(j_n),d_M)\cdot w_1'(d_H, d_M)^
                       {\omega} \cdot w_1'(d_M, \xi(r))}
                       \\ & = c(w[\beta_0, \beta_1{[}) =
    \cum{w(j_n, r)} = \chi(j_n, r).
  \end{align*}
  
  It remains to verify that~\ref{m4} and~\ref{m5} are satisfied. For
  Property~\ref{m4}, all boundary relations but the ones of
  the form $(\ell, x_m, j_m, \xx_m)$ are immediate. For those
  relations, we already observed in~\eqref{eq:31} that $c(w_1'(d_H,
  d_M)) = c(w_1'(d_M, \xi(r)))$, so that,
  $w'(j_m, r)$ and $w_1'(\xi(j_m), d_M)\cdot w_1'(d_H,
  d_M)^\omega $ lie in the same $\Req$-class modulo $\DRH$. Hence, the
  pseudovariety $\DRH$ satisfies
  \begin{align*}
    w'(\ell, r) & = w_1'(d_0, d_M)\cdot w_1'(d_H, d_M)^\omega
                  w_1'(d_M, \xi(r))
    \\ & \Req w_1'(d_0, d_M) \cdot w_1'(d_H, d_M)^\omega
    \\ & = w_1'(b_0, b_1)^{h_nM} \cdot w_1'(d_H, d_M)^\omega
    \\ & \just ={(*)} w_1'(b_{h_m}, b_{h_m+1}) w_1'(b_0,b_1)^{h_nM-1}\cdot
         w_1'(d_H, d_M)^\omega
    \\ & = w_1'(\xi(j_m), b_{h_m+1}) \cdot w_1'(b_0, b_1)^{h_nM-1} \cdot
         w_1'(d_H, d_M)^\omega
    \\ & \Req w_1'(\xi(j_m), d_M) \cdot w_1'(d_H, d_M)^\omega
    \\ & \Req w'(j_m, r).
  \end{align*}
  The validity of step $(*)$ is justified in view of $\iS_1$ having
  $\iM_1'$ as a model. More precisely, it follows from Property~\ref{m4}
  for the relation  $(b_{h_m}, z_m,  \xi(j_m), \z_m)$ and from Property
  \ref{m5} for the equation $(b_{h_m} \mid b_{h_m+1}) = (\xi(j_m) \mid
  b_{h_m+1})$, together with Lemma~\ref{sec:16}.
  Finally, as the inclusion $\xi_\Lambda(\iB_\h) \subseteq (\iB_\h)_1$
  holds, by Remark~\ref{sec:21} it is enough to show that for all $(i,j,\vs)
  \in \dom(M)$ and $\mu \in M(i,j,\vs)$, if $\Lambda(i,j,\vs,
  \mu) = ((\ldots, \vt\:), \mu')$, then the pseudoidentities
  \begin{align*}
    w'(i,j) & = w_1'(\xi(i), \xi(j));
    \\  \Psi'(i,j,\vs, \mu) & = \Psi_1'(\xi(j)^-, \xi(j), \vt, \mu')
  \end{align*}
  are valid in $\h$. Analyzing the construction of $\Psi'$, the second
  pseudoidentity becomes clear, since it is actually an equality of
  pseudowords. The first pseudoidentity $w'(i,j) = w_1'(\xi(i),
  \xi(j))$ is also immediate, whenever $j \neq r$, after noticing
  that $w'(i,j) = w_1'(\xi(i), \xi(j))$. It remains to prove that
  $w'(j_n, r) = w_1'(\xi(j_n), r)$ modulo $\h$. That is made clear
  in the next computation modulo $\h$:
  \begin{align*}
    w'(j_n,r)
    & = w_1'(\xi(j_n), d_M) \cdot w_1'(d_H, d_M)^\omega\cdot w_1'(d_M,
      \xi(r))
    \\ & =  w_1'(\xi(j_n), d_M) \cdot w_1'(d_M, \xi(r)) =
         w_1'(\xi(j_n), \xi(r)).
  \end{align*}
  This completes the proof.
\end{proof}

We have just completed the analysis of all the Cases
\ref{caso1}--\ref{caso6}. Thus, we proved  Theorem~\ref{main}.
The announced result follows from Corollary~\ref{gen7.5}.

\begin{theorem}
  Let $\h$ be a pseudovariety of groups. Then, the pseudovariety
  $\DRH$ is completely $\kappa$-reducible if and only if the
  pseudovariety~$\h$ is completely $\kappa$-reducible.
\end{theorem}

Now, we are able to supply a family of examples of completely $\kappa$-reducible
pseudovarieties.
\begin{cor}
  The pseudovariety $\DRH$ is completely $\kappa$-reducible for every
  locally finite pseudovariety of groups~$\h$.
\end{cor}
\begin{cor}
  The pseudovariety ${\sf DRAb}$ is completely $\kappa$-reducible.
\end{cor}
\section*{Acknowledgments}
The work of both authors was supported, in part, by CMUP
(UID/MAT/ 00144/2013), which is funded by FCT (Portugal) with national
(MEC) and European structural funds through the programs FEDER, under
the partnership agreement PT2020. The second author was
also partially supported by the FCT doctoral scholarship
(SFRH/BD/75977/ 2011), with national (MEC) and European
  structural funds through the program POCH.
\def\cprime{$'$}
\providecommand{\bysame}{\leavevmode\hbox to3em{\hrulefill}\thinspace}
\providecommand{\MR}{\relax\ifhmode\unskip\space\fi MR }
\providecommand{\MRhref}[2]{%
  \href{http://www.ams.org/mathscinet-getitem?mr=#1}{#2}
}
\providecommand{\href}[2]{#2}


\begin{thebibliography}{10}

\bibitem{MR1150933}
D.~Albert, R.~Baldinger, and J.~Rhodes, \emph{Undecidability of the identity
  problem for finite semigroups}, J. Symbolic Logic \textbf{57} (1992), no.~1,
  179--192.

\bibitem{livro}
J.~Almeida, \emph{Finite semigroups and universal algebra}, World Scientific
  Publishing Co. Inc., River Edge, NJ, 1994, Translated from the 1992
  Portuguese original and revised by the author.

\bibitem{hyperdecidable}
\bysame, \emph{Hyperdecidable pseudovarieties and the calculation of semidirect
  products}, Internat. J. Algebra Comput. \textbf{9} (1999), no.~3-4, 241--261.

\bibitem{unified_theory}
\bysame, \emph{Finite semigroups: an introduction to a unified theory of
  pseudovarieties}, Semigroups, algorithms, automata and languages ({C}oimbra,
  2001), World Sci. Publ., River Edge, NJ, 2002, pp.~3--64.

\bibitem{profinite}
\bysame, \emph{Profinite semigroups and applications}, Structural theory of
  automata, semigroups, and universal algebra, NATO Sci. Ser. II Math. Phys.
  Chem., vol. 207, Springer, Dordrecht, 2005, pp.~1--45.

\bibitem{JA}
J.~Almeida, J.~C. Costa, and M.~Zeitoun, \emph{Complete reducibility of systems
  of equations with respect to {$\sf R$}}, Port. Math. (N.S.) \textbf{64}
  (2007), no.~4, 445--508.

\bibitem{MR2142087}
J.~Almeida and M.~Delgado, \emph{Tameness of the pseudovariety of abelian
  groups}, Internat. J. Algebra Comput. \textbf{15} (2005), no.~2, 327--338.

\bibitem{MR1819091}
J.~Almeida and P.~V. Silva, \emph{S{C}-hyperdecidability of {$\bf R$}},
  Theoret. Comput. Sci. \textbf{255} (2001), no.~1-2, 569--591.

\bibitem{steinbergJA}
J.~Almeida and B.~Steinberg, \emph{On the decidability of iterated semidirect
  products with applications to complexity}, Proc. London Math. Soc. (3)
  \textbf{80} (2000), no.~1, 50--74.

\bibitem{MR1834943}
J.~Almeida and P.~G. Trotter, \emph{The pseudoidentity problem and reducibility
  for completely regular semigroups}, Bull. Austral. Math. Soc. \textbf{63}
  (2001), no.~3, 407--433.

\bibitem{drh}
J.~Almeida and P.~Weil, \emph{Free profinite {$\mathcal R$}-trivial monoids},
  Internat. J. Algebra Comput. \textbf{7} (1997), no.~5, 625--671.

\bibitem{palavra}
J.~Almeida and M.~Zeitoun, \emph{An automata-theoretic approach to the word
  problem for {$\omega$}-terms over~{$\sf{R}$}}, Theoret. Comput. Sci.
  \textbf{370} (2007), no.~1-3, 131--169.

\bibitem{MR1232670}
C.~J. Ash, \emph{Inevitable sequences and a proof of the ``type {${\rm II}$}
  conjecture''}, Monash {C}onference on {S}emigroup {T}heory ({M}elbourne,
  1990), World Sci. Publ., River Edge, NJ, 1991, pp.~31--42.

\bibitem{MR0179239}
G.~Baumslag, \emph{Residual nilpotence and relations in free groups}, J.
  Algebra \textbf{2} (1965), 271--282.

\bibitem{phd}
C.~Borlido, \emph{The word problem and some reducibility properties for
  pseudovarieties of the form $\sf{DRH}$}, Ph.D. thesis, University of Porto,
  2015.

\bibitem{MR753707}
J.~A. Brzozowski and F.~E. Fich, \emph{On generalized locally testable
  languages}, Discrete Math. \textbf{50} (1984), no.~2-3, 153--169.

\bibitem{MR1485465}
T.~Coulbois and A.~Khelif, \emph{Equations in free groups are not finitely
  approximable}, Proc. Amer. Math. Soc. \textbf{127} (1999), no.~4, 963--965.

\bibitem{MR2364777}
M.~Delgado, A.~Masuda, and B.~Steinberg, \emph{Solving systems of equations
  modulo pseudovarieties of abelian groups and hyperdecidability}, Semigroups
  and formal languages, World Sci. Publ., Hackensack, NJ, 2007, pp.~57--65.

\bibitem{MR0530382}
S.~Eilenberg, \emph{Automata, languages, and machines. {V}ol. {A}}, Academic
  Press [A subsidiary of Harcourt Brace Jovanovich, Publishers], New York,
  1974, Pure and Applied Mathematics, Vol. 58.

\bibitem{MR0530383}
\bysame, \emph{Automata, languages, and machines. {V}ol. {B}}, Academic Press,
  New York-London, 1976.

\bibitem{Makanin}
G.~S. Makanin, \emph{The problem of the solvability of equations in a free
  semigroup}, Mat. Sb. (N.S.) \textbf{103(145)} (1977), no.~2, 147--236, 319.

\bibitem{MR1723477}
J.~Rhodes, \emph{Undecidability, automata, and pseudovarieties of finite
  semigroups}, Internat. J. Algebra Comput. \textbf{9} (1999), no.~3-4,
  455--473.

\bibitem{topologia}
S.~Willard, \emph{General topology}, Dover Publications, Inc., Mineola, NY,
  2004, Reprint of the 1970 original [Addison-Wesley, Reading, MA].

\end{thebibliography}
\end{document}